\documentclass[a4paper,10pt]{article}
\usepackage[active]{srcltx}
\usepackage[all]{xy}
\usepackage{subfig}
\usepackage{hyperref}
\usepackage{mathrsfs}
\usepackage[english]{babel}

 \usepackage{amsmath, amsthm}
 \usepackage{stmaryrd}

\usepackage{enumerate}
\usepackage{esint}
\usepackage{amsmath}
\usepackage{amssymb,latexsym}
\usepackage{mathrsfs}
\usepackage{graphics}
\usepackage{latexsym}
\usepackage{psfrag}
\usepackage{import}
\usepackage{verbatim}
\usepackage{graphicx}
\usepackage[usenames]{color}
\definecolor{darkgreen}{rgb}{0.0, 0.2, 0.13}
\usepackage{pifont,marvosym}

\definecolor{antiquewhite}{rgb}{0.98, 0.92, 0.84}

\newcommand{\rest}[1]{\big\rvert_{#1}}
\theoremstyle{plain}
\newtheorem{lemma}{Lemma}[section]
\newtheorem{theorem}[lemma]{Theorem}
\newtheorem{proposition}[lemma]{Proposition}
\newtheorem{corollary}[lemma]{Corollary}

\theoremstyle{definition}

\newtheorem{definition}[lemma]{Definition}

\newtheorem{remark}[lemma]{Remark}

\newtheorem{example}{Example}
\newtheorem{notation}[lemma]{Notations}

\numberwithin{equation}{section}

\newcommand{\R}{\mathbb{R}}
\newcommand{\N}{\mathbb{N}}

\newcommand{\supp}{\text{\rm supp}}

\newcommand{\C}{\mathbb{C}}

\newcommand{\ve}{\varepsilon}

\newcommand{\weak}{\rightharpoonup}

\newcommand{\Geo}{{\rm Geo}}

\newcommand{\sfd}{\mathsf d}

\newcommand{\tr}{\mathrm{tr}}
\newcommand{\ee}{{\rm e}}

\newcommand{\st}{\mathsf{S}\mathsf{t}(r,\mathsf{H})}

\title{Geometry of Grassmannians \\and optimal transport of quantum states}
\author{Paolo Antonini\thanks{Dipartimento di Matematica e Fisica ``E. de Giorgi'', Universit\`a del Salento, Lecce (Italy)
email: paolo.antonini@unisalento.it}\,\,
 and Fabio Cavalletti\thanks{Mathematics Area, SISSA, Trieste (Italy), email: cavallet@sissa.it.}}

\begin{document}
\date{}

\maketitle

\abstract{\noindent Let $\mathsf{H}$ be a separable Hilbert space. We prove that the Grassmannian $\mathsf{P}_c(\mathsf{H})$ of the finite dimensional subspaces of $\mathsf{H}$ is an Alexandrov space of nonnegative curvature and we
employ its metric geometry to develop the  
theory of optimal transport for the normal states 
of the von Neumann algebra of linear and bounded operators $\mathsf{B}(\mathsf{H})$. Seeing density matrices as discrete probability measures on $\mathsf{P}_c(\mathsf{H})$ (via the spectral theorem) we define an optimal transport cost and the Wasserstein distance for normal states. In particular we obtain a cost which induces the $w^*$-topology. 
 
Our construction is compatible with the quantum mechanics 
approach of composite systems as tensor products $\mathsf{H}\otimes \mathsf{H}$.
We provide indeed an interpretation of the pure normal states 
of $\mathsf{B}(\mathsf{H}\otimes \mathsf{H})$ as families of transport maps. This also defines a Wasserstein cost for the pure normal states of $\mathsf{B}(\mathsf{H}\otimes \mathsf{H})$, reconciling with our proposal.
}
\tableofcontents

\section{Introduction}
In Connes' noncommutative geometry program \cite{Connes94}
many naturally singular spaces 
of great interest in geometry or quantum physics can be fruitfully addressed using noncommutative operator algebras.
There is nowadays a huge literature about these noncommutative spaces 
for which we refer to the aforementioned book \cite{Connes94}. 
We limit ourselves here to list some of the most known and interesting examples. Among these we find: leaf spaces of foliations, the space of unitary representations of a discrete group and the phase space of quantum mechanics. This last one is related with the current paper where we address the problem of optimal transport for quantum states.

Due to the pervasiveness of noncommutative spaces,
extensions of the classical tools as measure theory, topology, differential calculus and Riemannian geometry, have been pursued in the noncommutative setting and during the last few years, as it is naturally expected, also the search of an appropriate analogue of a Wasserstein distance received a great deal of attention. Some important progresses have been obtained.  

In the noncommutative setting states take over the role of probability measures; 
for example, in the case of the algebra of matrices (as well in $\mathsf{B}(\mathsf{H})$ if we consider only normal states) using the matrix trace,
  states can be identified with positive definite matrices with unit trace which are indeed called density matrices.
  
In the context of spectral triples considered as noncommutative manifolds, where the noncommutative algebra $\mathcal{A}$ interacts with a Dirac operator,
  Connes \cite{Connes1} defined a $1$-Wasserstein distance
  on the space of states of $\mathcal{A}$. This is thought as the dual distance in the spirit of Monge--Kantorovich, defined in terms of Lipschitz functions (or their noncommutative analog).
 Connes' distance and the Kantorovitch duality have been the subject of many works by Rieffel,  D'Andrea, Martinetti and collaborators. We refer in a non exhaustive way  to the papers  \cite{Rieffel,Dandrea} and the references therein.
 
In the realm of free probability, Biane and Voiculescu defined an analog of the Wasserstein distance on the space of the trace-states of a $C^*$-algebra \cite{Voiculescu}. Their metric extends the classical Wasserstein metric.

A proposal for the finite dimensional case, 
which follows the principle to adapt the dynamical formulation of optimal transport
 \`a la Benamou and Brenier \cite{BB00}
 has been given 
 by Carlen and Maas \cite{CarlMaas, CM17,CM19}.  
 Here one assigns a length to each path of probability measures connecting
 the marginals.
 
  A key property of the resulting quantum distance in loc. cit. is the fact that it is induced by a Riemannian metric on the manifold of quantum states and the quantum generalisation of the heat semigroup is the gradient flow of the von Neumann entropy 
${\rm Ent}(\rho) = {\rm Trace }(\rho \log \rho)$.
This replaces the classical relative entropy of the commutative case. 
Also the relation of this approach to the rate of convergence of the quantum Ornstein-Uhlenbeck semigroup \cite{CM19} have been established. 

Subsequent developments worth mentioning include: the one of Wirth \cite{Wir18}, based on the noncommutative Dirichlet forms of Cipriani and Sauvageot \cite{CS03}  and the work of Hornshow \cite{Horns} where also the approximately finite dimensional case is considered  estabilishing lower bounds on Ricci curvature. We refer to these papers for more details.  
Finally
another proposal by Golse, Mouhot and Paul \cite{GMP15} arose in the context of the study of the semiclassical limit of quantum mechanics and it relies 
on the concept of \emph{couplings} with applications to the study
of the mean-field limit of quantum mechanics. 
\smallskip

Our contribution goes in a new direction to study a static formulation 
of the optimal transport problem between quantum states. We base our constructions on the geometric structure of the Grassmann manifold of all the finite rank projections of the underlying separable Hilbert space $\mathsf{H}$.

Let us describe more precisely the setting.

Let $\mathcal{S}_{n}(\mathsf{B(H)})$ denotes the convex set of normal states of 
$\mathsf{B(H)}$, the von Neumann algebra of linear bounded operators on $\mathsf{H}$.
For more details we refer to Section \ref{Ss:Normal}.
Any such state $\varphi$ is identified with its density matrix $\rho_{\varphi}$ satisfying $$
\rho_{\varphi}^{*} = \rho_{\varphi},\quad  \rho_{\varphi}\geq 0 \quad  \textrm{and}\quad  \tr(\rho_{\varphi}) = 1.$$ We introduce a distance 
between density matrices relying on the optimal transport problem between  
probability measures over the Grassmanian of $\mathsf{H}$.  latter is denoted by 
$\mathsf{P}$ and is defined as the collection of all orthogonal projections of 
$\mathsf{H}$. Its connected components are labelled by the dimension of the ranges of the projections.

The map between density matrices and non-negative measures over $\mathsf{P}$
is induced by the Spectral Theorem: by compactness and self-adjointness 
the following correspondence is rather natural:
$$
\rho_{\varphi} = \sum_{i} \lambda_{i} P_{V_{i}} 
\quad 
\Longrightarrow  
\quad 
\mu_{\varphi} : = \sum_{i} \lambda_{i} \delta_{P_{V_{i}}}. 
$$
Here $P_{V_{i}}$ stands for the orthogonal projection with range the finite 
dimensional eigenspace $V_{i}$ with eigenvalue $\lambda_{i} > 0$.
The spectral decomposition is understood without repetitions.
Notice the projection onto the kernel does not belong to the support of the associate measure $\mu_{\varphi}$.
 Since $\tr(\rho_{\varphi}) = 1$, it follows that 
$\tr(\cdot) \mu_{\varphi}$ is a probability measure over the Polish space 
$\mathsf{P}_{c}$, the submanifold of $\mathsf{P}$ of finite rank orthogonal projections. 

The Polish structure of $\mathsf{P}_{c}$, making it amenable to standard measure theory techniques, is the one inherited as a Finsler submanifold 
of $\mathsf{B(H)}$. However $\mathsf{P}_{c}$ admits a more convenient 
geometric structure induced infinitesimally by viewing $\mathsf{P}_{c}$ as a 
submanifold of the space of the Hilbert-Schmidt operators. It follows that  
each connected component (where the trace is constant) of  $\mathsf{P}_{c}$ 
is an Alexandrov space of non-negative curvature.
A fact giving a very natural setting to explore geometric links between normal states and optimal transport. This will be 
thoroughly studied in Section \ref{Ss:project}.

Denoting by $\sfd$ the geodesic distance of $\mathsf{P}_{c}$, 
the Wasserstein distance between the normal states $\varphi,\psi$
can be then defined as the Wasserstein distance of the spectral measures 
$\mu_{\varphi}$ and $\mu_{\psi}$ after being normalized to be probability measures:
\begin{equation}\label{E:Wasserstein}
W_{p}(\varphi,\psi) : = 
W_{p}^{\mathsf{P}_{c}}\big{(}\tr(\cdot)\mu_{\varphi},\tr(\cdot)\mu_{\psi}\big{)}.
\end{equation}
In this formula $W_{p}^{\mathsf{P}_{c}}$ denotes the classical Wasserstein distance 
defined over the Polish space $(\mathsf{P}_{c},\sfd)$.
Because of the presence of different connected components,  
$W_{p}(\cdot,\cdot)$ might easily become infinite making 
$W_{p}$ an extended distance.  

We overcome this issue by considering a larger family of discrete measures 
representing density matrices. 
In particular for each normal state $\varphi$ we consider
the set $\Lambda_{\varphi}^{\perp}$ of discrete measures 
$\mu = \sum \lambda_{i} \delta_{P_{i}}$ with $\lambda_{i} \geq 0$
such that $\rho_{\varphi} = \sum_{i}\lambda_{i} P_{i}$  and $P_{i}\perp P_{j}$ 
whenever $i \neq j$.
In contrast with the representations considered before the eigenvalues now admit repetitions.
Then the natural extension of $W_{p}(\varphi,\psi)$ is obtained by defining 
the \emph{cost} between $\varphi$ and $\psi$ as 
\begin{equation}\label{E:Cost}
\mathcal{C}_{p}(\varphi,\psi) : = 
\inf_{\left.\begin{array}{c}\mu_{0} \in \Lambda_{\varphi}^{\perp} \\\mu_{1}\in \Lambda_{\psi}^{\perp}\end{array}\right.
} W_{p}^{\mathsf{P}_{c}}\big{(}\tr(\cdot)\,\mu_{0}, 
\tr(\cdot)\,\mu_{1}\big{)},
\end{equation}
as the Wasserstein distance between the two (compact) sets of associated measures representing the states.
The main properties we obtain for $\mathcal{C}_{p}$ are the following ones:
\begin{description}
\item[Existence of optimal configurations:]	for any couple of normal states $\varphi$ and $\psi$, the infimum in \eqref{E:Cost} can be replaced by the minimum (Proposition \ref{P:compactness}). Moreover optimal couplings always exist (Proposition \ref{P:attained}).
\item[Projections of dimension 1:] 
The optimal configurations $\mu_{0},\mu_{1}$ can always be taken with support contained inside the connected component $\mathsf{P}_{1}$, i.e. the space of projections with one dimensional rank (Proposition \ref{P:dimension1}). This is $\mathbb{P}(\mathsf{H})$ the projective space of $\mathsf{H}$, the space of the pure states of the $C^*$-algebra of the compact operators $\mathbb{K}$ .
\item[Topology:] $\mathcal{C}_{p}$ is a semi-distance inducing the
weak topology over $\mathcal{S}_{n}(\mathsf{B(H)})$ (Theorem \ref{T:equivalentconvergence}).
\end{description}

We also obtain the duality formula for $W_{p}$ with the Kantorovich potentials 
represented by densely defined operators (Theorem \ref{T:Dualmain} and Corollary \ref{C:duality}). 
Relying on the geodesic structure of $(\mathsf{P}_{c},\sfd)$, 
we also study $W_{p}$-geodesics of $\mathcal{S}_{n}(\mathsf{B(H)})$ in 
Section \ref{S:geodesicW}. 

\smallskip \noindent 
In the last section we study tensor product Hilbert spaces
 $\mathsf{H} \otimes \mathsf{H}$
corresponding in quantum mechanics to composite systems.
A natural 
way to match two normal states $\varphi,\psi$ of $\mathsf{B}(\mathsf{H})$ 
would be via a normal 
state $\Xi \in \mathcal{S}_{n}(\mathsf{B}(\mathsf{H}\otimes \mathsf{H}))$ 
satisfying the partial trace conditions
$J^{1}_{\flat}\Xi = \varphi$ and $J^{2}_{\flat}\Xi = \psi$ 
(for the notation see Section \ref{Ss:marginals}).
In Section \ref{S:tensor-Wasserstein} we reconcile this point of view
with the one presented in Section \ref{S:Wasserstein}.

In particular we prove the following (Theorem \ref{T:injection}).
\begin{description}
\item[Pure normal states of the tensor product as natural families of transport plans:]	given any element $\omega_{\zeta}$ of 
$\mathcal{PS}_{n}(\mathsf{B}(\mathsf{H}\otimes \mathsf{H}))$, i.e. any pure normal state of $\mathsf{B}(\mathsf{H}\otimes \mathsf{H})$ 
with partial traces $\varphi$ and $\psi$,
we associate a family of 
admissible transport plans between admissible representations of 
$\varphi$ and $\psi$. In particular this permits to assign a well-defined 
optimal transport cost to any pure normal state of $\mathsf{B}(\mathsf{H}\otimes \mathsf{H})$ (Remark \ref{R:costpure}).
\end{description}

\smallskip
We conclude by mentioning that we tried to keep the paper 
as self-contained as possible. 
In particular in Section \ref{S:preliminaries} 
we have collected, and in some cases re-proved,  
many of the known geometric properties 
of the Grassmanian $\mathsf{P}_{c}$ that are used in this paper and 
that were distributed through different references.

\subsection*{Acknowledgements}
The authors wish to thank Antonio Lerario for a number of interesting discussions.

\subsection{Notations}
In this paper we will switch freely from the standard notation for vectors in a Hilbert space to the Dirac notation with Bra and Kets. In particular we consider the inner product $\langle \cdot ,\cdot \rangle$ or  $\langle \cdot |\cdot \rangle$ with the {\em{Physicists convention}}: antilinear in the first entry.
For a linear operator $T$ on a vector space we denote $N(T)$ for its Kernel and $R(T)$ for its image. 

Projection means orthogonal projection ie. $P=P^*$ and $P^2=P$ when $P^2=P$ we say idempotent.
Also for two projections we write $Q \leq P$ if and only if 
$Q\mathsf{H}\subset P\mathsf{H}$. This is equivalent to $PQ=Q$ or  $QP=Q$.


\section{Preliminaries}\label{S:preliminaries}

\subsection{Geometry of the space of projections}\label{Ss:project}
Let us fix $\mathsf{H}$ an Hilbert space; $\mathsf{B}(\mathsf{H})$ will be the  space 
of bounded linear operators in $\mathsf{H}$ and 
$\mathsf{B}_{h}(\mathsf{H})$ the subspace of the self-adjoint ones (Hermitian); also denote by $\mathsf{B}_{sa}(\mathsf{H})=\{X\in \mathsf{B}(\mathsf{H}):\, X=-X^*\}$ the skew adjoints.

The Grassmannian of $\mathsf{H}$, denoted with $\mathsf{P}$ is the space of all the 
 projections:
$$ \mathsf{P}=\big{\{}P\in \mathsf{B}(\mathsf{H}): P=P^*\textrm{ and } P^2=P\big{\}}.$$
We describe its geometry mainly following \cite{Androu0,Andru1,Andru2,CPR,Porta1}.
Fundamental is the natural action of the unitary group $\mathsf{U}(\mathsf{H})$ by conjugation $g\cdot P=gPg^*$ for $g\in \mathsf{U}(\mathsf{H}).$ \\

We recall  here few but important facts about the group
 ${\mathsf{U}(\mathsf{H})}$. This is a Banach--Lie group, closed inside $\mathsf{B}(\mathsf{H})$ with Lie algebra identified with the skew adjoint operators $\mathfrak{u}:=T_1{\mathsf{U}(\mathsf{H})}=\mathsf{B}_{sa}(\mathsf{H})$ having the operators commutator as Lie bracket. The exponential map
 $\exp: \mathfrak{u} \longrightarrow {\mathsf{U}(\mathsf{H})}$
 is the operators exponentiation. It is surjective because in
 $\mathsf{B}(\mathsf{H})$ we may form Borel functions of normal operators; this gives a logarithm for every skew-adjoint operator.

All the curves in the form
$$
[-1,1] \ni t\longmapsto ue^{itX} \in {\mathsf{U}(\mathsf{H})}
$$ 
with $X\in \mathsf{B}_h(\mathsf{H})$
i.e. the translations of one parameter groups are called the {\em{group geodesics}} of ${\mathsf{U}(\mathsf{H})}$. The name is legitimated by the fact that we can find a natural class of linear connections on 
${\mathsf{U}(\mathsf{H})}$ creating such geodesics.
Moreover one can show that 
are minimal curves inside ${\mathsf{U}(\mathsf{H})}$ with respect to the natural Finsler structure inherited by the embedding $\mathsf{U}(\mathsf{H}) \subset \mathsf{B}(\mathsf{H})$ (see \cite{Atkin}).
We are now ready to discuss the geometry of $\mathsf{P}$.

\begin{itemize}
\item[1.] {\bf Manifold structure.} $\mathsf{P}$ is a submanifold of $\mathsf{B}_h(\mathsf{H})$ with complemented tangent. Its tangent space at $P$, as a submanifold is naturally identified  in the following way:
\begin{equation}\label{tangentidentification}
T_{P}\mathsf{P} = \big{\{} Y \in \mathsf{B}_h(\mathsf{H})\colon PY+YP=Y\big{\}};
\end{equation}
or equivalently with all the selfadjoint operators $Y$ satisfying $P Y P = (1-P)Y (1-P) = 0$. 
Indeed $P$ induces a block decomposition for the whole $\mathsf{B}_h(\mathsf{H})$ 
\begin{equation}\label{decomp}
A \longmapsto \bigg{(}\begin{array} {cc}PAP & PA(1-P) \\
 (1-P)AP & (1-P)A(1-P) \end{array} \bigg{)},
\end{equation}
so that can give the following.
\begin{definition}
The selfadjoint operators which are off-diagonal in the decomposition \eqref{decomp} are called \emph{co-diagonal} with respect to $P$. The space of all the co-diagonal operators with respect to $P$ is denoted by $\mathscr{C}_P$.	
\end{definition}
In symbols
\begin{equation}\label{codiagonalsidentification}
 \mathscr{C}_P =\big{\{} Y \in \mathsf{B}_h(\mathsf{H})\colon PY+YP=Y\big{\}}.
\end{equation}
Let us prove the \eqref{tangentidentification}. The first inclusion comes differentiating the relation $\gamma^2(t)=\gamma(t)$ for a smooth curve in $\mathsf{P}$ with $\gamma(0)=P$. For the reversed inclusion we make use of  \eqref{codiagonalsidentification} and we observe first
that any $X\in \mathscr{C}_P$ satisfies 
 $X=[[X,P],P]$. This also means (every commutator with $P$ is codiagonal) that $\mathscr{C}_P=\{i[X,P]: X\in \mathsf{B}_s(\mathsf{H}) \}
$. 
 Now if $X$ is codiagonal,  
$e^{t[X,P]}$ is a one parameter group of unitaries ($[X,P]$ is skew-adjoint) and the path $\gamma(t)= e^{t[X,P]}Pe^{-t[X,P]}$ satisfies $\dot{\gamma}(0)=[[X,P],P]=X.$ \\
We will see later that curves in the form of $\gamma$ are exactly the geodesics through $P$ with respect to a family of natural connections.
Summing up:
$$T_P\mathsf{P}=\mathscr{C}_P=\big{\{} Y \in \mathsf{B}_h(\mathsf{H})\colon PY+YP=Y\big{\}}=\big{\{}i[X,P]: X\in \mathsf{B}_h(\mathsf{H}) \big{\}}.$$

If we denote by $\mathscr{D}_P$ the selfadjoint operators which are diagonal in the decomposition \eqref{decomp} we have a linear splitting \begin{equation}\label{cosplitting}
\mathsf{B}_h(\mathsf{H})=\mathscr{C}_P \oplus \mathscr{D}_P.\end{equation}

\item[2.] {\bf Homogeneous space structure of the connected components.}\label{I:2} 

The ${\mathsf{U}(\mathsf{H})}$ action on $\mathsf{P}$ is locally transitive for if $\|P-Q\|<1$ then $Q=g\cdot P$ for some unitary $g$. 
Using this fact one shows that $P$ and $Q$ are in the same connected component if and only if there exists a path of unitaries $g_t$ with $g_0=1$ and $P=g_1 Q g_1^*$ (a proof in \cite[Corollary 5.2.9]{Wegge}).
 In other words 
the $\mathsf{U}(\mathsf{H})$-orbits, i.e. the conjugacy classes are the connected components in 
$\mathsf{P}:$  
$$
\mathcal{O}(P):=\big{\{}gPg^*: g \in {\mathsf{U}(\mathsf{H}})\big{\}}=\textrm{connected component of }P.
$$
These connected components are easily found;
let $R(Q)$ denote the range of the operator $Q$ and $N(Q)$ its kernel. Then
 $P$ and $Q$ are connected iff $\operatorname{dim} N(P)=\operatorname{dim} N(Q)$ and $\operatorname{dim} R(P)=\operatorname{dim} R(Q)$.

Let's now fix a reference point $P\in \mathsf{P}$ (for the rest of this section).
The stabiliser $\mathsf{I}_P=\{g: g\cdot p=p\}$ coincides with the subgroup $\{g \in {\mathsf{U}(\mathsf{H})}: [g,P]=0\}$ and the quotient $\mathsf{U}(\mathsf{H})/\mathsf{I}_P$ is diffeomorphic to $\mathcal{O}_P$. More precisely, using the canonical projection 
\begin{equation}\label{pbundlecomponent}
	\mathsf{U}(\mathsf{H}) \longrightarrow \mathsf{U}(\mathsf{H})/\mathsf{I}_P \cong \mathcal{O}_P,
\end{equation}
we get a principal bundle with equivariant projection. In other words $\mathcal{O}_P$ is an homogeneous space \cite[Proposition 2.2]{Androu0}.

The decomposition diagonal/codiagonal \eqref{cosplitting} defines on the principal bundle \eqref{pbundlecomponent} a canonical connection (indeed the homogeneous space structure is {\em{reductive}}). The canonical connection induces in the customary way a notion of parallel translation, covariant derivative and  geodesics for $\mathsf{P}$. We don't construct them explicitly  here because we will consider in a while, a second, more direct connection on $T\mathsf{P}$ sharing the same geodesics.

\smallskip
\item[3.] {\bf Connection on $T\mathsf{P}$}. To any $X \in \mathsf{B}_{h}(\mathsf{H})$ we can associate its co-diagonal part with respect to $P$ using the projection onto the codiagonals
\begin{equation}\label{E:projection}
E_{P} : \mathsf{B}_{h}(\mathsf{H}) \longrightarrow T_{P}\mathsf{P}, \quad E_{P}(X) : = PX(1-P) + (1-P)X P.
\end{equation}
This induces a connection (in the usual sense) on  $T\mathsf{P}$.
If $X$ is a tangent field
(i.e. $X \colon \mathsf{P} \longrightarrow \mathsf{B}_h(\mathsf{H})$ with $X(P)\in T_P\mathsf{P}$ for every $P$)
 and  $\gamma:I \longrightarrow \mathsf{P}$ a curve, then $X\circ \gamma$ is a vector field along $\gamma$ with covariant derivative
\begin{equation}\label{connect}
\frac{DX}{dt} = E_{\gamma(t)}\bigg{(} \dfrac{d}{dt}X(\gamma(t))\bigg{)}.
\end{equation}

\item[4.] {\bf Geodesics.} A curve $\gamma : I \to \mathsf{P}$ is a geodesic if, by definition
$$
\frac{D \dot \gamma}{dt} = 0, \qquad \forall \ t \in I.
$$
All the geodesics starting at $P \in \mathsf{P}$ are in the form 
$\gamma(t) = e^{itZ} P e^{-itZ}$ with $Z \in T_{P}(\mathsf{P})$ \cite{Andru1,CPR}. As anticipated we can prove that these are also all the geodesics with respect to the connection induced by the natural connection in $\mathsf{P}$ as an homogeneous reductive  space.

To check that the geodesic equation is satisfied for $\gamma(t)=e^{itZ} P e^{-itZ}$ we take the opportunity to discuss the manifold of symmetries 
$\mathsf{S}:=\big{\{}S \in  \mathsf{B}_{h}(\mathsf{H}): \, S^2=1\big{\}}$, diffeomorphic to $\mathsf{P}$ via the map
\begin{equation}\label{diffeomorphism}
\mathcal{F}:\mathsf{P}  \longrightarrow \mathsf{S}, \quad P \longmapsto 2P-1.
\end{equation}
The tangent space at $S \in \mathsf{S}$ consists in all the self-adjoint $X \in  \mathsf{B}_{h}(\mathsf{H})$ anticommuting with $S$ i.e.
$$
T_{S}\mathsf{S} = \Big{\{}X \in \mathsf{B}_{h}(\mathsf{H}) \colon SX + XS =0\Big{\}}.
$$ 
We have a corresponding projection on the tangent space which has the form 
$$
\operatorname{Pr}_S:\mathsf{B}_h(\mathsf{H})\longrightarrow \mathsf{B}_h(\mathsf{H}), \quad \operatorname{Pr}_S(Z)=(1-P)ZP+PZ(1-P); \quad 2P-1=S,
$$
also inducing a connection on $\mathsf{S}$. 
This is given by the same formula as \eqref{E:projection}. On the other hand  the map $\mathcal{F}:\mathsf{P}\longrightarrow \mathsf{S}$ is compatible with the two connections  on the domain and target
thus sending a geodesic to a geodesic.
In fact $\mathcal{F}$ is the restriction of a map defined on the whole of $\mathsf{B}_h(\mathsf{H})$ and its differential $d_P\mathcal{F}(X)=2X$ intertwines the two projections onto $\mathsf{P}$ and $\mathsf{S}$.

Now thanks to the inclusion $\mathsf{S}\subset \mathsf{U}(\mathsf{H})$ some formulas simplify when passing to $\mathsf{S}$. 
Start with the curve $\gamma(t)=e^{itZ}Pe^{-itZ}$ in $\mathsf{P}$ with $Z\in T_P\mathsf{P}$. Since $Z$ is $P$-codiagonal, it anticommutes with $S=\mathcal{F}(P)$ so that $e^{itZ}\mathcal{F}(P)=\mathcal{F}(P)e^{-itZ}$. We can now transform $\gamma$ under $\mathcal{F}$:
$$
\mathcal{F}(e^{itZ}Pe^{-itZ})=e^{itZ}\mathcal{F}(P)e^{-itZ}=\mathcal{F}(P)e^{-2itZ}=\mathcal{F}(P)e^{-itd\mathcal{F}(Z)}.
$$
It is immediate to check that this is a geodesic in $\mathsf{S}$ and by the properties of $\mathcal{F}$ we see that $\gamma$ is a geodesic too.
Moreover $\mathcal{F}(\gamma)$ is also a geodesic in $\mathsf{U}(\mathsf{H})$ (a traslation of a one parameter group). In other words $\mathsf{S}$ is totally geodesic inside $\mathsf{U}(\mathsf{H}))$.

Put $Y:=-iSZ/2 \in T_S\mathsf{S}$ then the geodesic in $\mathsf{S}$ can also be written as $t \mapsto e^{tXS/2}Se^{-tXS/2}$. Indeed the exponential map is the restriction of the family of analytic mappings
$$
\mathsf{B}(\mathsf{H}) \longrightarrow \mathsf{B}(\mathsf{H}),\quad 
Z\longmapsto e^{ZS/2}Se^{-ZS/2}.
$$
The exponential map for $\mathsf{P}$ follows using $\mathcal{F}$.
We note also the formula $\frac{d}{dt} \, e^{tXS/2}Se^{-tXS/2}=e^{tXS/2}Xe^{-tXS/2}$.
\end{itemize}

\subsubsection{Metric aspects}  
The Grassmannian $\mathsf{P}$ has a natural non-smooth 
reversible Finsler structure induced by the operator norm via the embedding $\mathsf{P}\subset \mathsf{B}_h(\mathsf{H})$. 
However the submanifold
$$\mathsf{P}_c:=\mathsf{P} \cap \mathbb{K},$$ of the compact and then finite rank projections
is contained in the Hilbert space $\mathsf{H}\mathsf{S}(\mathsf{H})$ of the (selfadjoint) Hilbert--Schmidt operators
with metric $(A,B) \mapsto  \Re \operatorname{tr}(A^*B)$
 and inherits a riemannian structure.
Any point $P \in \mathsf{P}_c$ is finite rank so that the co-diagonal operators at $P$  are finite rank too and we have the induced metric\footnote{since the operators are codiagonal the trace of $XY$ is real valued}
$$g( X,Y) :=\operatorname{tr}(XY), \quad X,Y \in T_P\mathsf{P}_c,$$
generalising the familiar riemannian (K\"ahler) structure on the finite dimensional Grassmann manifold.
We summarise some of the basic properties (see \cite{Andrupositive,isidro}) :
\begin{itemize}
\item the topology on $\mathsf{P}_c$ induced by the embedding $\mathsf{P}_c \subset B_h(\mathsf{H})$ where $B_h(\mathsf{H})$ is given with the norm topology coincides the topology induced by the embedding  $\mathsf{P}_c \subset \mathsf{H}\mathsf{S}(\mathsf{H})$. This is clear for if $T$ and $S$ are finite rank operators with range of dimension at most $n$ then:
$$\|T-S\| \leq \|T-S\|_2 \leq \sqrt{2n}\, \|T-S\|$$
with $\|\cdot \|_2$ the Hilbert--Schmidt norm.
\item The connection \eqref{connect} is exactly the Levi--Civita connection. 
We can compute an explicit formula following \cite{Dimitric}. 
We have orthogonal projections on the tangent space and on the normal space to $\mathsf{P}_c$ and the theory of submanifolds presents no differences with the finite dimensional case.
In fact the orthogonal projection is exactly 
the projection on the codiagonals that we have already used.\\
Now let $P\in \mathsf{P}_c$ and $X,Y$ vector fields tangent to $\mathsf{P}_c$; if we denote with $D_XY$ the covariant derivative in the flat space $\mathsf{H}\mathsf{S}(\mathsf{H})$, we have at $P$:
$$D_XY=\big{(}   PD_XY(1-P)+(1-P)D_XYP \big{)}+ (XY+YX)(1-2P).$$
The first addendum is tangential to $\mathsf{P}_c$ while the second one is normal. Therefore 
$$\nabla_XY= PD_XY(1-P)+(1-P)D_XYP , \quad \textrm{the connection of }\mathsf{P}_c \textrm{ at }P, $$
$$\sigma(X,Y)=(XY+YX)(1-2P) \quad \textrm{ the second fundamental form at }P.$$

\item The geodesics that we have already discussed are geodesics for the metric in $\mathsf{P}_c$ too. In particular $t \longmapsto e^{t[X,P]}Pe^{-t[X,P]}$ is the unique geodesic starting from $P$ with initial velocity $X$.
\item The curvature tensor is $$R(X,Y)Z=\big{[}[X,Y],Z\big{]}, \quad X,Y,Z \in T_P\mathsf{P}_c$$
as follows immediately from the Gauss formula (the ambient space is flat)
$$\langle R(X,Y)Z,W \rangle = \langle \sigma(X,W),\sigma(Y,Z) \rangle - \langle \sigma(X,Z),\sigma(Y,Z) \rangle.$$
 From the Cauchy--Schwartz inequality it follows the sectional curvature is non negative.
\item The length of a smooth or Lipschitz, curve $\gamma:I \longrightarrow \mathsf{P}_c$ is defined by 
$L(\gamma)=\int_{I}\|\dot{\gamma}\|dt.$ The geodesic distance 
$\sfd$ follows by minimisation over all the paths. 
If $P$ and $Q$ satisfy $\sfd(P,Q)<\pi/2$ are joined by a unique geodesic with length $L(\gamma)=\sfd(P,Q)$.
The metric space $(\mathsf{P}_c,\sfd)$ is complete. It follows that ($\mathsf{H}$ separable) is Polish.
\end{itemize}

To describe in more details the geometry of $\mathsf{P}_c$ is useful to follow the techniques in \cite{Mennucci} presented in the real case. The extension to our, complex case is straightforward as we will show in the following. \\

To start with, we present $\mathsf{P}_c$ as the base of a second principal bundle with fiber $\mathsf{U}_r$. This is in contrast with the previous discussion.
Firstly we introduce a notation for the connected components of $\mathsf{P}_c$
\begin{equation}\label{E:connectedPc}
\mathsf{P}_{r}:=\big{\{} P\in \mathsf{P}_c: \, \operatorname{dim}R(P)=r\big{\}}.
\end{equation}
Keeping the rank $r$ fixed, let $\mathsf{S}\mathsf{t}(r,\mathsf{H})$ be the (complex) Stiefel manifold. It is the manifold of all the Hilbert space embeddings $\varphi:\mathbb{C}^r \to \mathsf{H}$. Thus $\varphi^*\varphi= \operatorname{Id}_r$. 
Any $\varphi \in \mathsf{S}\mathsf{t}(r,\mathsf{H})$ is specified by a collection of $r$-orthonormal vectors in $\mathsf{H}$, the columns of the finite dimensional matrix of $\varphi$. We have in this way a natural embedding 
\begin{equation}\label{metricembedding}
\mathsf{S}\mathsf{t}(r,\mathsf{H}) \subset \underbrace{\mathsf{H} \times \cdots \times \mathsf{H}}_{r\textrm{ times}}
\end{equation}
 with
 tangent space 
 $$T_{\varphi}\mathsf{S}\mathsf{t}(r,\mathsf{H})=\big{\{} X \in \mathsf{B} (\C^r,\mathsf{H}) :\, X^*\varphi+ \varphi ^*X=0     \big{\}}.$$ This is the space of the linear maps $X:\C^r \longrightarrow \mathsf{H}$ such that $X^*\varphi$ is skew-adjoint.
Indeed the inclusion $\subset$ is straightforward. To see the second one first solve the o.d.e. $\dfrac{d}{dt}(\gamma^*\gamma)=\dot\gamma^*\gamma +\gamma^*\dot{\gamma}=0$ in the space of the finite rank maps $\mathsf{B}(\C^r,\mathsf{H})$ with initial data satisfying: $\gamma(0)=\varphi \in \mathsf{S}\mathsf{t}(r,\mathsf{H})$, $\dot{\gamma}(0)=X$ with $X^*\varphi + \varphi^*X=0$. It follows $\gamma(t) \in \mathsf{S}\mathsf{t}(r,\mathsf{H})$. \\

 The embedding \eqref{metricembedding} induces a riemannian metric on the Stiefel manifold: 
 $(X,Y)\mapsto \Re \operatorname{tr} (X^*Y)$ for  $X,Y \in T_{\varphi} \mathsf{S}\mathsf{t}(r,\mathsf{H})$ and we shall consider its rescaled version
 $$g(X,Y):=2\Re \operatorname{tr} (X^*Y) \quad X,Y \in T_{\varphi} \mathsf{S}\mathsf{t}(r,\mathsf{H}).$$ 
 We compute the orthogonal projection on the tangent space of $\mathsf{S}\mathsf{t}(r,\mathsf{H})$. In fact 
 the orthogonal decomposition 
 $$\mathsf{H}^r \cong \mathsf{B}(\C^r,\mathsf{H})= 
 T_{\varphi}\mathsf{S}\mathsf{t}(r,\mathsf{H}) \oplus 
 N_{\varphi}\mathsf{S}\mathsf{t}(r,\mathsf{H}) $$ 
 at $\varphi$ is obtained
combining the decomposition 
\begin{equation}
\label{polarizationofphi}
\mathsf{H}=R(\varphi) \oplus R(\varphi)^{\bot}
\end{equation}
 induced by the projection $\varphi\varphi^*$ together with the orthogonal decomposition in $\mathsf{B}(\C^r)$
 by Hermitian and Skew-Hermitian matrices (with projections denoted by $\operatorname{He}$ and $\operatorname{Sk}$). For any vector
 $X\in \mathsf{B}(\C^r,\mathsf{H})$ we write
$$X=\varphi \varphi^*X + (1-\varphi \varphi^*)X=\big{[}\varphi(\operatorname{Sk} \varphi^*X) + (1-\varphi \varphi^*)X \big{]}+ \varphi(\operatorname{He}\varphi^*X).$$
It is easy to check that these are respectively the tangent and normal component with: $X\mapsto \varphi(\operatorname{Sk} \varphi^*X) + (1-\varphi \varphi^*)X$ the tangent projection and $X\mapsto \varphi(\operatorname{He}\varphi^*X)$ the normal one. In particular we see that $N_{\varphi}\mathsf{S}\mathsf{t}(r,\mathsf{H})=\big{\{} \varphi S: \,\, S\in \mathsf{B}(\C^r),\, S=S^* \big{\}}.$
 
 There are two commuting left and right action
 $$\mathsf{U}(\mathsf{H}) \,\,\,\circlearrowright  \,\,\,\mathsf{S}\mathsf{t}(r,\mathsf{H}) \,\,\,  \circlearrowleft \,\,\, \mathsf{U}_r=\mathsf{U}(\C^r)$$ corresponding to post and pre composition    
$$  u \cdot \varphi:= u\circ \varphi \quad \textrm{and} \quad \varphi   \cdot g :=\varphi \circ g , \quad u \in \mathsf{U}(\mathsf{H}), \, g \in \mathsf{U}_r.$$
The $\mathsf{U}(\mathsf{H})$ action is transitive while the $\mathsf{U}_r$ one is free. Two points $\varphi$ and $\psi$ are in the same $\mathsf{U}_r$\,-orbit if and only if they have the same range. It follows the quotient is $\mathsf{P}_{r}$ with bundle projection
\begin{equation}\label{bundleprojection3}
\pi^{\textsf{S}\textsf{t}}:\mathsf{S}\mathsf{t}(r,\mathsf{H}) \longrightarrow \mathsf{S}\mathsf{t}(r,\mathsf{H})/\mathsf{U}_r \cong \mathsf{P}_{r}, \quad \varphi \longmapsto \varphi \mathsf{U}_r \longmapsto \varphi \varphi^*. 
\end{equation}
The vertical space at $\varphi$ is $V_{\varphi}\mathsf{S}\mathsf{t}(r,\mathsf{H})=\big{\{} \varphi X : \,X \in \mathsf{B}(\C^r), \, X^*+X=0   \big{\}}$ and we choose for horizontal space its orthogonal complement 
$$\mathscr{H}_{\varphi}\mathsf{S}\mathsf{t}(r,\mathsf{H})=V_{\varphi}\mathsf{S}\mathsf{t}(r,\mathsf{H})^{\bot}=
\big{\{}   X\in T_{\varphi}\mathsf{S}\mathsf{t}(r,\mathsf{H}): g(X,Y)=0, \forall \, Y \in V_{\varphi}    \big{\}}.$$
Therefore $X$ is horizontal if and only if $\Re \operatorname{tr}(X^*\varphi Y)=0 $ for every $Y \in \mathsf{B}_{sa}(\C^r)$. Since $X^*\varphi$ is skew-adjoint too this happens if and only if $X^*\varphi=0$.

Let us check that the projection $\eqref{bundleprojection3}$ is a riemannian submersion i.e. its differential induces an isometry from the horizontal space to the tangent space of $\mathsf{P}_{r}$. For horizontal vectors $X,Y \in T_{\varphi}\mathsf{S}\mathsf{t}(r,\mathsf{H})$ we have
\begin{eqnarray*}
	g\big{(}d_{\varphi}\pi^{\mathsf{S}\mathsf{t}}(X), d_{\varphi}\pi^{\mathsf{S}\mathsf{t}}(Y)\big{)}&=&
  \operatorname{tr}\big{(}  (X\varphi^* + \varphi X^*) (Y\varphi^*+\varphi Y^*) \big{)} \\&=&
2\Re \operatorname{tr}(X^*Y) + \Re\operatorname{tr}(X\varphi^*
Y\varphi^*  +\varphi X^*\varphi Y^*)\\&=& 2\Re \operatorname{tr}(X^*Y)=g(X,Y).
\end{eqnarray*}
We have used the properties of the trace and the fact that $X$ and $Y$ are horizontal. 

Following \cite{Edelman} we derive the geodesic equation
\begin{equation}\label{geodeq}
\ddot{\gamma}+\gamma ({\dot{\gamma}}^*\dot{\gamma})=0.\end{equation}
Starting with the the condition $\gamma^*\gamma=\operatorname{Id}_r$ and differentiating two times we get $\ddot{\gamma}^*\gamma +2\dot{\gamma}^*\dot{\gamma}+\gamma^*\ddot{\gamma}=0.$ If $\gamma$ is a geodesic, the normal component of the second derivative is zero i.e. $\ddot{\gamma}= -\gamma S$ for some curve $S(t)=S(t)^* \in \mathsf{B}(r,\mathsf{H})$. Inserting this condition in the previous equation we get \eqref{geodeq}. On the other hand if a curve $t \mapsto \mathsf{S}\mathsf{t}(r,\mathsf{H})$ satisfies \eqref{geodeq} is a geodesic because the normal component of its second derivative is zero.
 
 We take from \cite[Section 3.4.1]{Jurdevic}
 a closed formula for the geodesics starting from $\varphi_0 \in \st$.
 We continue to use the splitting \eqref{polarizationofphi} induced by $\varphi_0$ so that operators in $\mathsf{H}$ are $2\times 2$ block-matrices. For any skew-adjoint operator 
 $$\mathcal{M}=\left( \begin{array}{cc} A & B \\
 -B^* & 0
 \end{array}
 \right) 	
 \quad \textrm{with skew-adjoint} \quad  A:R(\varphi_0) \rightarrow R(\varphi_0),$$ put 
 $\mathcal{Q}:=\left( \begin{array}{cc} A/2 & 0 \\
 0 & 0	
 \end{array}
\right)
.$ 
 Then $\mathcal{Q}^*=-\mathcal{Q}$ and we have a curve $$t \longmapsto \gamma(t):=e^{t\mathcal{M}}e^{-t\mathcal{Q}}\varphi_0 \in \st.$$
\begin{proposition}The curve $\gamma$ is the geodesic  in $\st$ satisfying the initial conditions: $\gamma(0)=\varphi_0$ and $\dot{\gamma}(0)=\left( \begin{array}{cc} A/2 & B \\
 -B^* & 0
 \end{array}
 \right) 	\varphi_0
 .$ Since every tangent vector $X\in T_{\varphi_0}\st$ can be put in the form $X=\left( \begin{array}{cc} A/2 & B \\
 -B^* & 0
 \end{array}
 \right) 	\varphi_0$
(with skew-adjoint $A$) this exhausts all the geodesics. Concretely take
$$A=2(\varphi_0\varphi_0^*)X\varphi_0^*\rest{R(\varphi_0)} \quad \textrm{and} \quad B=\varphi_0X^*(\varphi_0\varphi_0^* -\operatorname{Id})\rest{R(\varphi_0)^{\bot}}.
$$
\end{proposition}
\begin{proof}
	The proof that $\gamma$ is a geodesic is the computation in \cite[Section 3.4.1]{Jurdevic} that we write for definiteness.
	Since we already know that $\gamma(t) \in \st$ at every time let's check that \eqref{geodeq} is satisfied  i.e. $\dot{\gamma}=Y$ and $\dot{Y}=-\gamma (Y^*Y).$ Put $\gamma(t)=g(t)\varphi_0$ with $g(t)=e^{t\mathcal{M}}e^{-t\mathcal{Q}}$. We also define $$\mathcal{P}=\left( \begin{array}{cc} A/2 & B \\ -B^* & 0 \end{array} \right)=\dot{\gamma}(0), \quad \textrm{and} \quad U(t):=e^{t\mathcal{Q}}\mathcal{P}e^{-t\mathcal{Q}}.$$
	It follows $\mathcal{P}+\mathcal{Q}=\mathcal{M}$ and $\dot{g}(t)=g(t)U(t)$. 
	We compute $Y=g(t)U(t)\varphi_0$ and 
	\begin{eqnarray*}
			\dot{Y}&=&\dot{g}(t)U(t) \varphi_0 + g(t)\dot{U}(t) \varphi_0=g(t)U^2 \varphi_0 +g(t)\dot{U}(t) \varphi_0 \\
			&=& g(t) e^{t\mathcal{Q}}(\mathcal{P}^2+[\mathcal{P},\mathcal{Q}]e^{-t\mathcal{Q}})\varphi_0 \\&=& g(t)\left( \begin{array}{cc}
 A^2/4 -e^{tA/2}BB^* e^{-tA/2} & 0 \\ 0 & 0 \end{array}\right)\varphi_0.	
 \end{eqnarray*}
 Before comparing this result with $-\gamma(Y^*Y)$ we notice that
 $U(t)^*=-U(t)$ and $ g(t)^*g(t)=\operatorname{Id}.$
Finally
\begin{eqnarray*}
-\gamma(Y^*Y)&=&g(t)\,\varphi_0 \varphi_0^*U(t)^* g(t)^*g(t)\, U(t)\varphi_0=-g(t) \varphi_0\varphi_0^* U^2\varphi_0  \\
&=&-g(t)\left( \begin{array}{cc} 1 &0 \\ 0 & 0 \end{array}\right) U^2 \varphi_0 \\ &=&-g(t)\left( \begin{array}{cc}
 A^2/4 -e^{tA/2}BB^* e^{-tA/2} & 0 \\ 0 & 0 \end{array}\right)\varphi_0.
	\end{eqnarray*}
It follows that $\gamma$ is a geodesic. The remaining statement is straightforward using the decomposition
$$X=\left( \begin{array}{cc}(\varphi_0\varphi_0^*)X \varphi_0^*  &  \varphi_0 X^*(\varphi_0\varphi_0^*-1)  \\ (1-\varphi_0\varphi_0^*)X\varphi_0^* & 0 \end{array} \right)\varphi_0,$$ where all the entries are intended restricted to $R(\varphi_0)$ or $R(\varphi_0)^\bot$.
\end{proof}
\begin{corollary}\label{propertygeodesic}
For a geodesic $\gamma: [0,1]\longrightarrow \mathsf{S}\mathsf{t}(r,\mathsf{H})$,
the image of the map $\gamma(t): \C^r \longrightarrow \mathsf{H}$ (for every $t$) is contained in the subspace of $\mathsf{H}$ spanned by $(\gamma(0),\dot{\gamma}(0))$. Of course its dimension is bounded by $2r$ and it follows that if $\gamma(0)$ and $\gamma(1)$ are independent then $\gamma(t)$ and $\dot{\gamma}(t)$ belong to $\operatorname{span}(\gamma(0),\gamma(1))$ for every $t\in [0,1]$. The geodesic {\em{moves inside a finite dimensional subspace of }} $\mathsf{H}$. 
\end{corollary}
\begin{proof}
From the formula of the geodesics we just have to examine the image of the operator
$e^{t\mathcal{M}}$
taking into account that $X=\dot{\gamma}(0)$. Then:
$$\mathcal{M}=\left( \begin{array}{cc} 2(\varphi_0\varphi_0^*) \dot{\gamma}(0)\varphi_0^* & \varphi_0\dot{\gamma}(0)^*(\varphi_0\varphi_0^*-\operatorname{Id})
 \\ (\operatorname{Id}-\varphi_0\varphi_0^*)\dot{\gamma}(0)\varphi_0^* & 0 \end{array} \right).$$ But  $R(\mathcal{M})\subset \operatorname{Span}(\varphi_0,\dot\gamma(0))$ and $\operatorname{Span}(\varphi_0,\dot\gamma(0))$ is stable under $\mathcal{M}$.
\end{proof}

\noindent An embedding $\iota:\mathsf{K} \hookrightarrow \mathsf{H}$ of Hilbert spaces induces embeddings  $\iota_*:\mathsf{S}\mathsf{t}(r,\mathsf{K}) \hookrightarrow \mathsf{S}\mathsf{t}(r,\mathsf{H})$ and  $\iota_*:\mathsf{P}_{r}(\mathsf{K}) \hookrightarrow \mathsf{P}_{r}(\mathsf{H})$ where we make a slight abuse of notation for using the same symbol for the two maps. Also the notation used for the Grassmannians of different Hilbert spaces is  self-explanatory. Indeed we define $\iota_* \varphi= \iota \circ \varphi$. This is $\mathsf{U}_r$-equivariant and induces the map at the level of the Grassmannians. These embeddings are very useful according to the following.
\begin{theorem}{\em{\cite{Mennucci}}}.
Let $\mathsf{K}$ be a Hilbert space; for every embedding $\iota:\mathsf{K}\longrightarrow \mathsf{H}$ the corresponding $ \iota_*:\mathsf{S}\mathsf{t}(r,\mathsf{K}) \hookrightarrow \mathsf{S}\mathsf{t}(r,\mathsf{H})$ is an isometric embedding with totally geodesic image. Moreover:
	\begin{enumerate}
	\item 	
	When $\operatorname{dim}\mathsf{K} \geq 2r$,
	if we denote with $\mathsf{d}_{\mathsf{H}}$ and $\mathsf{d}_{\mathsf{K}}$ the respective distances then $\mathsf{d}_{\mathsf{H}}(\iota_*(x),\iota_*(y))=\mathsf{d}_{\mathsf{K}}(x,y)$ for every $x,y \in \mathsf{S}\mathsf{t}(r,\mathsf{K})$.
	\item Let again $\operatorname{dim}\mathsf{K} \geq 2r$ and let $\gamma$ be a minimal geodesic inside $\mathsf{S}\mathsf{t}(r,\mathsf{K})$. Then $\iota_* \circ \gamma$ is a minimal geodesic.
	\item The diameter of $\mathsf{S}\mathsf{t}(r,\mathsf{H})$ equals the diameter of $\mathsf{S}\mathsf{t}(r,\C^{2r})$.
	\item Any two points in $\mathsf{S}\mathsf{t}(r,\mathsf{H})$ can be joined by a minimal geodesic. Every minimal geodesic $\gamma$ lies inside some submanifold $\mathsf{S}\mathsf{t}(r,V)$ where $V \subset \mathsf{H}$ is a $2r$-dimensional subspace depending on $\gamma$.
\item Fix two points $x,y \in \mathsf{S}\mathsf{t}(r,\mathsf{H})$; then $y$ is in the cut locus of $x$ if and only if there is a $2r$-dimensional subspace $V \subset \mathsf{H}$ such that $x=\iota_*(\tilde{x})$, $y=\iota_*(\tilde{y})$ and $\tilde{y}$ is in the cut locus of $\tilde{x}$.
	\end{enumerate}
	All these properties hold for the Grassmannian manifold $\mathsf{P}_{r}(\mathsf{H})$ too. In particular any two points $x,y \in \mathsf{P}_{r}(\mathsf{H})$ are joined by a minimal geodesic.
\end{theorem}
\begin{proof}
As already mentioned, the proof in \cite{Mennucci} is performed for the real Stiefel and Grassmannian manifolds. 
The key being the fundamental property of the geodesics in Corollary \ref{propertygeodesic}.
One checks immediately that every argument is transferred without changes to the complex case. We write here the proof in loc. cit. in a somewhat sketchy way for the first statement of the Theorem and of properties $1.,2.$ and $4.$ both for the Stiefel and the Grassmannians manifolds. We will use these in the proof of Theorem \ref{Alexandrov} below.

First one checks the following fact: \\

a). Fixed $y\in \mathsf{St}(r,\mathsf{H})$ the set of all the $x$ such that the columns of $x,y$ are independent is dense in the Stiefel manifold. \\
Then the proof follows the steps:
\begin{description}
	\item[Step 1.] The first statement of the Theorem (for the Stiefel manifold) and points $1$., $2$. and $4$. hold when $\mathsf{H}$ is finite dimensional.
	\item[Step 2.] The statements in {\bf{Step 1}} hold in the infinite dimensional case.
	\item[Step 3.] Every statement also holds for the Grassmannian.
\end{description}
{\bf{Proof of Step 1.}} For $\iota: \mathsf{K} \hookrightarrow \mathsf{H}$ let $\mathsf{U}(\iota(\mathsf{K})^\bot)$ be the unitary group of the complement. it is included (diagonally) in $\mathsf{U}(\mathsf{H})$ and acts by isometries on $\mathsf{S}\mathsf{t}(r,\mathsf{H})$ with fixed points being exactly $\iota_*(\mathsf{S}\mathsf{t}(r,\mathsf{K}))$. Therefore $\iota_*(\mathsf{S}\mathsf{t}(r,\mathsf{K}))$ is totally geodesic because is the fixed point set of a set of isometries. For the statement $1.$
we prove it only for those couple of points $x,y$ of the Stiefel manifold with independent images.
Then by Lipschitz continuity of the distances and by the fact a). it will hold for every couple of points.
Now $\mathsf{d}_{\mathsf{K}}(x,y)\geq \mathsf{d}_{\mathsf{H}}(\iota_* (x),\iota_*(y))$ because $\iota_*(\mathsf{S}\mathsf{t}(r,\mathsf{K}))$ is totally geodesic. For the reversed inclusion, let $\gamma \subset \mathsf{S}\mathsf{t}(r,\mathsf{H})$ be a minimal geodesic (Hopf--Rinow in finite dimensions) joining $\iota_*(x)$ with $\iota_*(y)$. Then since the images of $\iota_*(x)$ and $\iota_*(y)$ are independent, by Corollary \ref{propertygeodesic} we have that the image of $\gamma(t)$ is contained in the span of the images of $\iota_*(x)$ and $\iota_*(y)$ which is contained in $\mathsf{K}$. 
In other words $\gamma= \iota_* \circ \widetilde{\gamma}$ for a geodesic  $\widetilde{\gamma} \subset \mathsf{S}\mathsf{t}(r,\mathsf{K})$. Using $\widetilde{\gamma}$ the inequality $\mathsf{d}_{\mathsf{H}}(\iota_*(x),\iota_*(y)) \geq \mathsf{d}_{\mathsf{K}}(x,y)$ immediately follows.
Point $2$. is direct consequence of point $1$. Point $4$ is already known from the Corollary \ref{propertygeodesic}.

{\bf{Proof of Step 2.}}
The unique point which has a different proof in the infinite dimensional case is point $1$. Here of course $\mathsf{d}_{\mathsf{K}}(x,y) \geq \mathsf{d}_{\mathsf{H}}(\iota_*(x),\iota_*(y))$. To prove the converse, one takes any smooth path $\zeta$ connecting $\iota_*(x)$ and $\iota_*(y)$. We can divide $\zeta$ in subpaths $\zeta\rest{[t_i,t_{i+1}]}$ $(i=1,...,n)$ such that each one is contained in a normal neighborhood and using the exponential map each couple $\zeta(t_i)$ and $\zeta(t_{i+1})$ can be joined by a minimising geodesic. We get a piecewise smooth path $\eta(t)$ joining $\iota_*(x)$ and $\iota_*(y)$ with $\ell(\eta) \leq \ell(\zeta)$. Moreover from all the extreme points $(\zeta_{t_i})_{i=1,...,n-1}$ and the velocities $(\dot{\eta}_{t_i})_{i=1,...,n-1}$ we manifacture a finite dimensional vector space $\widetilde{\mathsf{K}}$ which contains every image of the map $\eta(t)$ for every $t$. Of course we can enlarge it to ensure $\mathsf{K}\subset \widetilde{\mathsf{K}}$. Now we apply the finite dimensional case (in $\widetilde{\mathsf{K}}$) to estimate
$$\mathsf{d}_{\mathsf{K}}(x,y)=\mathsf{d}_{\widetilde{\mathsf{K}}}(\iota_*(x),\iota_*(y)) \leq \ell(\eta)\leq \ell(\xi)$$ and we are done.

{\bf{{Proof of step 3.}}}
We check just point $1.$ and $2.$ in the finite dimensional case because the infinite dimensional etension is similar to the one performed for the Stiefel case. First point: we have $\mathsf{d}_{\mathsf{P}_r(\mathsf{K})}(x,y) \geq \mathsf{d}_{\mathsf{P}_r(\mathsf{H})}(\iota_*(x),\iota_*(y))$ as before. 
Also assume that the subspaces $x$ and $y$ in the Grassmannian are independent and they generate a $2r$-dimensional space. Of course the corresponding fact a). also holds for the Grassmannian.
Now let $\gamma$ be a minimal geodesic in $\mathsf{P}_r(\mathsf{H})$ joining $\iota_*(x)$ and $\iota_*(y)$. Lift this to a curve $\zeta(t)$ in the Stiefel manifold $\mathsf{S}\mathsf{t}(r,\mathsf{H})$. The images of the maps $\zeta(0)$ and $\zeta(1)$ are exactly $x$ and $y$. This means that $\zeta$ belongs to the image of $\mathsf{S}\mathsf{t}(r,\mathsf{K})$ and in turn that $\gamma$ belongs to the image of the embedding $\iota_*:\mathsf{P}_r(\mathsf{K}) \hookrightarrow  \mathsf{P}_r(\mathsf{H})$. It follows $\mathsf{d}_{\mathsf{P}_r(\mathsf{K})}(x,y) \leq \mathsf{d}_{\mathsf{P}_r(\mathsf{H})}(\iota_*(x),\iota_*(y)).$ As before this fact implies the point 2.
\end{proof}

\begin{theorem}\label{Alexandrov}Every connected component $\mathsf{P}_r$ of finite rank Grassmannian is an Alexandrov space with non negative curvature.	
\end{theorem}
\begin{proof}
	According to \cite{Petrunin} a complete metric space $\mathcal{X}$ with intrinsic metric i.e. the metric derived from the length of curves is Alexandrov with non negative scalar curvature if and only if any four points $p,x,y,z \in \mathcal{X}$ satisfy the inequality
	$$\mathsf{d}(p,x)^2 + \mathsf{d}(p,y)^2 + \mathsf{d}(p,z)^2 \geq 1/3 (\mathsf{d}(x,y)^2 + \mathsf{d}(y,z)^2 + \mathsf{d}(z,x)^2).$$
	For a finite dimensional manifold this condition is equivalent to the non negativity of the sectional curvature.
	But in our case such a configuration of four points is always included in a finite dimensional totally geodesic submanifold of non negative sectional curvature.
\end{proof}

Now we prove a simple fact that will be useful later. 
\begin{proposition}\label{P:dimension1}
	Let $P,Q \in \mathsf{P}$ then $Q\leq P \Longrightarrow T_P\mathsf{P} \subset T_Q\mathsf{P}.$ 
Let moreover $Q\leq P$ be projections in $\mathsf{P}_c$ and let $\gamma:[0,1]\rightarrow \mathsf{P}_c$ be the geodesic 
	$\gamma(t)=e^{itZ}Pe^{-itZ}$ starting from $P$. Then $$Q_1:=e^{iZ}Qe^{-iZ}\leq \gamma(1)=:P_1 \quad \textrm{and}\quad   \mathsf{d}(Q,Q_1) \leq \ell(\gamma).$$
	In particular taking $\gamma$ minimal  $\mathsf{d}(Q,Q_1)\leq \mathsf{d}(P,P_1).$ 
\end{proposition}
\begin{proof}
	Let $Z\in T_P\mathsf{P}$; we have to show that $QZQ= (1-Q)Z(1-Q)=0$. This is immediate to check under the block decomposition induced by $P$ where:
	\begin{equation}\label{QandZ}Q=\left( \begin{array}{cc}QP & 0 \\ 0& 0 \end{array}\right), \quad \quad Z=\left( \begin{array}{cc} 0 & X \\ X^* & 0 \end{array}\right).\end{equation}
Now $Q(t):=e^{itZ}Qe^{-itZ}$ is a geodesic from $Q$ to $Q_1$ with $\dot{Q}(0)=i[Q,Z]$ and $\|\dot{Q}(0)\|^2_{T\mathsf{P}_c}=\operatorname{tr}(\dot{Q}(0)^2)$.
 Using \eqref{QandZ} we easily compute 
$$\dot{Q}(0)^2=\left( \begin{array}{cc}   QX^*XQ & 0 \\
 0 & XPQPX^* \end{array}\right).
$$
From the properties of the trace we get $\|\dot{Q}(0)\|_{T_Q\mathsf{P}_c}^2=2\operatorname{tr}(QX^*XQ)$.
In the same way $\|\dot{\gamma}(0)\|^2_{T\mathsf{P}_c}=2\operatorname{tr}(X^*X)$. The result is clear from the positivity of $X^*X$.
  \end{proof}

\bigskip
\subsection{Normal States}\label{Ss:Normal}

Let $\mathcal{A}$ be a $C^*$-algebra.
A linear functional $\varphi:\mathcal{A} \longrightarrow \mathbb{C}$ is positive if
 $\varphi(a^*a)\geq 0$ for every $a\in \mathcal{A}$. Then $\varphi$ is automatically bounded; if $\|\varphi\|=1$ it is called a state. When the algebra is unital this normalisation is equivalent to the condition $\varphi(1)=1$.
Denoted with $\mathcal{S}(\mathcal{A})$, the space of the states of $\mathcal{A}$ included in the dual $\mathcal{A}^*$ and considered with the topology induced by the $w^*$-one. 
For convenience of the reader 
we include a sketch of the proof of the following well-know fact.
\begin{proposition}
The space of states is always convex. When $\mathcal{A}$ is unital it is compact.	
\end{proposition}
 \begin{proof}
 When $\mathcal{A}$ is unital the convexity is immediate. In general every $C^*$-algebra has an {\em{approximate unit}}: an increasing net $(u_{j})_{j\in J}$ of positive elements with $\|u_j\|\leq 1$ for every $j \in J$ such that	
 $$\lim_{j\in J}\|a-u_ja\|=0, \quad \lim_{j \in J}\|a-au_j\|=0, \quad \forall a \in \mathcal{A}.$$
 If the algebra is separable  we can take a sequence for $(u_j)$.
 Now for a linear bounded functional $\varphi : \mathcal{A} \longrightarrow \C$ 
positivity implies
 $\lim_{j\in J}\varphi(u_j)=\|\varphi\|$ (the converse statement also holds but we don't need it). It follows that convex combinations of states are states. The rest of the proof is just the theorem of Banach--Alaoglu.
 \end{proof}
We will denote by $\mathcal{PS}(\mathcal{A})$ the set of \emph{pure states} 
that is the extreme boundary of $\mathcal{S}(\mathcal{A})$ i.e. the subset of extremal points of the boundary of the convex set $\mathcal{S}(\mathcal{A})$.

\smallskip

Our object of study will be the space of states of $\mathbb{K}=\mathbb{K}(\mathsf{H})$, the $C^*$-algebra of compact operators. We have an identification 
 \begin{equation}\label{duality}
 \mathbb{K}' \cong \mathcal{L}^1 \quad \quad \textrm{(Banach dual)}\end{equation}
with the Banach space of the trace class operators $\mathcal{L}^1(\mathsf{H})=\{A \in \mathsf{B}(\mathsf{H}):  \operatorname{tr}|A|< \infty\}$ with norm $\|A\|_1:=\operatorname{tr}|A|$. Here $A \in \mathcal{L}^1$ defines the functional $T \mapsto \operatorname{tr}(AT)$ for $T \in \mathbb{K}$. One can also prove that $\mathcal{L}^1$ is the {\em{predual}} of $\mathsf{B}(\mathsf{H})$ in the sense that 
$(\mathcal{L}^1)' =\mathsf{B}(\mathsf{H})$. 
\smallskip
Restricting to the positive and norm one functionals we immediately see that 
for any state $\varphi \in \mathcal{S}(\mathbb{K}(\mathsf{H}))$ there exists 
a unique {\em{density matrix}}, an operator $\rho \in \mathcal{L}^1$ positive with
$$
\operatorname{tr}(\rho)=1, \quad \varphi(B)=\operatorname{tr}(\rho B), 
\quad \textrm{for every }B \in \mathbb{K}.
$$
Viceversa all the density matrices give states on $\mathbb{K}$. We define such space of density matrices by $\mathcal{C}(\mathsf{H})$ or just $\mathcal{C}$, if the context is clear: 
\begin{equation}\label{E:density}
\mathcal{C}(\mathsf{H}) : = \{\rho \in \mathcal{L}^{1} \colon \rho \geq 0, \ \tr(\rho) = 1 \},
\end{equation}
with the identification denoted by
\begin{equation}\label{ddmatrices}
\mathcal{C}(\mathsf{H}) \ni \rho \longmapsto \varphi_\rho \in \mathcal{S}(\mathbb{K}).\end{equation}
Viceversa we may, sometimes, use the notation $\rho_{\varphi}$ or $\rho(\varphi)$ for the density matrix of $\varphi$.

\begin{example}Every unit vector $\xi \in \mathsf{H}$ defines a state 
$\omega_{\xi}: \mathsf{B}(\mathsf{H}) \longrightarrow \R$ by 
$\omega_{\xi}(B)=\langle \xi, B\xi \rangle.$ 
The density matrix of $\omega_{\xi}$ is the rank one operator 
$\rho:\mathsf{H}\to \mathsf{H}$ with 
$\rho(\eta)=\langle \xi, \eta \rangle \xi$. 
This follows from: $\operatorname{tr}(\rho B)=\operatorname{tr}( B\rho)
=\langle \xi, B\xi \rangle.$
In Dirac notation our vector is $|\,\xi \rangle$ so that 
$$
\rho=|\,\xi\rangle \langle \xi \,|.
$$
\end{example}

\noindent Density matrices define states of $\mathbb{K}$ that 
extend to states of $\mathsf{B}(\mathsf{H})$; on the 
other hand there are many states on $\mathsf{B}(\mathsf{H})$ which are not in this form.
Precisely a state $\varphi \in \mathcal{S}(\mathsf{B}(\mathsf{H}))$ 
comes from a density matrix if 
and only if it satisfies one of the following 
equivalent properties (see \cite[Theorem 4.12]{Landsman}, \cite[Theorem 7.1.8]{Kadison} and \cite[Theorem 1, Part I, Chapter 4]{Dixmier}):
\begin{enumerate}
\item it is {\em{normal}}: $
\varphi(T) = \sup_{\mathcal{F}} \varphi(F)
$
for every directed family $\mathcal{F}\subset \mathsf{B}(\mathsf{H})^+$ 
of positive operators with $T=\sup_{\mathcal{F}}F$. 
\item The state is {\em{completely additive}}:
for every orthogonal family $(p_j)_j$ of projections  ($p_j^*=p_j$ and $p_jp_i=\delta_{ij}p_j$) then  
$$
\varphi(\sum_j p_j)=\sum_j \varphi(p_j). 
$$
The sum ${\sum}_j p_j$ is defined as the projection on the closure of the smallest subspace in $\mathsf{H}$ containing all the $p_j\mathsf{H}$. This is 
is exactly the operation of forming $\sup_j p_j$ in the partially ordered set of all the projections in $\mathsf{B}(\mathsf{H})$ with the order given by the inclusion $e\leq f$ iff $e\mathsf H \subset f\mathsf{H}$ (see \cite{Landsman}).
Also $\sum_j p_j$ is the limit of all the finite sums in the strong operator topology.
\item There is a sequence of vectors $(\xi_n)_n$ with $\sum_{n=1}^{\infty}\|\xi_n\|^2=1$ such that $$\varphi=\sum_{n=1}^{\infty} \omega_{\xi_n}$$ in the sense of norm convergence. 
The vectors $\xi_n$ can be taken pairwise orthogonal \cite[Theorem 7.1.9]{Kadison}.
\end{enumerate}

By the spectral theorem we see that that a pure state $\varphi$ of $\mathbb{K}$ is necessarily a vector state i.e. in the form $\varphi=\omega_{\xi}$ for a unit vector $\xi$. Of course $\omega_{\xi}=\omega_{\eta}$ if and only if  $\xi=\lambda \eta$ for a {\em{phase}}, a scalar $\lambda\in \C$ with $|\lambda|=1$. Thus $\mathcal{P}\mathcal{S}(\mathbb{K})\cong \mathbb{P}(\mathsf{H})$. On the right we have the projective space of $\mathsf{H}$, the quotient of the unit sphere by the $\mathsf{U}(1)$-action by scalar multiplication.

We conclude with a basic useful fact. 
\begin{lemma}\label{L:normalsum}
For a normal state in the form $\varphi=\sum_{n=1}^{\infty} \omega_{\xi_n}
$ with $\sum_{n=1}^{\infty}\|\xi_n\|^2=1$,
 let $P_n$ be the projection onto $[ \xi_n]$ (the line generated by the vector). Then the density matrix of $\varphi$ is:
$$\rho(\varphi)=\sum_{n=1}^{\infty}\|\xi_n\|^2 P_n \quad \textrm{norm convergence of operators.}$$  
\end{lemma}

\begin{proof}
This fact is more general. 
A proof can be found in \cite[Theorem 7.1.9]{Kadison} (see also the following remark therein). 
In our case the proof is simpler. 
The series $\sum_{n=1}^{\infty}\|\xi_n\|^2P_n$ converges in the operator norm 
to a non negative operator $T$. 
Of course $T$ is compact and can be diagonalised 
$T=\sum_{n=1}^{\infty}\lambda_n^2\, |\eta_n \rangle \langle \eta_n|$ with the 
complete orthonormal system $\{\eta_n\}_n$. 
Now for every $A\in \mathsf{B}(\mathsf{H})$ we compute
$$
\varphi(A)=\sum_{n=1}^{\infty}\langle A\xi_n, \xi_n\rangle 
=\sum_{n=1}^{\infty} \langle A \sum_{m=1}^{\infty} \langle \xi_n, 
		\eta_m \rangle \eta_m, \xi_n \rangle 
=\sum_{n=1}^{\infty}\sum_{m=1}^{\infty} \langle A\eta_m, \xi_n \rangle \langle \xi_n, \eta_m \rangle.
$$
We can interchange the sums because the series converges absolutely and using the identity $P_n=\dfrac{|\xi_n \rangle \langle \xi_n|}{\|\xi_n\|^2}$ we get
$$
\varphi(A)=\sum_{m=1}^{\infty}\langle A\eta_m, \sum_{n=1}^{\infty} \overline{\langle \xi_n, \eta_m \rangle} \xi_n \rangle = \sum_{m=1}^{\infty}\langle A\eta_m, (\sum_{n=1}^{\infty} \|\xi_n \|^2P_n ) \eta_m \rangle = \sum_{m=1}^{\infty}\langle A\eta_m, T\eta_m \rangle.
$$ 
This is $\varphi(A)=\operatorname{tr}(TA)$. 
\end{proof}

We will denote with $\mathcal{S}_{n}(\mathsf{B}(\mathsf{H}))$ the collection 
of all normal states. To summarise we have recalled that 
$$
\mathcal{S}(\mathbb{K}) = \mathcal{S}_{n}(\mathsf{B}(\mathsf{H})), \qquad 
\mathcal{S}_{n}(\mathsf{B}(\mathsf{H})) \simeq \mathcal{C}(\mathsf{H}),
$$
where the symbol $\simeq$ denotes an isomorphism between the two convex sets. This isomorphism maps extremals to extremals: 
any
 pure state $\omega$ on $\mathbb{K}$ 
has a unique  extension to a normal  state $\omega'$ on $\mathsf{B}(\mathsf{H})$ given by the same density operator which is extremal for $\mathcal{S}_n(\mathsf{B}(\mathsf{H}))$.
We refer to \cite{Landsman} for more details.
Based on this we make the following
\begin{definition}We denote with $\mathcal{P}\mathcal{S}_n(\mathsf{B}(\mathsf{H}))$
the set of the \em{pure normal states} of $\mathsf{B}(\mathsf{H})$. These are
precisely the extremals of $\mathcal{S}_n(\mathsf{B}(\mathsf{H}))$ identifiable with $\mathbb{P}(\mathsf{H})$ the projective space of $\mathsf{H}$. 
\end{definition}

\subsubsection{Topology on the space of states}
 We discuss now the various topologies that can be  considered on $\mathcal{S}(\mathbb{K})$ according to the inclusion $\mathcal{S}(\mathbb{K}) \subset  \mathbb{K}'$.
\begin{itemize}
\item The uniform topology is the metric topology induced by the Banach dual structure on $\mathbb{K}'$. In terms of two density matrices:
$$
\| \varphi_{\rho}-\varphi_{\mu}\|=\sup_{B \in \mathbb{K}, \, \,\|B\|=1} |\operatorname{tr}(\rho-\mu)B|= \operatorname{tr}|\rho -\mu|=\|\rho-\mu\|_1,
$$
because by the Kaplansky density Theorem the supremum can be computed over the unit ball of $\mathsf{B}(\mathsf{H})$ leading immediately to the trace norm. 
\item The weak$^*$ topology $\sigma(\mathcal{S}(\mathbb{K}), \mathbb{K})$ is induced by the weak$^*$ topology on $\mathbb{K}'$. 
In particular 
$\xymatrix@1{{\varphi}_{\rho_n}\ar@^{>}[r]^{w^*}&\varphi_{\rho}}$ 
if $\varphi_{\rho_n}(B)=\operatorname{tr}(\rho_nB) \longrightarrow \varphi_{\rho}(B)
=\operatorname{tr}(\rho B)$ 
for every $B\in \mathbb{K}$.
\item Instead of evaluating against every $B\in \mathbb{K}$ in the above convergence we can take all the tests $B\in \mathsf{B}(\mathsf{H})$. 
This defines $\sigma(\mathcal{S}(\mathbb{K}),\mathsf{B}(\mathsf{H}))$ 
called the weak topology in virtue of the identification 
$\mathsf{B}(\mathsf{H})=(\mathcal{L}^1)'$.
\end{itemize}
Importantly Robinson proved that all the above topologies coincide 
\cite[Theorem 1]{Robinson}:

\begin{theorem}\label{T:convergenceweakL1}
 The three topologies above described all coincide. In particular
for a sequence $\rho_n \in \mathcal{L}^1$ we have
$$
\xymatrix@1{{\rho}_n \ar[r]^{\mathcal{L}^1}& \rho& \textrm{iff}& \varphi_{{\rho}_n} \ar@^{>}[r]^{w^*}& \varphi_{\rho}}.
$$
\end{theorem}
%

\subsubsection{Partial traces and marginals}\label{Ss:marginals}

Let now $\mathsf{H}$ and $\mathsf{K}$ be two Hilbert spaces. The use of $\mathsf{K}$ does not create confusion with the notation designated for the compact operators which is $\mathbb{K}$.
 The tensor product Hilbert space $\mathsf{H} \otimes \mathsf{K}$
corresponds, in quantum mechanics, to a composite system. The isomorphism 
$\mathsf{B}(\mathsf{H}\otimes \mathsf{K})\cong \mathsf{B}(\mathsf{H})\otimes \mathsf{B}(\mathsf{K})$ 
induces two maps with the meaning of taking marginals:
$$
J^{\mathsf{H}}_{\flat}:\mathcal{S}(\mathsf{B}(\mathsf{H}\otimes \mathsf{K}))\longrightarrow \mathcal{S}(\mathsf{B}(\mathsf{H})),
$$ 
and the corresponding map $J^{\mathsf{K}}_{\flat}$.
 For definiteness we give the formula of the first one by dualising the inclusion 
 $J^{\mathsf{H}} : \mathsf{B}(\mathsf{H}) \longrightarrow \mathsf{B}(\mathsf{H})\otimes \mathsf{B}(\mathsf{K})$, $T \longmapsto T \otimes \operatorname{Id}_{\mathsf{K}}$:
$$
J^{\mathsf{H}}_{\flat}\varphi(T)=\varphi(T \otimes \operatorname{Id}_{\mathsf{K}}).
$$ 
\noindent Let us describe now partial traces. We follow closely the lecture notes \cite{Attall} where all the proofs can be found. \\
\noindent Assume that $\mathsf{H}$ and $ \mathsf{K}$ are separable. Every vector $\xi \in \mathsf{K}$ defines  linear bounded operators
$$
\begin{array}{ccc}
R_{\xi}:\mathsf{H}\longrightarrow \mathsf{H}\otimes \mathsf{K}, & {}&R_{\xi}^*:\mathsf{H}\otimes \mathsf{K} \longrightarrow \mathsf{H}, 
\end{array}$$
uniquely specified on simple tensors by
	$$
	\begin{array}{ccc}
R_{\xi} \eta= \eta \otimes \xi & {} &R_{\xi}^* \zeta \otimes \eta= \langle \xi, \eta \rangle \zeta .
\end{array}
$$
It is immediate to verify that $\|R_{\xi}\|=\|R_{\xi}^*\|=\|\xi \|$. If $T \in \mathsf{B}(\mathsf{H}\otimes \mathsf{K})$ then we get a bounded operator on $\mathsf{H}$ via\footnote{in \cite{Attall} is denoted by ${}_{\mathsf{K}}\langle \xi | T|\xi \rangle_{\mathsf{K}}$ } 
$${}_{\xi}T_{\xi}:=R_{\xi}^* TR_{\xi}.$$ By definition:
$\langle \zeta , {}_{\xi}T_{\xi}\,\eta \rangle =\langle \zeta \otimes \xi, T \eta \otimes \xi \rangle ,$  for every $\zeta, \eta \in \mathsf{H}$ and one proves $$T \in \mathcal{L}^1(\mathsf{H}\otimes \mathsf{K}) \quad \implies \quad  {}_{\xi}T_{\xi} \in \mathcal{L}^1(\mathsf{H}).$$
\begin{theorem}
Let $T\in \mathcal{L}^1(\mathsf{H}\otimes \mathsf{K})$ be a trace class operator; there is a unique trace class operator $\operatorname{Tr}_{\mathsf{K}}(T) \in \mathcal{L}^1(\mathsf{H})$ such that 
\begin{equation}\label{partialtrace}
\operatorname{tr}(\operatorname{Tr}_{\mathsf{K}}(T)\,B)=\operatorname{tr}(T( B\otimes \operatorname{Id}_{\mathsf{K}}))\end{equation} for every $B\in \mathsf{B}(\mathsf{H})$.
Concretely $\operatorname{Tr}_{\mathsf{K}}(T)$, that we call the partial trace with respect to $\mathsf{K}$, can be constructed taking any ortonormal basis $(\xi_n)_n$ of $\mathsf{K}$:
$$\operatorname{Tr}_{\mathsf{K}}(T) = \sum_n {}_{\xi_n}T\,{}_{\xi_n}\quad  (\textrm{series convergent in }\mathcal{L}^1(\mathsf{H})).$$
We have the following properties 
\begin{itemize}
\item	$\operatorname{Tr}_{\mathsf{K}}(T)=\operatorname{tr}(B)A\quad$  if $T=A\otimes B$ with $A\in \mathcal{L}^1(\mathsf{H})$ and $B\in \mathcal{L}^1(\mathsf{K})$,
\item $\operatorname{tr}(\operatorname{Tr}_{\mathsf{K}}(T))=\operatorname{tr} T$
\item $\operatorname{Tr}_{\mathsf{K}}((A\otimes \operatorname{Id}_{\mathsf{K}}) T(B \otimes
\operatorname{Id}_{\mathsf{K}}))=A \operatorname{Tr}_{\mathsf{K}}(T) B\quad $ for every $A,B \in \mathsf{B}(\mathsf{H})$.\end{itemize}
\end{theorem}
Exchanging the role of $\mathsf{H}$ and $\mathsf{K}$ we define in the same way the partial trace $\operatorname{Tr}_{\mathsf{H}}$. If $\mathsf{K}=\mathsf{H}$, the unique case we shall treat we denote with $\operatorname{Tr}_1$ and $\operatorname{Tr}_2$ the two partial traces. For instance for $\xi \otimes \eta \in \mathsf{H} \otimes \mathsf{H}$:
$$
\operatorname{Tr_1} \Big{(}| \xi \otimes \eta \rangle \langle \xi \otimes \eta | \Big{)}=||\xi \|^2 \, |\eta \rangle \langle \eta |, \quad \operatorname{Tr}_2 \Big{(}|\xi \otimes \eta \rangle \langle \xi \otimes \eta |\Big{)}= \|\eta\|^2\, | \xi \rangle \langle \xi|.
$$
\noindent Let now $\varphi \in \mathcal{S}(\mathsf{B}(\mathsf{H}\otimes \mathsf{K}))$ be a state with density matrix $\rho_{\varphi}$. The defining property of the partial trace \eqref{partialtrace} immediately means that, for the density matrix of the first marginal we have:
$$
\rho( J^{\mathsf{H}}_{\flat} \varphi )=\operatorname{Tr}_{\mathsf{K}}(\rho_{\varphi}).
$$ 
The density matrices of the partial traces are usually called {\em{reduced density matrices}}.

\smallskip

Given normal states $\varphi \in \mathcal{S}_n(\mathsf{B}(\mathsf{H}))$ and $\psi \in \mathcal{S}_n(\mathsf{B}(\mathsf{K}))$ the tensor product $\varphi \otimes \psi $ is a normal state on $\mathsf{B}(\mathsf{H}\otimes \mathsf{K})$. We say that $\varphi \otimes \psi$ is {\em{separable}}. More generally we agree with \cite{Eisert} on the following.

\begin{definition}[Separable and entangled states]\label{D:entangled}
A normal state $\varphi$ on $\mathsf{B}(\mathsf{H}\otimes \mathsf{K})$ is separable if it is limit in the trace norm of a sequence $\varphi_k$ of normal states each of them is an infinite convex combination of states:
$$\varphi_k= \sum_i p_i^{(k)}\eta_{\mathsf{H}}^{(k,i)}\otimes \eta_{\mathsf{K}}^{(k,i)},$$ with the coefficients $\{p_i^{(k)}\}_{i=1}^{\infty}$ forming a probability measure. 
The trace norm is referred in the above sum to the corresponding density matrices.
A normal state on $\mathsf{B}(\mathsf{H}\otimes \mathsf{K})$ is {\em{entangled}} if it is not separable.
\end{definition}

Notice in particular that a pure state 
$\omega_{\zeta}$ with $\zeta \in \mathsf{H}\otimes\mathsf{K}$ is separable
if and only if $\zeta$ is a simple tensor product, i.e. $\zeta = \xi \otimes \eta$.

\begin{notation}
Summing up the notation we are using: $\rho$ is a generic density matrix, $\rho_{\varphi}$ or $\rho(\varphi)$ is the density matrix of the normal state $\varphi$. If instead we start with $\rho$, then $\varphi_{\rho}$ is the associated state. Finally vector states defined by $\xi$ are called $\omega_{\xi}$ with density matrix $\rho=\rho(\omega_{\xi})=|\, \xi \rangle \langle \xi\, |.$
\end{notation}

\section{Spectral-projections measures}\label{S:spectral measure}

To any density matrix $\rho \in \mathcal{C}(\mathsf{H})$ we can associate 
its unique spectral decomposition for self-adjoint and compact operators
\begin{equation}\label{E:spectral}
\rho = \sum_{i} \lambda_{i} P_{V_{i}},\qquad V_{i} \subset \mathsf{H}, \qquad 
\tr(\rho) = \sum_{i} \lambda_{i} \dim(V_{i}) = 1,
\end{equation}
where $\lambda_{i}> 0$ are the eigenvalues of $\rho$ and
$P_{V_{i}} \in \mathsf{P}_{c}$ is the  projection onto the corresponding finite dimensional eigenspace $V_{i}$.
In \eqref{E:spectral} the eigenvalues are meant to be listed without repetitions so that:
$$i\neq j \quad \Longrightarrow \quad  \lambda_{i}\neq \lambda_{j} \quad \textrm{and} \quad V_{i} \perp V_{j}.$$
Then it is natural to identify the spectral decomposition \eqref{E:spectral} 
with a discrete, finite and non-negative measure over $\mathsf{P}_{c}$. 
Before going into details we fix the notation: 
$\mathcal{P}(\mathsf{P}_{c})$ it will denote the space 
of Borel probability measures 
(i.e. non-negative and total mass 1) defined over the Polish space 
$(\mathsf{P}_{c},\sfd)$ while $\mathcal{M}_{+}(\mathsf{P}_{c})$ is the space 
of non-negative Radon measures.
We now introduce the following set
\begin{equation}\label{E:discrete}
\mathcal{D}(\mathsf{P}_{c}) : = \Big{\{}  \mu = \sum_{i}\lambda_{i} \delta_{P_{i}} \colon \ P_{i} \in \mathsf{P}_{c}, \ \lambda_{i} \geq 0  \Big{\}},
\end{equation}
with $\mathcal{D}(\mathsf{P}_{c})$ mnemonic for ``discrete'' measures. 
Then we consider the following subsets
\begin{equation}\label{E:discretetr}
\mathcal{D}_{1}(\mathsf{P}_{c}) : = 
\Big\{ \mu \in \mathcal{D}(\mathsf{P}_{c}) 
\colon \tr(\cdot)\,\mu  \in \mathcal{P}(\mathsf{P}_{c}) \Big\},
\end{equation}
playing the role of probability measures 
and
\begin{equation}\label{E:discretetrort}
\mathcal{D}_{1}^{\perp}(\mathsf{P}_{c}) : = 
\Big{\{} \mu \in \mathcal{D}_{1}(\mathsf{P}_{c}) 
\colon \tr(PQ) = 0, \,  \ \textrm{for all } P\neq Q \in \supp(\mu)  \Big{\}},
\end{equation}
for the space of all the measures supported on orthogonal collections of projections.
Of course the defining condition for $\mathcal{D}_1(\mathsf{P}_c)$ 
means $\sum_i \lambda_i \operatorname{tr}(P_i)=1$.

We are then ready to define the following injection: 
\begin{equation}\label{Phidensity}
\Phi : \mathcal{C}(\mathsf{H}) \longrightarrow \mathcal{D}_{1}(\mathsf{P}_{c}) \subset \mathcal{D}(\mathsf{P}_{c}), 
\qquad
\Phi(\rho) = \Phi\left(\sum_{i} \lambda_{i} P_{V_{i}}\right) : =  \sum_{i} \lambda_{i} \delta_{P_{V_{i}}}.
\end{equation}
For consistency, we will also denote 
$\Phi(\rho)$ by $\mu_{\rho}$. Notice that $\tr(\cdot)\mu_{\rho}(\mathsf{P}_{c}) = 1$ follows from 
$\tr(\rho) = 1$.
The spectral Theorem implies that $\Phi(\rho_{\varphi}) \subset \mathcal{D}_1^{\bot}(\mathsf{P}_c)$
for every $\rho_{\varphi}$ (pairwise orthogonal projections). Moreover 
as no repetition of eigenvalues 
is present in $\Phi(\varphi)$:
$$
\Phi(\varphi)(P) \neq \Phi(\varphi)(Q), \qquad \textrm{for all } P\neq Q \in \supp\big{(}\Phi(\varphi)\big{)}.
$$
This property actually characterizes the image $\Phi(\mathcal{C}(\mathsf{H})).$

\begin{definition}\label{D:Phimap}
Using the isomorphism between $\mathcal{S}_{n}(\mathsf{B}(\mathsf{H}))$ and 
$\mathcal{C}(\mathsf{H})$, the map $\Phi$ given by
\begin{equation}\label{E:Phistates}
\Phi : \mathcal{S}_{n}(\mathsf{B}(\mathsf{H})) \longrightarrow 
\mathcal{D}_{1}(\mathsf{P}_{c}), \qquad
\Phi(\varphi) : = \sum_{i} \lambda_{i} \delta_{P_{V_{i}}},
\end{equation}
is well defined (with a slight abuse of notation).
The notation $\mu_{\varphi}$ in place of $\Phi(\varphi)$ 
will sometimes be preferred.
\end{definition}

\begin{remark}\label{R:support}
The support of $\mu_{\varphi}$ is $\{ P_{V_{i}} \colon i \in \N \} \subset \mathsf{P}_{c}$, 
a totally disconnected set;
notice indeed that by orthogonality of the eigenspaces, 
 $\| P_{V_{i}} - P_{V_{j}}\|_{\mathsf{B}(\mathsf{H})} = 1$ whenever 
$i  \neq j$. Hence $\{ P_{V_{i}} \colon i \in \N \}$ is discrete and then closed.  
Notice also that the projection onto the 
possibly infinite dimensional subspace $N(\rho_{\varphi})$ does not belong to 
$\supp(\mu_{\varphi})$. 
\end{remark}

\smallskip
We define now the converse correspondence.

\begin{definition}\label{D:mapPsi}
To each element of $\mathcal{D}_{1}(\mathsf{P}_{c})$  
we associate a density matrix in the following form:
\begin{equation}\label{E:Psiinverse}
\Psi : \mathcal{D}_{1}(\mathsf{P}_{c}) \longrightarrow \mathcal{C}(\mathsf{H}), \quad 
\Psi(\mu) = \Psi\left(\sum_{i}\lambda_{i} \delta_{P_{i}}\right) :  = 
\sum_{i}\lambda_{i} P_{i}.
\end{equation}
Notice indeed  $\rho = \sum_{i}\lambda_{i} P_{i}$ converges in the trace norm to a well defined symmetric operator having $\tr(\rho) = \int \tr(P)\,\mu(dP)= 1$; hence
$\Psi(\mathcal{D}_{1}(\mathsf{P}_{c})) \subset \mathcal{C}(\mathsf{H})$.
\end{definition}

By the spectral Theorem again one notices that
$$
\Psi(\Phi (\rho )) = \rho.
$$
Hence $\Psi$ is the left-inverse of $\Phi$ while, in general,
it fails to satisfy $\Phi(\Psi (\mu)) = \mu$. 
Particularly relevant for us will be the sets 
\begin{equation}\label{E:allmeasures}
\Lambda_{\varphi}^{\perp} := 
\Psi^{-1}(\varphi) \cap \mathcal{D}^{\perp}_{1}(\mathsf{P}_{c}) 
= \Big{\{} \mu \in \mathcal{D}^{\perp}_{1}(\mathsf{P}_{c}) \colon \Psi(\mu) = \rho_{\varphi}\Big{\}},
\end{equation}
the set of the measures concentrated on pairwise orthogonal projections whose 
corresponding symmetric operator is the density matrix of $\varphi$.
In particular any element of $\Lambda_{\varphi}^{\perp}$
represents a spectral decomposition of $\rho_{\varphi}$ 
admitting repeated eigenvalues.\\

Coming to the topological properties of these sets, 
we recall that a sequence of probability measures 
$\mu_{n} \in \mathcal{P}(\mathsf{P}_{c})$ is said to weakly 
converge to $\mu \in \mathcal{P}(\mathsf{P}_{c})$ if 
by definition 
$$
\int_{\mathsf{P}_{c}} f(P)\,\mu_{n}(dP) \longrightarrow 
\int_{\mathsf{P}_{c}} f(P)\,\mu(dP), \qquad \forall \ f \in 
C_{b}(\mathsf{P}_{c}).
$$
It is well-known that, for Polish spaces, 
the L\'evy-Prokhorov metric gives a 
metrization of weak convergence; in particular 
it makes $\mathcal{P}(\mathsf{P}_{c})$ 
complete and separable. 
It will be therefore enough to describe topologically 
the subsets of $\mathcal{P}(\mathsf{P}_{c})$ only using 
weakly converging sequences.

Moreover we recall the following classical fact about compact subsets of
probability measures:
if $(\mathcal{X},\mathsf{d})$ is a metric space 
(considered with its Borel $\sigma$--algebra), a set 
$\mathcal{S}\subset \mathcal{P}(\mathcal{X})$ 
of probability measures is {\em{tight}}
whether for every 
$\ve>0$ there is a compact $K_{\ve}\subset \mathcal{X}$ 
such that 
$\mu(K_{\ve}) \geq 1-\ve$ for every 
$\mu\in \mathcal{S}$. 
The Prohorov Theorem states that every tight family is relatively compact. If $\mathcal{X}$ is Polish the converse is true: 
every relatively compact family is tight.

\begin{lemma}\label{L:continuityPsi}
The map $\Psi : \mathcal{D}_{1}(\mathsf{P}_{c}) \longrightarrow 
\mathcal{S}_{n}(\mathsf{B(H)})$ is continuous in the following sense: if $\tr(\cdot) \mu_{n} \weak \tr(\cdot) \mu$, then 
$\Psi(\mu_{n}) \weak \Psi(\mu)$.
\end{lemma}

\begin{proof}
For each $B \in \mathcal{L}^{1}$ we consider the function
$$
f_{B} : \mathsf{P}_{c} \longrightarrow \R, \quad f_{B}(P) = \tr(BP)/\tr(P)
$$
and zero on $0\in \mathsf{P}_c$.
The function is easily seen to be continuous on the same connected component of $\mathsf{P}_{c}$ 
and it is bounded by
$|f_{B}(P)| = \| B\|$.
Then the following identities
$$
\int_{\mathsf{P}_{c}} f_{B}(P)\tr(P) \,\mu_{n}(dP) = 
\int_{\mathsf{P}_{c}} \tr(BP) \,\mu_{n}(dP) = 
 \Psi(\mu_{n}) (B),
$$
imply that $\Psi(\mu_{n}) (B) \to \Psi(\mu)(B)$, for all 
$B \in \mathcal{L}^{1}$. 
By density in the norm sense, this is enough to conclude that 
$\Psi(\mu_{n}) (B) \to \Psi(\mu)(B)$, for all 
$B \in \mathbb{K}$ and the conclusion comes from Theorem 
\ref{T:convergenceweakL1}.
\end{proof}

Then we analyse topological properties of subsets of discrete measures.
In particular, the next result will be crucial in the study of the optimal transport problem between normal states.

\begin{proposition}\label{P:compactness}
The set 
$$
\tr(\cdot)\mathcal{D}_{1}^{\perp}(\mathsf{P}_{c}): = 
\Big{\{}\tr(\cdot) \mu \colon \mu \in 
\mathcal{D}_{1}^{\perp}(\mathsf{P}_{c})  \Big{\}}
$$ 
is closed.
Moreover for any $\varphi \in \mathcal{S}_{n}(\mathsf{B(H)})$,  
$\tr(\cdot)\Lambda_{\varphi}^{\perp}$ is compact.
\end{proposition}

\begin{proof}
{\bf Step 1.}
Consider a sequence $\mu_{n} \in \mathcal{D}_{1}^{\perp}(\mathsf{P}_{c})$ and $\eta \in \mathcal{P}(\mathsf{P}_{c})$
such that $\tr(\cdot)\mu_{n} \weak \eta$.   
Then for any $P \in \supp(\eta)$ there exists a sub-sequence $n_{k}$ and $P_{k} \in \supp(\mu_{n_{k}})$ 
such that $P_{k} \to P$. Any two distinct 
projections $P,Q \in \supp(\mu_{n})$ verify $\|P - Q \| = 1$, 
then necessarily $\eta$ is a discrete measure, i.e. 
$\eta = \sum_{i} \lambda_{i} \delta_{P_{i}}$. 
For the same reason, $\tr(P_{i}P_{j}) = 0$ whenever $i \neq j$.
Since by assumption $\eta(\mathsf{P}_{c}) = 1$, to have the claim 
is enough to define $\mu: = \eta/\tr(\cdot)$ to have that 
$\eta \in \tr(\cdot)\mathcal{D}_{1}^{\perp}(\mathsf{P}_{c})$.

\smallskip
{\bf Step 2.}
We fix the following notation: 
$\Phi(\varphi) = \sum_{i} \lambda_{i} P_{V_{i}}$, where 
$P_{V_{i}}$ denotes the  
projection onto $V_{i}$.
Given any $\ve>0$ there exists $m_{\ve} \in \N$
such that 
$$
\sum_{i \geq m_{\ve} } \lambda_{i} \tr(P_{V_{i}}) \leq \ve, \quad 
\text{with }\quad  \lambda_{i} > \lambda_{i+1}.
$$
Let $N:=\max_{i\leq m_{\ve}}\operatorname{dim}V_i.$
For every $i\leq m_{\ve}$ we say that a decomposition of $P_{V_i}$ is a $N$-tuple $(Q_1,...,\,Q_N) \in \mathsf{P}_c^N$ 
such that the $Q_j$ that are different from zero are mutually orthogonal and satisfy 
 $\sum_{j=1}^r P_{Q_j}=P_{V_i}$.
 If we call $\mathcal{Q}_i$ the set of all such decompositions 
we have $N$ projections 
$q_j^{(i)}:\mathcal{Q}_i \longrightarrow \mathsf{P}_c$. Define
$$F_{V_i}:=\bigcup_{j=1}^Nq_j^{(i)}\big{(}\mathcal{Q}_i\big{)}$$ the set of all the projections appearing in at least one decomposition of $V_i$.
Let $\mathbb{G}({V_i},d)$ be the Grassmann manifold of all the subspaces of dimension $d$ inside $V_i$; we can embed
$\mathcal{Q}_i$
  into the union of all products $\mathbb{G}\big{(}V_i, \operatorname{tr}(Q_1)\big{)} \times \cdots \times \mathbb{G}\big{(}V_i, \operatorname{tr}(Q_N)\big{)}$ with the union running over the finite set of all the possible ways of writing $N=s_1 + \cdots +s_N$ with $s_j \in \mathbb{N}$ (including zero). We adopt the convention that $\mathbb{G}\big{(}V_i,0\big{)}=\bullet$, the space with a point.
Now since the Grassmannians are compact we see that $F_{V_i}$ and also $\bigcup_{i=1}^{m_{\ve}}F_{V_i}$ are relatively compact inside $\mathsf{P}_c$.

Now pick any $\mu \in \Lambda_{\varphi}^{\perp}$. We write it in the form $\mu = \sum_{i} \lambda_{i} \delta_{P_{Z_{i}}}$ 
where the eigenvalues are the same as the eigenvalues of $\Phi(\varphi)$ but here now they may be repeated.
It holds true, 
$$
\tr(\cdot )\mu \bigg{(}\bigcup_{i=1} ^{m_{\ve}} F_{V_{i}}\bigg{)} = 
\sum_{i\leq m_{\ve}} \lambda_{i} \tr(P_{Z_{i}}) 
= \sum_{i\leq m_{\ve}} \lambda_{i} \tr(P_{V_{i}}) \geq 1-\ve,
$$
where the second identity is valid collecting different projections with the same eigenvalue. 
This proves that $\tr(\cdot)\Lambda_{\varphi}^{\perp}$ is tight.
To prove compactness is enough to recall that tightness 
is equivalent to precompactness in $\mathcal{P}(\mathsf{P}_{c})$.
Moreover by Lemma \ref{L:continuityPsi} and the previous 
part of the proof $\tr(\cdot)\Lambda_{\varphi}^{\perp}$ is closed; 
hence the claim follows.
\end{proof}

\subsection{Weak* topology and convergence of projections} 

We now 
relate weak$^{*}$ convergence of normal states and
 spectral decomposition of the associated density matrices.

\begin{lemma}\label{L:convergenceprojections}
Let $(P_{n})_n$ be a sequence of   projections inside $\mathcal{L}^{1}$; 
assume moreover $P_{n} \weak P$ in the $w$-topology to some 
$P \in \mathcal{L}^{1}$, i.e. in duality with 
$\mathsf{B}(\mathsf{H})$.
Then $P$ is a projection: $P^{2} = P$ and $P^{*} = P$.
\end{lemma}

\begin{proof}
To check $P = P^{*}$ it is sufficient to notice that 
$\langle P x, x\rangle \in \R$, being the limit 
of $\langle P_{n} x, x\rangle \in \R$.

To prove $P^{2} = P$, first we can assume that $P \neq 0$ otherwise the claim is trivial.
From $P_{n} \weak P$ in the weak topology 
we deduce that 
$$
\| P_{n}\|_{\mathcal{L}^{1}}= \tr(P_{n}) \longrightarrow \tr(P) = \| P\|_{\mathcal{L}^{1}}, 
$$
implying (see Theorem \ref{T:convergenceweakL1}) that $P_{n} \to P$
in $\mathcal{L}^{1}$. Then for any $K \in \mathbb{K}$ %
$$
\tr(P_{n} K P_{n}) = \tr(P^{2}_{n} K )  = \tr(P_{n}K ) \longrightarrow \tr(P K);
$$
on the other hand, since $K P_{n} \to KP$ in $\mathcal{L}^{1}$, it follows that 
$\tr(P_{n} K P_{n}) \to \tr(PKP)$. Therefore $P^{2} = P$.
\end{proof}

\begin{lemma}\label{L:Pconverge}
Let $\varphi_{n},\varphi \in \mathcal{S}_{n}(\mathsf{B(H)})$
such that $\varphi_{n} \weak \varphi$.
Consider the corresponding density matrices $\rho_{n} = \rho_{\varphi_{n}}$, 
 $\rho = \rho_{\varphi}$ 
for which we consider any spectral decompositions
(in the sense of \eqref{E:allmeasures})
$$
\rho_{n} = \sum_{i} \lambda_{i}^{n} P_{i}^{n}, \qquad 
\rho = \sum_{i} \lambda_{i} P_{i},
$$
in particular repetitions of eigenvalues are allowed.
Let $\big{(} \lambda_{i}^{n_{k}}\big{)}_{k\in \N}$ be a 
subsequence converging to $\widetilde{ \lambda}_i \neq 0$ and  
$\widetilde{P}_i$ be any $w^{*}$-limit of the corresponding 
subsequence of projections $\big{(} P_{i}^{n_{k}} \big{)}_{k\in \N}$.

Then 
$P_{i}^{n_{k}} \to \widetilde{P_{i}}$ in $\mathcal{L}^{1}$  
and therefore $\widetilde{P}_i$ is a projection.
Moreover there exists $j \in \N$ such that 
\begin{equation}\label{E:uniqueP}
\widetilde{P}_{i} \leq P_{j}, \qquad \widetilde{\lambda_{i}} = \lambda_{j}.
\end{equation}
\end{lemma}

\begin{proof}
We start noticing the following: for any 
$B \in \mathsf{B}(\mathsf{H})$ it holds true
$
\tr\big{(}\rho_{n_{k}}P_{i}^{n_{k}}B\big{)} \longrightarrow 
\tr\big{(}\rho\widetilde{P}_{i}B\big{)}. 
$
Indeed 
$$
\tr\big{(}\rho_{n_{k}}P_{i}^{n_{k}}B\big{)} - 
\tr\big{(}\rho\widetilde{P}_{i}B\big{)} 
= \tr\big{(}\rho_{n_{k}}P_{i}^{n_{k}}B\big{)} - 
\tr\big{(}\rho P_{i}^{n_{k}}B\big{)} 
+\tr\big{(}\rho P_{i}^{n_{k}}B\big{)} - 
\tr\big{(}\rho\widetilde{P}_{i}B\big{)}; 
$$
then the first term goes to zero from $\rho_{n_{k}} \longrightarrow \rho$
in $\mathcal{L}^{1}$ while the second one converges to zero from the $w^{*}$-convergence of $P_{i}^{n_{k}}$ to 
$\widetilde{P}_{i}$ and by the compactness of $\rho$.
Moreover by the orthogonality of projections it follows that 
$$
\rho_{n_{k}}P_{i}^{n_{k}} = 
\lambda_{i}^{n_{k}}P_{i}^{n_{k}};
$$
hence by $\lambda_{i}^{n_{k}} \to 
\widetilde{\lambda}_{i} \neq 0$,   
$P_{i}^{n_{k}}$ is $w$-converging to 
$\rho \widetilde{P}_{i}/\widetilde {\lambda}_{i}$. 

Then, since $w$ and $w^{*}$ limits coincide, we deduce that 
$\widetilde{P}_{i} = \rho \widetilde{P}_{i}/\widetilde{\lambda}_{i}$ and that 
$P_{i}^{n_{k}} \to \widetilde{P}_{i}$ weakly.
Therefore following the proof of Lemma \ref{L:convergenceprojections},
$P_{i}^{n_{k}} \to \widetilde{P}_{i}$ in $\mathcal{L}^{1}$ 
and $\widetilde{P}_{i}$ is a  projection, proving the 
first part of the claim.

To obtain the second part we observe that the previous identity 
$\rho \widetilde{P}_{i} = \widetilde {\lambda}_{i} \widetilde{P}_{i}$ implies the claim 
together with  the uniqueness of the spectral decomposition of the compact and self-adjoint operator $\rho$.
\end{proof}

\begin{proposition}\label{P:weakweak}
Let $\varphi_{n},\varphi \in \mathcal{S}_{n}(\mathsf{B(H)})$
such that $\varphi_{n} \weak \varphi$. Then for any 
sequence $\mu_{n} \in \Lambda_{\varphi_{n}}^{\perp}$ 
there exist a subsequence $\mu_{n_{k}}$ and  
$\mu \in \Lambda_{\varphi}^{\perp}$ such that 
$\tr(\cdot) \mu_{n_{k}} \weak \tr(\cdot) \mu$, i.e.
in duality 
with continuous and bounded functions $C_{b}(\mathsf{P}_{c})$.
\end{proposition}

\begin{proof}
{\bf Step 1.} Consider the sequences 
$( \lambda_{i}^{n})_{i\in \N}$ and  $(\lambda_{i})_{i\in \N}$
of eigenvalues of $\rho_{n}$ and $\rho$, respectively, 
arranged in decreasing order and repeated according to the multiplicity; in particular both sequences have norm $1$ in $\ell^{1}$. 
Then \cite[Theorem 2]{Robinson} proves that 
$$
\sum_{i} |\lambda_{i}^{n} - \lambda_{i}| \leq \tr|\rho_{n}-\rho|,
$$
giving that 
$( \lambda_{i}^{n})_{i\in \N} \longrightarrow ( \lambda_{i})_{i\in \N}$
in the $\ell^{1}$-norm as $n\to \infty$. 
As a straightforward consequence for each 
$\varepsilon$ there exist $n_{\ve}, M \in \N$ such that 
\begin{equation}\label{E:uniformsmall}
\sum_{i \geq M} \lambda_{i}^{n} \leq \ve, \quad \forall \ n \geq n_{\ve}, \qquad
\sum_{i \geq M} \lambda_{i} \leq \ve.
\end{equation}
It is not restrictive to assume $\lambda_{i} > 0$ 
for each $i< M$ for if this is not the case we can simply lower $M$ without changing the validity of \eqref{E:uniformsmall};  then the $\ell^{1}$-convergence will imply 
\eqref{E:uniformsmall} for $(\lambda_{i}^{n})_i$ as well.

\smallskip
\noindent {\bf Step 2.}
Consider now any sequence 
$\mu_{n} \in \Lambda_{\varphi_{n}}^{\perp}$. To fix the notations we write 
$$
\mu_{n} = \sum_{j} \lambda_{i_{j}}^{n} \delta_{P_{j}^{n}}, 
\qquad \rho_{n} = \sum_{j} \lambda_{i_{j}}^{n} P_{j}^{n}.
$$
Then we proceed as follows: denote with $m\in \N$ the first 
number such that $\lambda_{m} = 0$, with 
$(\lambda_{i})_{i\in \N}$ seen as an element of $\ell^{1}(\N)$;
in particular $m \geq M$. 
If $\lambda_{i} > 0$ for all $i \in \N$, we pose $m = \infty$.

From $\ell^{1}$-convergence we have  
$\lambda_{i}^{n} \longrightarrow \lambda_{i}$ as $n\to \infty$.
Hence for each $j \in N$ such that $i_{j}< m$, the sequence 
of projections $\big{(} P_{j}^{n}\big{)}_{n\in \N}$ has trace uniformly bounded; hence $w^{*}$-precompactness and 
Lemma \ref{L:Pconverge} imply the existence of a subsequence 
$n_{k}$ and of a projection $P_{j}$ such that 
$$
P_{j}^{n_{k}} \longrightarrow P_{j},\quad \rho P_{j} = \lambda_{j} P_{j}.
$$
Via the usual diagonal argument,
we deduce the existence of a subsequence, 
still denoted by $n_{k}$, such that for each 
$j \in \N$ such that $i_{j}< m$ 
$$
P_{j}^{n_{k}} \longrightarrow P_{j},\quad \rho P_{j} = \lambda_{j} P_{j},
$$
as $k \to \infty$. We define then 
$\mu : = \sum_{j} \lambda_{j} \delta_{P_{j}}$.
By the norm convergence, if $j_{1}\neq j_{2}$ then
$$
\tr(P_{j_{1}}P_{j_{2}}) = 0, 
$$
and by $\rho P_{j} = \lambda_{j} P_{j}$ it follows that 
$\rho \geq \sum_{j} \lambda_{j} P_{j}$.
Moreover, since 
$$
\sum_{j \colon i_{j}\geq M }\lambda_{i_{j}}^{n_{k}} 
\tr(P_{j}^{n_{k}}) = \sum_{i \geq M} \lambda_{i}^{n} \leq \ve,
$$
it follows that 
$$
\sum_{j} \lambda_{j} \tr(P_{j}) \geq 
\limsup_{n} \sum_{j \colon i_{j} \leq M} 
\lambda_{i_{j}}^{n} \tr(P_{j}^{n_{k}}) \geq 1 -\ve.
$$
Since $\ve$ was arbitrarily chosen and did not play any role in the construction of $\mu$, it follows that  
$\sum_{j} \lambda_{j} \tr(P_{j}) = 1$, giving, 
$\mu \in \Lambda_{\varphi}^{\perp}$.
As byproduct we have also shown that the sequence
$\big{(}\lambda_{i_{j}^{n_{k}}} \tr(P_{j}^{n_{k}}) \big{)}$ 
converges to $\big{(}\lambda_{j} \tr(P_{j})\big{)}$  in $\ell^{1}(\N)$.

\smallskip
\noindent {\bf Step 3.}
The claim is now equivalent to proving that 
for any $f \in C_{b}(\mathsf{P}_{c})$
$$
\lim_{k\to \infty} 
\sum_{j} \lambda_{i_{j}}^{n_{k}}\tr(P_{j}^{n_{k}})f(P_{j}^{n_{k}})
= \sum_{j} \lambda_{j}\tr(P_{j})f(P_{j}). 
$$
This now follows from the $\ell^{1}(\N)$-convergence of
$\big{(}\lambda_{i_{j}^{n_{k}}} \tr(P_{j}^{n_{k}}) \big{)}$ to
$\big{(}\lambda_{j} \tr(P_{j})\big{)}$ and the norm convergence of each 
$P_{j}^{n_{k}}$ to $P_{j}$ (implying convergence in $\mathsf{d}$) coupled with continuity and boundedness of $f$.
\end{proof}

We summarise the results in the next
statement whose proof will be an easy consequence of previous convergence results.

\begin{theorem}\label{T:summary}
Let $\varphi_{n}, \varphi \in \mathcal{S}_{n}(\mathsf{B}(\mathsf{H}))$ be normal
states and consider 
$\mu_{n}\in \Lambda_{\varphi_{n}}^{\perp}, \mu \in \Lambda_{\varphi}^{\perp}$.
Then 
\begin{itemize}
\item[1.] If $\tr(\cdot) \, \mu_{n} \weak \tr(\cdot)\, \mu$ in duality with 
$C_{b}(\mathsf{P}_{c})$, then $\varphi_{n}\weak \varphi$ in the $w^{*}$-sense.
\item[2.] If $\varphi_{n}\weak \varphi$ in the $w^{*}$-sense
then there exist a subsequence $(n_{k})_{k\in\N}$ and 
$\bar \mu \in \Lambda_{\varphi}^{\perp}$ such that 
$\tr(\cdot) \, \mu_{n_{k}} \to \tr(\cdot)\bar\mu$ in duality with $C_{b}(\mathsf{P}_{c})$.
\end{itemize}
\end{theorem}

\begin{proof}
The first point is Lemma \ref{L:continuityPsi}
while the second part of the claim is 
precisely Proposition \ref{P:weakweak}. 
\end{proof}

\bigskip
\section{Wasserstein distance between normal states}\label{S:Wasserstein}

We will use the metric structure of $(\mathsf{P}_c,\mathsf{d})$ reviewed in 
Section \ref{Ss:project}, together with the map $\Psi$ 
(Definition \ref{D:mapPsi}) 
 to define a \emph{static} Wasserstein distance between 
normal states of $\mathsf{B}(\mathsf{H})$. 
The classical definition of $p$-Wasserstein distance 
over $(\mathsf{P}_c,\mathsf{d})$ (being $\sfd$ an extended metric does not hurt the definition), the plan is to push it to normal states via $\Psi$. 

\noindent With this motivation in mind, we begin describing in details $W^{\mathsf{P}_{c}}_{p}$.
\smallskip

\subsection{Wasserstein distance over $\mathcal{P}(\mathsf{P}_{c})$}

In the classical setting optimal transportation is encoded in transport plans,
i.e. probability measures over the product space with assigned marginals.
As the metric $\sfd$ is finite solely when restricted on each connected component of 
$\mathsf{P}_{c}$, 
we will consider a more stringent notion of transport plan
(recall that $P,Q\in \mathsf{P}_{c}$ belong to the same connected component if and only if
$\operatorname{dim} R(P)=\operatorname{dim} R(Q)$).

\begin{definition}\label{D:transportplan}
Given two probability measures $\mu_{0},\mu_{1} \in \mathcal{P}(\mathsf{P}_{c})$, 
the set of $\sfd$-transport plans between $\mu_{0}$ and $\mu_{1}$ 
will be given by
\begin{equation}\label{E:FPi}
\Pi_{\sfd}(\mu_{0}, \mu_{1}) : = 
\Big\{ \nu \in \Pi(\mu_{0}, \mu_{1}) \colon 
\dim(R(P)) = \dim(R(Q)), \ \nu-a.e.\Big\},
\end{equation}
where $\Pi(\mu_{0}, \mu_{1}) = 
\{ \nu \in \mathcal{P}(\mathsf{P}_{c} \times \mathsf{P}_{c}) \colon 
(\pi_{1})_{\sharp} \nu = \mu_{0},\ (\pi_{2})_{\sharp} \nu = \mu_{1} \}$ is the classical notation for transport plans and 
$\pi_{i} : \mathsf{P}_{c} \times \mathsf{P}_{c} \to \mathsf{P}_{c}$ is the projection 
on the $i$-th component, for $i = 1,2$.
\end{definition}

The set of $\sfd$-transport plans $\Pi_{\sfd}(\mu_{0}, \mu_{1})$ 
is a, possibly empty, convex subset of $\mathcal{P}(\mathsf{P}_{c} \times \mathsf{P}_{c})$.
Then we will define the following optimal transport distance.

\begin{definition}\label{D:Wassersteindistance} 
Given $\mu_{0},\mu_{1} \in \mathcal{P}(\mathsf{P}_{c})$, 
for any $p\geq 1$ we define their $W_{p}^{\mathsf{P}_{c}}$ distance as follows:
\begin{equation}\label{E:OTstates}
W_{p}^{\mathsf{P}_{c}}(\mu_{0},\mu_{1}) : = 
\inf_{\nu \in \Pi_{\sfd}(\mu_{0},\mu_{1})}
\left( \int_{\mathsf{P}_{c}\times \mathsf{P}_{c}} \sfd(P,Q)^{p}\, \nu(dPdQ),
\right)^{\frac{1}{p}}
\end{equation}
where $\sfd$ is the extended geodesic distance of $\mathsf{P}_{c}$.
Whenever the set $\Pi_{\sfd}((\mu_{0},\mu_{1}))$ is empty we pose 
$W_{p}^{\mathsf{P}_{c}}(\mu_{0},\mu_{1}) := +\infty$.
\end{definition}

It is fairly easy (and almost identical to the classical case) 
to prove existence of optimal transport plans.

\begin{theorem}[Existence of optimal plans]\label{T:existence}
Given $\mu_{0},\mu_{1} \in \mathcal{P}(\mathsf{P}_{c})$, 
there exists an optimal plan $\nu \in \Pi_{\sfd}(\mu_{0},\mu_{1})$ such that 
$$
W_{p}^{\mathsf{P}_{c}}(\mu_{0},\mu_{1})^{p} = \int_{\mathsf{P}_{c}\times \mathsf{P}_{c}} 
\sfd(P,Q)^{p}\, \nu(dPdQ),
$$
provided the set of admissible plan $\Pi_{\sfd}(\mu_{0},\mu_{1})$ is not empty.
\end{theorem}

\begin{proof}
Since $\Pi_{\sfd}(\mu_{0},\mu_{1}) \neq \emptyset$ and $\sfd \leq \frac{\pi}{2}$,
there exists a minimizing sequence 
$\nu_{n} \in \Pi_{\sfd}(\mu_{0},\mu_{1})$ such that 
$$
\lim_{n} \int_{\mathsf{P}_{c}\times \mathsf{P}_{c}} \sfd(P,Q)^{p} \,\nu_{n}(dPdQ)= 
W_{p}^{\mathsf{P}_{c}}(\mu_{0},\mu_{1})^{p}. 
$$
Thanks to the following Lemma \ref{L:compactness}, there exist
$\nu_{n_{k}},\nu \in \Pi_{\sfd}(\mu_{0},\mu_{1})$ such that 
$\nu_{n_{k}}\weak \nu$, in duality with $C_{b}(\mathsf{P}_{c}\times \mathsf{P}_{c})$. 
Being the distance continuous and bounded it follows that 
$$
W_{p}^{\mathsf{P}_{c}}(\mu_{0},\mu_{1})^{p} = \lim_{k} \int_{\mathsf{P}_{c}\times \mathsf{P}_{c}} 
\sfd(P,Q)^{p} \, \nu_{n_{k}}(dPdQ) =
\int_{\mathsf{P}_{c}\times \mathsf{P}_{c}} \sfd(P,Q)^{p} \,\nu(dPdQ) 
$$
proving the claim.
\end{proof}

\begin{lemma}\label{L:compactness}
Given $\mu_{0},\mu_{1} \in \mathcal{P}(\mathsf{P}_{c})$,
for any sequence  $\nu_{n} \in \Pi_{\sfd}(\mu_{0},\mu_{1})$
there exist a subsequence $\nu_{n_{k}}$ and 
$\nu \in \Pi_{\sfd}(\mu_{0},\mu_{1})$ 
such that $\nu_{n_{k}} \weak \nu$ in duality with  any 
$f \in C_{b}(\mathsf{P}_{c}\times \mathsf{P}_{c})$.
\end{lemma}

Even tough the argument is similar to the classical case, for readers' convenience we include the proof.

\begin{proof}
By inner regularity of probability measures over Polish spaces, 
for any $\ve>0$
there exit compact sets $K_{1} \subset \supp(\mu_{0})$ and 
$K_{2} \subset \supp(\mu_{1})$ such that $\mu_{0}(K_{1}), \mu_{1}(K_{2}) \geq 1-\ve$ implying that $\nu_{n}(K_{1}\times K_{2}) \geq 1 - 2\ve$,
showing that 
$\nu_{n}$ is tight. 
Then Prohorov's Theorem ensures the existence of subsequence 
$\nu_{n_{k}}$ and of $\nu \in \mathcal{P}(\mathsf{P}_{c}\times\mathsf{P}_{c})$ 
such that $\nu_{n_{k}}\weak \nu$ in duality with  any 
$f \in C_{b}(\mathsf{P}_{c}\times \mathsf{P}_{c})$. 
In particular this implies that $(\pi_{1})_{\sharp} \nu = \mu_{0}$ and
$(\pi_{2})_{\sharp} \nu = \mu_{1}$, 
proving that $\nu \in \Pi(\mu_{0},\mu_{1})$.

To conclude, consider the function $f : \mathsf{P}_{c}\times \mathsf{P}_{c} \to \R$ defined by $f(P,Q) : = \dim(R(P)) - \dim(R(Q))$. The function $f$ is locally constant 
and therefore continuous. Hence the set 
$$
C: = 
\{ (P,Q) \in \mathsf{P}_{c} \times \mathsf{P}_{c} \colon f(P,Q) > 0 \}
$$
is open giving that $0=\liminf \nu_{n_{k}}(C) \geq \nu(C) \geq 0$, proving the claim.
\end{proof}

We conclude this short overview on optimal transport in 
$\mathcal{P}(\mathsf{P}_{c})$ by recalling the simple
relation between Wasserstein 
topology and weak topology. 
Here we refer to \cite[Theorem 6.9]{Vil}: 
if $\mu_{n},\mu  \in \mathcal{P}(\mathsf{P}_{k})$ 
for some $k$ independent of $n$, then 
$$
\mu_{n}\weak \mu \iff W_{p}^{\mathsf{P}_{c}}(\mu_{n},\mu) \to 0,
$$
for any $p \geq 1$. Recall indeed that 
$(\mathsf{P}_{k},\sfd)$ is a complete and separable metric 
spaces with $\sfd \leq \pi/2$.


\subsection{The Optimal Transport Cost in $\mathcal{S}_{n}(\mathsf{B}(\mathsf{H}))$}\label{Ss:cost}

We now consider the natural optimal transport problem 
between normal states.

The use of multiple representations for the density matrices, 
i.e. $\Lambda_{\varphi}^{\perp}$ 
and $\Lambda_{\psi}^{\perp}$, 
together with the many connected components $(\mathsf{P}_c,\mathsf{d})$,
motivate the following defintion.

\begin{definition}\label{D:W_{p}}
For any $\varphi,\psi \in \mathcal{S}_{n}(\mathsf{B(H)})$ and $p\geq 1$ define their optimal transport \emph{cost} by
\begin{equation}\label{E:C_{p}}
\mathcal{C}_{p}(\varphi,\psi) : 
= \inf_{\left.\begin{array}{c}\mu_{0} \in \Lambda_{\varphi}^{\perp} \\\mu_{1}\in \Lambda_{\psi}^{\perp}\end{array}\right.
} W_{p}^{\mathsf{P}_{c}}\Big{(}\tr(\cdot)\,\mu_{0}, 
\tr(\cdot)\,\mu_{1}\Big{)},
\end{equation}
where $\Lambda_{\varphi}^{\perp},\Lambda_{\psi}^{\perp} \subset \mathcal{D}^{\perp}_{1}(\mathsf{P}_{c})$ have been defined in \eqref{E:allmeasures}.
\end{definition}

\begin{remark}\label{R:distancecompact}
Clearly an alternative way of writing $\mathcal{C}_{p}$ 
is to interpret it as the distance between two disjoint 
compact sets: For any $\varphi,\psi \in \mathcal{S}_{n}(\mathsf{B(H)})$ and $p\geq 1$
\begin{equation}\label{E:distancecompactse}
\mathcal{C}_{p}(\varphi,\psi)  = 
W_{p}^{\mathsf{P}_{c}}
(\tr(\cdot) \Lambda_{\varphi}^{\perp},\tr(\cdot) \Lambda_{\psi}^{\perp}),
\end{equation}
where as usual the distance between two compact sets 
is computed taking the infimum of all possible distances.
\end{remark}

It is immediate to check that $\mathcal{C}_{p}$ is bounded. 

\begin{lemma}
Given any $\varphi,\psi \in \mathcal{S}_{n}(\mathsf{B}(\mathsf{H}))$
we have
$\mathcal{C}_{p}(\varphi,\psi) \leq \pi/2$.
\end{lemma}

\begin{proof}
It is sufficient to observe that given any 
$\varphi \in \mathcal{S}_{n}(\mathsf{B}(\mathsf{H}))$ there exists
$\mu_{0} \in 
\Lambda_{\varphi}^{\perp}$ 
such that $\supp(\mu_0) \subset \mathsf{P}_{1}$ 
(recall \eqref{E:connectedPc}).
Hence by definition 
$$
\mathcal{C}_{p}(\varphi,\psi) \leq W_{p}^{\mathsf{P}_{c}}(\mu_{0},\mu_{1}) \leq \pi/2,
$$
where $\mu_{1} \in \Lambda_{\psi}^{\perp}$ 
and $\supp(\mu_{0}),\supp(\mu_{1}) \subset \mathsf{P}_{1}$. 
The second inequality follows from $\sfd(P,Q) \leq \pi/2$ whenever $P,Q \in \mathsf{P}$
belong to the same connected component.
\end{proof}

Relying on Proposition \ref{P:dimension1},
we deduce that looking among those spectral representations 
of states using projections with one-dimensional range does not change 
the cost functional. 

\begin{proposition}\label{P:dimension1Cost}
For any $\varphi,\psi \in \mathcal{S}_{n}(\mathsf{B(H)})$, 
$$
\mathcal{C}_{p}(\varphi,\psi) = 
\inf_{\left.
\begin{array}{c}
\mu_{0} \in \Lambda_{\varphi}^{\perp}\cap \mathcal{P}(\mathsf{P}_{1}) \\
\mu_{1}\in \Lambda_{\psi}^{\perp}\cap \mathcal{P}(\mathsf{P}_{1})
\end{array}\right .} 
W_{p}^{\mathsf{P}_{c}}\Big{(}\mu_{0}, 
\mu_{1}\Big{)},
$$
\end{proposition}

\begin{proof}
Consider any $\mu_{0} \in \Lambda_{\varphi}^{\perp}$, 
$\mu_{1} \in \Lambda_{\psi}^{\perp}$ and
$\nu \in \Pi_{d}(\tr(\cdot)\mu_{0}, \tr(\cdot)\mu_{1})$ (it is not restrictive to assume the existence of at least one transport plan).

We prove the claim showing the existence of 
$\gamma \in \mathcal{P}(\mathsf{P}_{1}\times \mathsf{P}_{1})$ 
such that 
$$
\int \sfd(P,Q)^{p} \gamma(dPdQ) \leq \int \sfd(P,Q)^{p} \nu(dPdQ), 
$$
with $(\pi_{1})_{\sharp}\gamma \in \Lambda_{\varphi}^{\perp}$
and $(\pi_{1})_{\sharp}\gamma \in \Lambda_{\psi}^{\perp}$.

We proceed by writing $\nu$ as follows: if 
$\tr(\cdot)\mu_{0} = \sum_{i} \alpha_{0,i}\delta_{P_{0,i}}$ 
(and analogous one for $\tr(\cdot) \mu_{1}$), 
then
$$
\nu = \sum_{i,j} \beta_{i,j} \delta_{P_{0,i}}\otimes\delta_{P_{1,j}},
$$
for some $\beta_{i,j}\geq 0$ summing to 1. 
Whenever $\beta_{i,j}>0$ and $\tr(P_{0,i}) =r > 1$, 
we consider any orthonormal frame of $R(P_{0,i})$, say $e_{1},\dots, e_{r}$
such that $\sum_{k\leq r} P_{e_{k}} = P_{0,i}$. 

We also consider $Z \in T_{P_{0,i}}\mathsf{P}$ such that 
$P_{1,j} = e^{iZ} P_{0,i}e^{-iZ}$ and consequently define  
$P_{1,e_{k}}: = e^{iZ} P_{e_{k}}e^{-iZ}$. 
Clearly 
$$
\sum_{k \leq r}P_{1,e_{k}} = P_{1,j},
$$
and by Proposition \ref{P:dimension1} $\sfd(P_{e_{k}},P_{1,e_{k}}) \leq 
\sfd(P_{0,i}, P_{1,j})$. 
We therefore define a new transport plan $\bar \nu$ replacing
$\delta_{P_{0,i}}\otimes \delta_{P_{0,j}}$ by 
$$
\frac{1}{r}\sum_{k}\delta_{P_{e_{k}}}\otimes \delta_{P_{1,e_{k}}}.
$$
Then 
$$
\int \sfd(P,Q)^{p} \nu(dPdQ) - \int \sfd(P,Q)^{p} \bar  \nu(dPdQ)
=\beta_{i,j} \left(\sfd(P_{0,i},P_{1,j})^p - \frac{1}{r} \sfd(P_{e_{k}},P_{1,e_{k}})^p \right) \geq 0
$$
It is clear from the construction that the marginal measures of $\nu$ 
are still admissible measures for the states $\varphi$ and $\psi$. 
Repeating the argument at most countably many times proves the claim.
\end{proof}

After Proposition \ref{P:dimension1}, we therefore introduce the 
following additional notation: 
\begin{equation}\label{E:Lambdadim1}
\Lambda_{\varphi}^{\perp,1} : =  \Lambda_{\varphi}^{\perp}\cap \mathcal{P}(\mathsf{P}_{1}).
\end{equation}
Notice that $\Lambda_{\varphi}^{\perp,1}$ is closed, and therefore compact, 
like $\Lambda_{\varphi}^{\perp}$.

Next we prove that the infimum of \eqref{E:C_{p}} can be replaced by a minimum. 

\begin{proposition}\label{P:attained}
Given any $\varphi,\psi \in \mathcal{S}_{n}(\mathsf{B}(\mathsf{H}))$, 
there exist $\mu_{0}, \mu_{1}$ and $\nu$, elements 
of $\Lambda_{\varphi}^{\perp,1}, \Lambda_{\psi}^{\perp,1}$ 
and of $\Pi_{\sfd}\big{(}\mu_{0},\mu_{1}\big{)}$ respectively, such that
$$
\mathcal{C}_{p}(\varphi,\psi) 
= W_{p}^{\mathsf{P}_{c}} \big{(}\mu_{0},\mu_{1}\big{)}
= \left( \int_{\mathsf{P}_{c}\times \mathsf{P}_{c}} \sfd(P,Q)^{p} \, \nu(dPdQ) \right)^{1/p}.
$$
\end{proposition}

\begin{proof} 
The second identity is proved in Theorem \ref{T:existence}. It is enough therefore 
to show the first one.
By Proposition \ref{P:dimension1}, 
there exists two sequences 
$(\mu_{0,n})_{n\in \N}\subset \Lambda_{\varphi}^{\perp,1}, 
(\mu_{1,n})_{n\in \N}\subset \Lambda_{\psi}^{\perp,1}$ such that 
$$
\lim_{n\to\infty} W_{p}^{\mathsf{P}_{c}} \big{(}\mu_{0,n},\mu_{1,n}\big{)} = 
\mathcal{C}_{p}(\varphi,\psi).
$$
By compactness of 
$\tr(\cdot)\Lambda_{\varphi}^{\perp}$ and 
$\tr(\cdot)\Lambda_{\psi}^{\perp}$ 
(Proposition \ref{P:compactness}), 
we assume, up to subsequences that we omit, 
$\mu_{0,n} \weak \mu_{0},\
\mu_{1,n}\weak \mu_{1}$, for some 
$\mu_{0} \in \Lambda_{\varphi}^{\perp,1}$ and 
$\mu_{1} \in \Lambda_{\psi}^{\perp,1}$.

Now take $\nu_{n} \in \Pi_{\sfd}(\mu_{0,n}, \mu_{1,n})$
any optimal transport plan (Theorem \ref{T:existence}). Since its marginal are converging, by tightness,  
$\nu_{n}$ is weakly converging, up to subsequences, as well to a certain 
$\nu \in \Pi(\mu_{0},\mu_{1})$. 
Since $\sfd$ is continuous and bounded on $\mathsf{P}_{1}$:
$$
\mathcal{C}_{p}(\varphi,\psi) = \lim_{n\to\infty}W_{p}^{\mathsf{P}_{c}}(\mu_{0,n}, \mu_{1,n})^{p} 
= \int_{\mathsf{P}_{c} \times \mathsf{P}_{c}} \sfd(P,Q)^{p}\,\nu(dPdQ) 
\geq 
W_{p}^{\mathsf{P}_{c}}(\mu_{0},\mu_{1})^{p}.  
$$
Continuing the previous chain of inequalities with 
$\geq \mathcal{C}_{p}(\varphi,\psi)$ proves the claim.\end{proof}

Hence we have shown that the optimal transport problem defining 
the transport cost has always a solution.
Moreover, by the symmetry of $\sfd$, it is trivial to check that 
$\mathcal{C}_{p}(\varphi,\psi) = \mathcal{C}_{p}(\psi,\varphi)$.

Finally 
$\mathcal{C}_{p}(\varphi,\psi) = 0$ implies $\varphi = \psi$. 
Indeed by the previous Proposition \ref{P:attained}, there exist 
$\mu_{0} \in \Lambda_{\varphi}^{\perp,1}$ 
and $\mu_{1} \in \Lambda_{\psi}^{\perp,1}$ 
such that 
$W_{p}^{\mathsf{P}_{c}}\big{(}\mu_{0},\mu_{1}\big{)} = 0$. 
Hence $\mu_{0} = \mu_{1}$ and by definition 
of $\Lambda_{\varphi}^{\perp}$ and 
$\Lambda_{\psi}^{\perp}$
$$
\varphi = \Psi(\mu_{0}) = \Psi(\mu_{1}) = \psi.
$$
We have therefore that 
$\mathcal{C}_{p} : 
\mathcal{S}_{n}(\mathsf{B(H)}) \times \mathcal{S}_{n}(\mathsf{B(H)}) \longrightarrow [0,\infty)$ is a semi-distance.  
First references for semi-distances date back 
to the first half of 20th century, see for instance \cite{Wilson}. We refer however to the recent 
\cite{CJT} for a general overview on the topic.

\begin{remark}
Concerning triangular inequality for the cost $\mathcal{C}_{p}$,
using Remark \ref{R:distancecompact}, 
one can deduce the following property: 
given $\varphi,\psi,\phi \in \mathcal{S}_{n}(\mathsf{B(H)})$
$$
\mathcal{C}_{p}(\varphi,\psi) 
\leq 
\inf_{\mu_{0},\, \mu_{1}, \mu_{2}}
\Big{\{} 
W_{p}^{\mathsf{P}_{c}}\big{(}\mu_{0}, \mu_{1}\big{)} + 
W_{p}^{\mathsf{P}_{c}}\big{(}\mu_{1}, \mu_{2}\big{)} 
\Big{\}},
$$
infimum with respect to 
$
\mu_{0} \in \Lambda_{\varphi}^{\perp,1}, 
\mu_{1} \in \Lambda_{\phi}^{\perp,1}$ and $
\mu_{2} \in \Lambda_{\psi}^{\perp,1}.$

We do not present a proof of the previous inequality 
because it follows a classical 
argument (gluing) in optimal transport that will be also used in 
the proof of the following Lemma \ref{L:triangular}.
Moreover, whenever the intermediate normal state, say $\phi$, has density matrix with only simple eigenvalues 
(so that there is only one element in $\Lambda_{\phi}^{\perp,1}$), 
then again by gluing one obtains the triangular inequality: 
$$
\mathcal{C}_{p}(\varphi,\psi) \leq \mathcal{C}_{p}(\varphi,\phi)  + 
\mathcal{C}_{p}(\phi,\psi) .
$$
The proof of Lemma \ref{L:triangular} will clarify this point.
\end{remark}

We now investigate the topology induced by $\mathcal{C}_{p}$, 
starting by its converging sequences.
Notice indeed that semi-distances induce a topology whose open sets are in the form $U \subset \mathcal{S}_{n}(\mathsf{B(H)})$ for which for every $\varphi \in U$ there exists 
$r>0$ so that $B_{r}(\varphi) : = \big{\{} \psi \in \mathcal{S}_{n}(\mathsf{B(H)})\colon \mathcal{C}_{p}(\varphi,\psi) < r\big{\}} \subset U$.

\begin{theorem}\label{T:equivalentconvergence}
Let $\varphi_{n},\varphi$  be normal states of 
$\mathsf{B}(\mathsf{H})$. 
Then 
$$
\mathcal{C}_{p}(\varphi_{n},\varphi) \longrightarrow 0 \quad \iff
\quad 
\xymatrix@1{{\varphi}_{n}\ar@^{>}[r]^{w^*}&{\,}\varphi}.
$$ 
\end{theorem}

\begin{proof}
Suppose first that
$\mathcal{C}_{p}(\varphi_{n}, \varphi ) \to 0$ as $n \to \infty$.
By Proposition \ref{P:attained} 
there exist $\mu_{0,n}\in \Lambda_{\varphi_{n}}^{\perp,1}$, 
$\mu_{1,n} \in \Lambda_{\varphi}^{\perp,1}$ and 
$\nu_{n} \in \Pi_{\sfd}(\mu_{0,n},\mu_{1})$ such that
$$
\mathcal{C}_{p}(\varphi_{n},\varphi) 
= W_{p}^{\mathsf{P}_{c}} 
(\mu_{0,n}, \mu_{1,n})
= \left( \int_{\mathsf{P}_{c}\times \mathsf{P}_{c}} \sfd(P,Q)^{p} \, \nu_{n}(dPdQ) \right)^{1/p} \to 0
$$
By compactness of $\Lambda_{\varphi}^{\perp,1}$ 
in weak topology, $\mu_{1,n}$ has a converging subsequence to 
some $\mu_{1} \in \Lambda_{\varphi}^{\perp,1}$.
By Lemma \ref{L:continuityPsi}, 
$\Psi(\mu_{0,n}) \to \Psi(\mu_{1})$ in $w^{*}$-convergence.
By definition $\Psi(\mu_{0,n}) = \varphi_{n}$ and 
$\Psi(\mu_{1}) = \varphi$ giving the first claim.

Assume $\varphi_{n} \weak \varphi$ now.
By Proposition \ref{P:attained} 
there exist $\mu_{0,n}\in \Lambda_{\varphi_{n}}^{\perp,1}$, 
$\mu_{1,n} \in \Lambda_{\varphi}^{\perp,1}$ and 
$\nu_{n} \in \Pi_{\sfd}(\mu_{0,n},\mu_{1,n})$ such that
$$
\mathcal{C}_{p}\big{(}\varphi_{n},\varphi) 
= W_{p}^{\mathsf{P}_{c}} (\mu_{0,n},\mu_{1,n}\big{)}.
$$
We now invoke Proposition 
\ref{P:weakweak}: from $\varphi_{n} \weak \varphi$ we deduce the 
existence of
a subsequence $\mu_{0,n_{k}}$ and  
$\mu_{1} \in \Lambda_{\varphi}^{\perp,1}$ such that 
$\mu_{0,n_{k}} \weak  \mu_{1}$.
Then 
$\mu_{0,n_{k}} \to \mu_{1}$ also in 
Wasserstein distance over $\mathsf{P}_{1}$.
Hence 
$$
\mathcal{C}_{p}(\varphi_{n},\varphi) 
= W_{p}^{\mathsf{P}_{c}} (\mu_{0,n},\mu_{1,n})
\leq W_{p}^{\mathsf{P}_{c}}(\mu_{0,n_{k}}, \mu_{1}) \to 0,
$$
giving the claim.
\end{proof}

Theorem \ref{T:equivalentconvergence} together with 
 \cite[Theorem 4.2]{CJT} imply following 

\begin{corollary}
The topology $\tau_{\mathcal{C}_{p}}$ over the set of 
normal states coincide 
with the $w^{*}$-topology.
\end{corollary}

\begin{proof}
\cite[Theorem 4.2]{CJT} states that 
$\tau_{\mathcal{C}_{p}}$ coincide with the topology 
induced by sequences in the usual sense: 
$U$ is closed if contains limit points (w.r.t. to $\mathcal{C}_{p}$) of all converging sequences all contained inside $U$. 
Then Theorem \ref{T:equivalentconvergence} and metrizability of $w^{*}$-topology over bounded set
proves the claim.
\end{proof}

\subsection{Wasserstein distance in $\mathcal{S}_{n}(\mathsf{B}(\mathsf{H}))$}

Even though the cost functional $\mathcal{C}_{p}$ 
is fully satisfactory (see Theorem \ref{T:equivalentconvergence}), for completeness we address the issue of 
the lack of triangular inequality for 
$\mathcal{C}_{p}$. 
Using the spectral decomposition without repetitions of 
eigenvalues permits to obtain the triangular inequality. 
As a drawback this produces an extended distance (not finite).

\begin{definition}\label{D:Wasserstein-distance}
For any $\varphi,\psi \in \mathcal{S}_{n}(\mathsf{B(H)})$ and $p\geq 1$ define their $p$-Wasserstein distance by
\begin{equation}\label{E:W_{p}}
W_{p}(\varphi,\psi) 
: = 
W_{p}^{\mathsf{P}_{c}}
(\tr(\cdot)\Phi(\varphi),\tr(\cdot)\Phi(\psi)),
\end{equation}
with the map $\Phi$ defined in \eqref{E:Phistates}.
Recall that by Definition \ref{D:Wassersteindistance}, 
if no admissible transport plans exist, we assign to 
$W_{p}(\varphi,\psi)$ the value $+\infty$.
\end{definition}

We will now prove indeed that the map
$$
W_{p} : \mathcal{S}_{n}(\mathsf{B}(\mathsf{H})) \times  
\mathcal{S}_{n}(\mathsf{B}(\mathsf{H})) \longrightarrow [0,\infty]
$$
defines an extended distance over $\mathcal{S}_{n}(\mathsf{B}(\mathsf{H}))$. 
As before, the symmetry of $\sfd$ implies the symmetry of $W_{p}$
and if
$W_{p}(\varphi,\psi) = 0$, it is straightforward to check that 
$\varphi = \psi$. The triangular inequality 
is the content of the following

\begin{lemma}(Triangular inequality for $W_p$).\label{L:triangular}
Let $\varphi, \psi$ and $\phi$ be three elements of 
$\mathcal{S}_{n}(\mathsf{B}(\mathsf{H}))$. 
Then 
$$
W_{p}(\phi, \varphi) \leq W_{p}(\phi, \psi)+ W_{p}(\psi, \varphi).
$$
\end{lemma}

\begin{proof}
Consider 
$\nu_{1} \in \Pi_{\sfd}\big{(}\tr(\cdot)\, \mu_{\phi},\tr(\cdot)\, \mu_{\psi}\big{)}$ and 
$\nu_{2} \in \Pi_{\sfd}\big{(}\tr(\cdot)\,\mu_{\psi},\tr(\cdot)\,\mu_{\varphi}\big{)}$ 
optimal plan whose existence is assured by Theorem \ref{T:existence}.

If $\tr(\cdot)\, \mu_{\phi} = \sum_{i}\alpha_{i} \delta_{P_{V_{i}}}$, 
$\tr(\cdot)\, \mu_{\psi} = \sum_{i}\beta_{i} \delta_{P_{W_{i}}}$ and 
$\tr(\cdot)\,\mu_{\varphi} = \sum_{i}\gamma_{i} \delta_{P_{Z_{i}}}$
then the transport plans $\nu_{1}$ and $\nu_{2}$ can be written as
$$
\nu_{1} = \sum_{i,j} \alpha^{1}_{i,j} \delta_{P_{V_{i}}}\otimes \delta_{P_{W_{j}}},
\quad
\nu_{2} = \sum_{i,j} \alpha^{2}_{i,j} \delta_{P_{W_{i}}}\otimes \delta_{P_{Z_{j}}},
$$
with $\alpha^{1}_{i,j},\alpha^{2}_{i,j}\geq 0$ and
$\sum_{i,j} \alpha^{1}_{i,j} = \sum_{i,j} \alpha^{2}_{i,j} =1$; moreover 
marginal constraint are given in the following form
$$
\alpha_{i} = \sum_{j} \alpha^{1}_{i,j}, \quad   
\sum_{i} \alpha^{1}_{i,j} = \beta_{j} = \sum_{i}\alpha^{2}_{j,i}, \quad 
\gamma_{j} = \sum_{i} \alpha^{2}_{i,j}.
$$
Following the classical gluing procedure of transport plans, we define
$$
\Theta : = \sum_{i,j,k} \frac{\alpha^{1}_{i,j}\alpha^{2}_{j,k}}{\beta_{j}} \delta_{P_{V_{i}}}\otimes \delta_{P_{W_{j}}}\otimes \delta_{P_{Z_{k}}},
$$
and one can check that
$\nu_{1} = (\pi_{12})_{\sharp} \Theta$, $\nu_{2} = (\pi_{23})_{\sharp}\Theta$   
and $(\pi_{13})_{\sharp}\Theta \in \Pi(\tr(\cdot)\, \mu_{\phi},\tr(\cdot)\,\mu_{\varphi})$.

We also need to check that 
$(\pi_{13})_{\sharp}\Theta$ is admmissible: 
for $(\pi_{13})_{\sharp}\Theta$-a.e. $P,Q$ it holds 
$i(P,Q) = 0$ (or $\dim(R(P)) = \dim(R(Q)))$. Moreover from \cite{ASS} 
if $(P,Q)$ and $(Q, V)$ are Fredholm pairs, 
and either $Q - V$ or $P-Q$ is compact, then $(P, V)$ is a Fredholm pair and
$$
i(P,Q) = i(P,V) + i(V,Q).
$$
For $\Theta$-a.e. $(P,V,Q) \in \mathsf{P}_{c} \times \mathsf{P}_{c}\times \mathsf{P}_{c}$, 
we have that
$$
i(P,Q) = i(P,V) + i(V,Q)=0, \qquad \Theta-a.e. 
$$
showing that $(\pi_{13})_{\sharp}\Theta \in 
\Pi_{\sfd}(\tr(\cdot)\,\mu_{\phi},\tr(\cdot)\,\mu_{\varphi})$. 
For the same reason, $\Theta$-a.e. the projections $P,Q$ and $V$ belong to the same 
connected component of $\mathsf{P}_{c}$ where 
triangular inequalities can be used.
Hence for any $p\geq 1$: 
\begin{align*}
W_{p}(\phi,\varphi) 
&~ \leq \left(\int \sfd(P,Q)^{p} \, (\pi_{13})_{\sharp}\Theta(dPdQ)\right)^{1/p} \\
&~ = \left(\int \sfd(P,Q)^{p} \, \Theta(dPdVdQ)\right)^{1/p} \\
&~ \leq \left(\int \left(\sfd(P,V) + \sfd(V,Q)\right)^{p} \, \Theta(dPdVdQ)\right)^{1/p} \\
&~ \leq \left(\int \sfd(P,V)^{p} \Theta(dPdVdQ)\right)^{1/p}
+ \left(\int \sfd(V,Q)^{p} \, \Theta(dPdVdQ)\right)^{1/p} \\
&~ = W_{p}(\phi,\psi) + W_{p}(\psi,\varphi),
\end{align*}
concluding the proof.
\end{proof}

\begin{corollary}
The couple $(\mathcal{S}_{n}(\mathsf{B}(\mathsf{H})), W_{p})$ is 
an extended metric space in the sense that   
$$
W_{p} : \mathcal{S}_{n}(\mathsf{B}(\mathsf{H})) \times \mathcal{S}_{n}(\mathsf{B}(\mathsf{H})) \to [0,\infty]
$$ 
verifies  for any $\varphi$ and $\psi$ the following 
properties:
$W_{p}(\varphi,\varphi) = 0$, and 
if $W_{p}(\varphi,\psi) = 0$ then $\varphi = \psi$; 
$W_{p}(\varphi,\psi) = W_{p}(\psi,\varphi)$ and the triangular inequality holds true.
\end{corollary}

By definition is straightforward to check that 
$$
\mathcal{C}_{p}(\varphi,\psi) \leq W_{p}(\varphi,\psi).
$$
In particular $W_{p}$-convergence implies 
$\mathcal{C}_{p}$-convergence and, by Theorem 
\ref{T:equivalentconvergence},
$w^{*}$-convergence.
However, as expected  $w^{*}$-convergence does not imply 
$W_{p}$-convergence.
We have a simple counterexample.
\begin{example}
Consider the case of $\mathsf{H} = \C^{2}$
and 
$$
\varphi_{n} : = \left(\frac{1}{2} - \frac{1}{n}\right)| e_{1} \rangle \langle e_{1} | 
+\left(\frac{1}{2} + \frac{1}{n}\right)| e_{2} \rangle \langle e_{2} | 
\weak \varphi := \frac{1}{2} \left(|e_{1} \rangle \langle e_{1} | + 
|e_{2} \rangle \langle e_{2} | \right);
$$ 
the corresponding measures over the space of projections of $\C^{2}$ will be 
$$
\mu_{\varphi_{n}} = \left(\frac{1}{2} - \frac{1}{n}\right) \delta_{P_{1}}
+ \left(\frac{1}{2} + \frac{1}{n}\right) \delta_{P_{2}}, 
\qquad 
\mu_{\varphi} = \delta_{\operatorname{Id}},
$$
where $P_{1}$ and $P_{2}$ are the projections over the span of $e_{1}$ and 
$e_{2}$, respectively. Since $P_{1},P_{2}$ and $\operatorname{Id}$ belong to 
two different connected components of $\mathsf{P}_{c}$,
$W_{p}\big{(}\tr(\cdot)\mu_{\varphi_{n}},\tr(\cdot)\mu_{\varphi}\big{)}= \infty$.
\end{example}

\smallskip

\section{Kantorovich duality for $W_{p}$ and consequences}\label{S:Dual}

In this part we will go through the Kantorovich duality for the optimal transport problem over 
$(\mathsf{P}_{c},\sfd)$. In particular we will 
analyse cyclically montone sets and solutions of the dual problem. 
The duality will always be referred to the 
Wasserstein distance $W_{p}$.

\smallskip

\subsection{Kantorovich duality}\label{Ss:Kantor}

As before, when dealing with optimal transport arguments, we will repeatedly restrict $\sfd$ to each 
connected component $\mathsf{P}_{n}$ of $\mathsf{P}_{c}$ and invoke the classical results.
We start recalling the following classical definition from the theory of optimal transport:
A subset $\Gamma$ of $\mathsf{P}_{c}\times \mathsf{P}_{c}$ is 
$\sfd^{p}$-cyclically monotone if and only if for any $n \in \N$
and $(P_{1},Q_{1}),\dots (P_{n},Q_{n}) \in \Gamma$ the following inequality is valid
$$
\sum_{i \leq n} \sfd(P_{i},Q_{i})^{p} \leq \sum_{i \leq n} \sfd(P_{i},Q_{i+1})^{p}, 
$$
with the convention $Q_{n+1} = Q_{1}$. It is also tacitly assumed that for 
each $(P,Q) \in \Gamma$, $\operatorname{dim}R(P)=\operatorname{dim}R(Q)$.
Accordingly, given $\varphi,\psi \in \mathcal{S}_{n}(\mathsf{B}(\mathsf{H}))$,  
$\nu \in \Pi_{\sfd}(\tr(\cdot)\,\mu_{\varphi},\tr(\cdot)\,\mu_{\psi})$ will be called 
$\sfd^{p}$-cyclically monotone if there exists 
a $\sfd^{p}$-cyclically monotone set $\Gamma$ such that $\pi(\Gamma) = 1$.

By lower semicontinuity of $\sfd$, it is well-known that
$\sfd^{p}$-cyclical monotonicity is a necessary condition for being optimal 
(see for instance \cite[Proposition B.16]{biacarave}).

\begin{proposition}
Let $\varphi,\psi \in \mathcal{S}_{n}(\mathsf{B}(\mathsf{H}))$ be given and $p \geq 1$. 
Then any optimal transport plan $\nu \in \Pi_{\sfd}(\tr(\cdot)\,\mu_{\phi},\tr(\cdot)\,\mu_{\psi})$ 
for the $W_{p}$ distance is $\sfd^{p}$-cyclically monotone.
\end{proposition}

Looking at the transport on each single connected component of $\mathsf{P}_{c}$,
it is clear that cyclical monotonicity is indeed a sufficient condition for 
global optimality. 

\begin{proposition}\label{P:cyclicsufficient}
Let $\varphi,\psi \in \mathcal{S}_{n}(\mathsf{B}(\mathsf{H}))$ be given and 
let $\nu \in \Pi_{\sfd}(\tr(\cdot)\,\mu_{\varphi},\tr(\cdot)\,\mu_{\psi})$ 
be any $\sfd^{p}$-cyclically monotone transport plan.
Then $\nu$ is $W_{p}$-optimal, i.e. 
$$
\int \sfd(P,Q)^{p}\,\nu(dPdQ) = W_{p}(\varphi,\psi)^{p}.
$$
\end{proposition}

\begin{proof}
Decompose both $\mu_{\varphi}$ and $\mu_{\psi}$ into as sum of their restriction 
to each connected component of $\mathsf{P}_{c}$, 
$\mathsf{P}_{n} = \{ P \in \mathsf{P}_{c} \colon \tr(P) = n \}$.
Then 
$$
\mu_{\varphi} = \sum_{n} \mu_{\varphi,n},\quad \mu_{\psi} = \sum_{n} \mu_{\psi,n},
$$
with $\mu_{\varphi,n}$ and $\mu_{\psi,n}$ having supports in $\mathsf{P}_{n}$. 
Then any plan $\nu \in \Pi_{\sfd}(\tr(\cdot)\,\mu_{\varphi},\tr(\cdot)\,\mu_{\psi})$ 
has to send $\mu_{\varphi,n}$ to $\mu_{\psi,n}$ 
and its optimality is equivalent to optimality between each $\mu_{\varphi,n}$ and $\mu_{\psi,n}$.

Let us now consider $\nu \in \Pi_{\sfd}(\tr(\cdot)\,\mu_{\varphi},\tr(\cdot)\,\mu_{\psi})$ and $\Gamma$ a $\sfd^{p}$-cyclically monotone 
set with $\nu(\Gamma) = 1$. 
We decompose as above $\nu = \sum \nu_{n}$ with $\nu_{n} \perp \nu_{m}$ if $n \neq m$
and $\nu_{n}$ having marginals $\mu_{\varphi,n}$ and $\mu_{\psi,n}$. 
Here $\nu_{n} \perp \nu_{m}$ is  in  the sense of measure theory i.e. with disjoint supports.
Then $\sfd$ restricted to $\mathsf{P}_{n}$ is finite and therefore, by classical theory of optimal transport (see for instance \cite{Vil}), $\sfd^{p}$-cyclical monotonicity is equivalent to optimality giving 
that each $\nu_{n}$ is optimal and therefore optimality of $\nu$ follows.
\end{proof}

From the classical theory \cite[Theorem 5.10]{Vil},
the following dual formulation of the problem is valid: for any 
$\varphi,\psi \in \mathcal{S}_{n}(\mathsf{B}(\mathsf{H}))$
\begin{align*}
&~\min_{\nu \in \Pi_{\sfd}(\tr(\cdot)\,\mu_{\varphi},\tr(\cdot)\,\mu_{\psi})} \int_{\mathsf{P}_{c}\times\mathsf{P}_{c}} \sfd(P,Q)^{p}\, \pi(dPdQ)  \\
&~ = \sup_{\substack{f,g \in C_{b}(\mathsf{P}_{c}) \\ g(Q) - f(P) \leq \sfd^{p}(P,Q)} } 
\left( \int g(Q)\tr(Q)\, \mu_{\psi}(dQ)- \int f(P)\tr(P) \,\mu_{\varphi}(dP)\right)  
\end{align*}

The right hand side can actually be substituted with some special couples of functions.

\begin{definition}[$\sfd^{p}$-convex function]
A function $f : \supp(\mu_{\varphi}) \to \R \cup \{+\infty \}$
is $\sfd^{p}$-convex if it is not identically $+\infty$ and there exists 
$h : \supp(\mu_{\psi}) \to \R \cup \{\pm\infty \}$ such that for each 
$P \in \supp(\mu_{\varphi})$ 
$$
f(P) = \sup_{Q \in \supp(\mu_{\psi})} h(Q) - \sfd(P,Q)^{p}.
$$
Then its $\sfd^{p}$-transform is a function $f^{\sfd^{p}}: \supp(\mu_{\psi}) \to \R$ 
defined for each $Q \in \supp(\mu_{\psi})$ by:
$$
f^{\sfd^{p}}(Q) : = \inf_{P \in \supp(\mu_{\varphi})} f(P) + \sfd(P,Q)^{p}.
$$
\end{definition}
Theorem 5.10 of \cite{Vil} gives that the previous duality can be rewritten as follows
\begin{align*}
\min_{\nu \in \Pi_{\sfd}(\tr(\cdot)\,\mu_{\varphi},\tr(\cdot)\,\mu_{\psi})} \int_{\mathsf{P}_{c}\times\mathsf{P}_{c}} \sfd(P,Q)^{p}\, \pi(dPdQ)  
= \sup_{f \in L^{1}(\tr(\cdot)\mu_{\varphi})} \left( \int f^{\sfd^{p}}\tr(\cdot)\, \mu_{\psi}- 
\int f\tr(\cdot)\, \mu_{\varphi}\right), 
\end{align*}
and in the above supremum one might as well impose that $f$ be $\sfd^{p}$-convex.
The previous supremum is actually achieved and 
the maximum will be called a Kantorovich potential.

\begin{theorem}\label{T:Dualmain}
Given any $\varphi,\psi \in \mathcal{S}_{n}(\mathsf{B}(\mathsf{H}))$ with 
$W_{p}(\varphi,\psi) < \infty$, 
there exists $f \in L^{1}(\tr(\cdot)\mu_{\varphi})$ and $\sfd^{p}$-convex
such that 
$$
W_{p}(\varphi,\psi)^{p} = \int f^{\sfd^{p}}\tr(\cdot)\, \mu_{\psi}- 
\int f\tr(\cdot)\, \mu_{\varphi}. 
$$
In particular, $\nu \in \Pi_{\sfd}(\tr(\cdot)\,\mu_{\varphi},\tr(\cdot)\,\mu_{\psi})$ 
is $W_{p}$-optimal 
if and only if 
$$
\nu\Big( \big\{ (P,Q) \in \mathsf{P}_{c}\times \mathsf{P}_{c} \colon f^{\sfd^{p}}(Q) - f(P) = \sfd(P,Q)^{p}  \big\} \Big)
= 1.
$$
\end{theorem}

\begin{proof}
Reasoning like in the proof of Proposition \ref{P:cyclicsufficient}, 
on each connected component $\mathsf{P}_{n}$ of $\mathsf{P}_{c}$, the metric $\sfd$ is continuous 
yielding (see \cite[Theorem 5.10]{Vil})  for each $n \in \N$  the existence
of $\sfd^{p}$-convex functions $f_{n}:\supp(\mu_{\varphi,n}) \to \R$, 
meaning that it is not identically $+\infty$ and there exists 
$h : \supp(\mu_{\psi,n}) \to \R \cup \{\pm\infty \}$ such that for each 
$P \in \supp(\mu_{\varphi,n})$ 
$$
f_{n}(P) = \sup_{Q \in \supp(\mu_{\psi,n})} h_{n}(Q) - \sfd(P,Q)^{p},
$$
such that 
a transport plan between $\mu_{\varphi,n}$ and $\mu_{\psi,n}$ is optimal if and only if 
is concentrated inside the following $\sfd^{p}$-cyclically monotone set:
$$
\Big\{ (P,Q) \in \supp(\mu_{\varphi,n})\times \supp(\mu_{\psi,n})\colon f_{n}^{\sfd^{p}}(Q) - f_{n}(P) = \sfd(P,Q)^{p} \Big\},
$$
where  $f^{\sfd^{p}}_{n}$ is defined considering the infimum only among those 
$P \in \supp(\mu_{\varphi,n})$.
In particular,
$$
\int_{\supp(\mu_{\varphi,n})\times \supp(\mu_{\psi,n})} \sfd(P,Q)^{p}\, \nu(dPdQ) = 
\int f^{\sfd^{p}}_{n} \tr(\cdot)\,\mu_{\psi_{n}} - \int f_{n}\tr(\cdot)\, \mu_{\varphi_{n}}.
$$

\noindent
Define then $f (P) : = f_{n}(P)$ and $h(Q) : = h_{n}(Q)$ 
for each $P \in \supp(\mu_{\varphi,n})$ and $Q \in \supp(\mu_{\psi,n})$ and notice that 
$$
f(P) = \sup_{Q \in \supp(\mu_{\psi})} h(Q) - \sfd(P,Q)^{p},
$$
giving that $f$ is $\sfd^{p}$-convex. 
Simply noticing that  $\sfd$ takes value $+\infty$ if $P$ and $Q$ does not belong to the same connected component of $\mathsf{P}_{c}$, it follows that for $Q \in \supp(\mu_{\psi,n})$
satisfies $f^{\sfd^{p}}(Q) = f_{n}^{\sfd^{p}}(Q)$, where  $f^{\sfd^{p}}$ 
is its $\sfd^{p}$-transform, given by 
$$
f^{\sfd^{p}}(Q) : = \inf_{P \in \supp(\mu_{\varphi})} f(P) + \sfd(P,Q)^{p}.
$$
Hence 
$W_{p}(\varphi,\psi)^{p} = \int f^{\sfd^{p}}\tr(\cdot)\, \mu_{\psi}- 
\int f\tr(\cdot)\, \mu_{\varphi}$, and the second claim follows straightforwardly.
\end{proof}

We now focus on representing Kantorovich potentials.

\begin{lemma}\label{L:represent}
For any $f \in L^{1}(\tr(\cdot)\,\mu_{\varphi})$, there exists an unbounded linear and densely defined operator 
$C$  such that $C\rho_{\varphi} \in \mathcal{L}^{1}(\mathsf{H})$ (the composition extends from the domain to a bounded operator)
and $$\tr(P)f(P) = \tr(CP), \quad P \in \supp(\mu_{\varphi}).$$ 
\end{lemma}
\begin{proof}
Let $\sum_i \lambda_i P_{V_i}$ be the spectral decomposition of $\rho_{\varphi}$ with strictly decreasing eigenvalues. Then $f$ defines a Borel function on the spectrum of $\rho_{\varphi}$ with $f(\lambda_i):=f(P_i)$ and possibly $f(0)=0$.
We simply define $C:=f(\rho_{\varphi})$ by the functional calculus. In particular 
$$\operatorname{Dom}(C)=\Big{\{}x \in \mathsf{H}: \, \sum_i |f(P_{V_i})|^2 \|P_{V_i}x\|^2< \infty\Big{\}}$$
is of course dense. 
The rest is straightforward noticing that  the condition
$f \in L^{1}(\tr(\cdot)\mu_{\varphi})$ 
implies that the sequence $(\lambda_i f(P_i))_i$ is bounded.  
\end{proof}

\begin{corollary}\label{C:duality}
Given any $\varphi,\psi \in \mathcal{S}_{n}(\mathsf{B}(\mathsf{H}))$ with 
$W_{p}(\varphi,\psi) < \infty$, 
there exist $C$ and $C^{\sfd^{p}}$  
unbounded linear and densely defined operators on  $\mathsf{H}$
such that the following points are verified. 
\begin{itemize}
\item[1.] The $W_{p}$-cost verifies 
$W_{p}(\varphi,\psi) = \tr(C^{\sfd^{p}}\rho_{\psi}) - \tr(C\rho_{\varphi})$.
\item[2.] Any $\nu \in \Pi_{\sfd}(\mu_{\varphi},\mu_{\psi})$ is $W_{p}$-optimal if and only if 
$$
\nu \left( 
\Big\{ (P,Q) \in \supp(\mu_{\varphi})\times \supp(\mu_{\psi})\colon 
\tr(C^{\sfd^{p}}Q) - \tr(CP) = \frac{\sfd(P,Q)^{p}}{\tr(P)} \Big\}
\right)=1
;
$$
with $C, C^{\sfd^{p}}$ are such that
$$
\tr(C^{\sfd^{p}}Q) - \tr(CP) \leq  \frac{\sfd(P,Q)^{p}}{\tr(P)} , \qquad \forall \, (P,Q) \in 
\supp(\mu_{\varphi})\times \supp(\mu_{\psi}), \ \tr(P)= \tr(Q).
$$
\end{itemize}
\end{corollary}

\begin{proof}
To prove the first point we use  Theorem \ref{T:Dualmain} to deduce the existence of a solution $f \in L^{1}(\tr(\cdot)\mu_{\varphi})$ 
of the dual problem with 
$$
W_{p}(\mu_{\varphi},\mu_{\psi}) 
= \int f^{c_{p}}\tr(\cdot)\, \mu_{\psi} - \int f \tr(\cdot)\, \mu_{\varphi}. 
$$
Then apply Lemma \ref{L:represent} to such $f$ to obtain $C$  
such that 
$f(P)\tr(P) = \tr(CP)$ for all $P \in \supp(\mu_{\varphi})$, implying 
$$
\int f\tr(\cdot)\, \mu_{\varphi} = \tr(C\rho_{\varphi}).
$$
Denoting with $C^{\sfd^{p}}$ any linear map  representing $f^{\sfd^{p}}$, the first point follows.
The second point is then a reformulation of the second point of  Theorem \ref{T:Dualmain}.
\end{proof}

\medskip
\subsection{Wasserstein geodesics}\label{S:geodesicW}
In this section and in the following one we will study how to match 
two other possible approaches in defining a Wasserstein type distance over normal states 
with the one we introduced in Section \ref{S:Wasserstein}. 

The geodesic structure of $\mathsf{P}_{c}$ 
will permit to investigate the geodesic structure of $\mathcal{S}_{n}(\mathsf{B}(\mathsf{H}))$.
We begin by recalling the classical definition of geodesic adapted to the setting of
$\mathcal{S}_{n}(\mathsf{B}(\mathsf{H}))$.

\begin{definition}\label{D:W-geodesic}
Given $\varphi,\psi \in \mathcal{S}_{n}(\mathsf{B}(\mathsf{H}))$, a curve
$$
[0,1] \ni t \mapsto \phi_{t} \in \mathcal{S}_{n}(\mathsf{B}(\mathsf{H})), \qquad 
\phi_{0} = \varphi, \ \phi_{1} = \psi,
$$
is a $\mathcal{C}_{p}$-geodesic (resp. a $W_{p}$-geodesic) if
$\mathcal{C}_{p}(\phi_{t},\phi_{s}) = |t-s| \mathcal{C}_{p}(\varphi,\psi)$ 
(resp. $W_{p}(\phi_{t},\phi_{s}) = |t-s| W_{p}(\varphi,\psi)$), for any $s,t \in [0,1]$.
\end{definition}

We start looking
for geodesic convexity of suitable subsets of $\mathcal{P}(\mathsf{P}_{c})$.

\begin{proposition}\label{P:D_{1}convex}
The set $\mathcal{D}_{1}(\mathsf{P}_{c})$ of discrete, non-negative measures having integral of the trace equal to $1$ as defined in \eqref{E:discretetr} is 
weakly convex with respect to $W_{p}^{\mathsf{P}_{c}}$ in the following sense: for any 
$\mu_{0},\mu_{1} \in \mathcal{D}_{1}(\mathsf{P}_{c})$ 
such that $W_{p}^{\mathsf{P}_{c}}\big{(}\tr(\cdot)\mu_{0},\tr(\cdot)\mu_{1}\big{)} < \infty$, there exists a curve 
$(\mu_{t})_{t\in [0,1]} \subset \mathcal{D}_{1}(\mathsf{P}_{c})$ 
with initial point $\mu_{0}$ and final point $\mu_{1}$ such that 
$ t\longmapsto \tr(\cdot)\mu_{t}$ is a $W_{p}^{\mathsf{P}_{c}}$-geodesic.
\end{proposition}

\begin{proof}
Given $\mu_{0},\mu_{1} \in \mathcal{D}_{1}(\mathsf{P}_{c})$ such that 
$W_{p}^{\mathsf{P}_{c}}\big{(}\tr(\cdot)\mu_{0},\tr(\cdot)\mu_{1}\big{)} < \infty$, Theorem \ref{T:existence} ensures the existence 
of an optimal transport plan $\nu \in \Pi_{\sfd}(\mu_{0}, \mu_{1})$.
If $\mu_{0} = \sum_{i}\alpha_{i} \delta_{P_{V_{i}}}$
and $\mu_{1} = \sum_{i}\beta_{i} \delta_{P_{Z_{i}}}$,
there exist non-negative coefficients $\theta_{i,j}$ such that 
$$
\nu = \sum_{i,j} \theta_{i,j}\delta_{P_{V_{i}}}\otimes\delta_{P_{Z_{j}}}, \qquad 
\sum_{j}\theta_{i,j} = \alpha_{i}\tr(P_{V_{i}}), \quad 
\sum_{i}\theta_{i,j} = \beta_{j}\tr(P_{Z_{j}}).
$$
Since $\nu$ is admissible, whenever $\theta_{i,j}> 0$ it follows that 
$\tr(P_{V_{i}})=\tr(P_{Z_{j}})$ hence we can 
consider $\gamma_{i,j}$ any geodesic of $(\mathsf{P}_{c},\sfd)$ connecting $V_{i}$ to $Z_{j}$. Its
existence is assured by the fact that $V_{i}$ and $Z_{j}$ belong
same connected component of $(\mathsf{P}_{c},\sfd)$. In particular $\tr(\gamma_{i,j}(t))$ 
is constant for each $t \in [0,1]$ and depends only on $i$.

\noindent Now define the following non-negative measure over $\Geo(\mathsf{P}_{c})$ %
$$
\boldsymbol{\gamma} : = \sum_{i,j} \frac{\theta_{i,j}}{\tr(P_{V_{i}})}\delta_{\gamma_{i,j}} 
$$
and, denoting by $\ee_{t} : \Geo(\mathsf{P}_{c}) \to \mathsf{P}_{c}$ 
the evaluation map at time $t$ we have a
curve of measures $[0,1]\ni t \mapsto \mu_{t} : = (\ee_{t})_{\sharp}(\boldsymbol{\gamma})$.
First notice that $\mu_{t} \in \mathcal{D}_{1}(\mathsf{P}_{c})$:
indeed $\boldsymbol{\gamma}$ is a discrete measure therefore the same is valid for $\mu_{t}$ and 
\begin{align*}
\int_{\mathsf{P}_{c}}\tr(P)\, \mu_{t}(dP) 
= \int_{\mathsf{P}_{c}}\tr(P)\, (\ee_{t})_{\sharp}(\boldsymbol{\gamma})(dP) 
&~ = \sum_{i,j}\theta_{i,j} = 1.
\end{align*}
Hence $\mu_{t}\in \mathcal{D}_{1}(\mathsf{P}_{c})$ and finally
\begin{align*}
W_{p}^{\mathsf{P}_{c}}(\tr(\cdot)\mu_{s},\tr(\cdot) \mu_{t})^{p} 
&~ \leq 
\int \sfd(P,Q)^{p} \, ((\ee(s),\ee(t))_{\sharp}(\sum_{i,j}\theta_{i,j}\delta_{\gamma_{i,j}}))(dPdQ)  \\
&~ = 
|s-t|^{p} 
\int \sfd(P,Q)^{p} \, ((\ee(0),\ee(1))_{\sharp}(\sum_{i,j}\theta_{i,j}\delta_{\gamma_{i,j}}))(dPdQ)  \\
&~ = |s-t|^{p}  \int \sfd(P,Q)^{p} \, \nu(dPdQ)  \\
&~= |s-t|^{p} 
W_{p}^{\mathsf{P}_{c}}(\tr(\cdot)\mu_{0},\tr(\cdot) \mu_{1})^{p}. 
\end{align*}
This proves the claim.
\end{proof}

To obtain a Wasserstein geodesic between normal states,
Proposition \ref{P:D_{1}convex} must be reinforced with 
the additional assumption that 
 $\mu_{t} \in \mathcal{D}_{1}^{\perp}(\mathsf{P}_{c})$.
The condition $\mu_{t} \in \mathcal{D}_{1}^{\perp}(\mathsf{P}_{c})$ is actually quite demanding 
and has  the strong and rigid consequences on the two measures 
$\mu_{\varphi},\mu_{\psi}$
it is linking.

Recall the definition \eqref{E:discretetrort} of $\mathcal{D}_{1}^{\perp}(\mathsf{P}_{c})$ consisting of discrete measures supported on orthogonal families of projections and integrating the trace to one.

\begin{proposition}\label{P:consequences}
Given $\mu_{0},\mu_{1} \in \mathcal{D}_{1}^{\perp}(\mathsf{P}_{c})$ such that 
with $W_{p}^{\mathsf{P}_{c}}(\mu_{0},\mu_{1}) < \infty$. 
Let $\mu_{t}$ be any $W_{p}^{\mathsf{P}_{c}}$-geodesic provided from Proposition \ref{P:D_{1}convex} 
and assume $\mu_{t} \in \mathcal{D}_{1}^{\perp}(\mathsf{P}_{c})$  for all $t \in [0,1]$. 

Then there exists a bijective map $T: \supp(\mu_{0}) \to \supp(\mu_{1})$ 
such that $(Id,T)_{\sharp}\mu_{0} \in \Pi_{\sfd}(\mu_{0}, \mu_{1})$ 
is an optimal plan. In particular, if  
$\mu_{0} = \sum_{i}\alpha_{i} P_{0,i}$, $\mu_{1} = \sum_{i}\beta_{i} P_{1,i}$
with $\alpha_{i} > \alpha_{i+1}$ and $\beta_{i} > \beta_{i+1}$, 
then
$$
\tr (P_{0,i}) = \tr (P_{1,i}), \quad \alpha_{i} = \beta_{i},
$$
and $T(P_{0,i}) = P_{1,i}$.
\end{proposition}

\begin{proof}
From the classical theory of optimal transport applied to each connected component of $\mathsf{P}_{c}$, $\tr(\cdot)\mu_{t} = (\ee_{t})_{\sharp}(\boldsymbol{\gamma})$ with 
$\boldsymbol{\gamma} \in \mathcal{P}(\Geo(\mathsf{P}_{c}))$.
Hence $(\ee_{0},\ee_{1})_{\sharp}\boldsymbol{\gamma} \in \Pi_{\sfd}(\tr(\cdot)\mu_{\varphi},\tr(\cdot)\mu_{\psi})$.  
Posing $\nu = (\ee_{0},\ee_{1})_{\sharp}\boldsymbol{\gamma}$, necessarily 
$$
\nu = \sum_{i,j} \theta_{i,j} \delta_{P_{V_{i}}}\otimes\delta_{P_{Z_{j}}}, \qquad 
\sum_{i,j} \theta_{i,j} = 1, \quad \theta_{i,j} \geq 0.
$$
Assume now by contradiction there exist $i_{1} \neq i_{2}$ and $j \in \N$ 
such that both $\theta_{i_{1},j}, \theta_{i_{2},j} > 0$. 
Then there exists $\gamma_{i_{1},j}, \gamma_{i_{2},j} \in \Geo(\mathsf{P}_{c})$ such that 
$$
\gamma_{i_{1},j}(t),\gamma_{i_{2},j}(t) \in \supp(\mu_{t}),
\quad \gamma_{i_{1},j}(0) = P_{V_{i_{1}}},\gamma_{i_{2},j}(0) = P_{V_{i_{2}}}, 
\quad \gamma_{i_{1},j}(1) = P_{Z_{j}},\gamma_{i_{2},j}(1) = P_{Z_{j}},
$$
with $P_{V_{i_{1}}} \perp P_{V_{i_{2}}}$. Then $\mu_{t} \in \mathcal{D}_{1}^{\perp}(\mathsf{P}_{c})$
 implies that either 
$\gamma_{i_{1},j}(t)= \gamma_{i_{2},j}(t)$ or $\tr(\gamma_{i_{1},j}(t)\gamma_{i_{2},j}(t)) = 0$ 
with the former verified at $t = 0$ and the latter at $t = 1$;
continuity of $t \mapsto \gamma_{i_{1},j}(t),\gamma_{i_{2},j}(t)$ gives a contradiction.

The argument can be reverted and implies that for each $i \in \N$ there is only one $j \in \N$
such that $\theta_{i,j} > 0$ and for each $j \in \N$ there is only one  $i \in \N$ such that 
$\theta_{i,j} > 0$: this is equivalent to the existence of a bijective 
map $T : \supp(\mu_{\varphi}) \to \supp(\mu_{\psi})$ such that 
$$
\nu = (\operatorname{Id} \times T)_{\sharp} \mu_{0}, 
$$
proving the first part of the claim. The remaining claims are straightforward 
consequences.
\end{proof}

\begin{corollary}\label{C:geo-conv}
Fix $p\geq 1$. 
Given $\varphi,\psi \in \mathcal{S}_{n}(\mathsf{B}(\mathsf{H}))$
with $W_{p}(\varphi,\psi) < \infty$,
consider $\mu_{\varphi},\mu_{\psi}$ and any $\mu_{t}$ from Proposition \ref{P:D_{1}convex}. 
If $\mu_{t} \in \mathcal{D}_{1}^{\perp}(\mathsf{P}_{c})$, then posing
$$
\rho_{t} : = \Psi(\mu_{t}) \in \mathcal{C}(\mathsf{H}), 
$$
the curve of normal state $[0,1] \ni t\mapsto \varphi_{\rho_{t}}$ is a $W_{p}$-geodesic.
\end{corollary}

\begin{proof}
To fix notation, $\mu_{t}$ from Proposition \ref{P:D_{1}convex}
can be then written as $\mu_{t} = \sum_{i} \alpha_{i}\delta_{P_{i}(t)}$.
Then using the notations of Section \ref{S:spectral measure}, 
$\Psi(\mu_{t}) = \sum_{i} \alpha_{i}P_{i}(t)$ is a 
well-defined element of $\mathcal{C}(\mathsf{H})$. 
From Proposition \ref{P:consequences} it follows that
$$
\Phi(\varphi_{\rho_{t}}) = \Phi(\Psi(\mu_{t})) = \mu_{t}.
$$
Notice indeed that $\mu_{t}$ is an element of 
$\mathcal{D}_{1}^{\perp}(\mathsf{P}_{c})$ giving different weights 
on each element of its support.
Hence, by definition of $W_{p}$ (recall \eqref{E:W_{p}})
\begin{align*}
W_{p}(\varphi_{\rho_{s}},\varphi_{\rho_{t}}) 
&~ = W_{p}^{\mathsf{P}_{c}}(\tr(\cdot) \Phi(\varphi_{\rho_{s}}), \tr(\cdot) \Phi(\varphi_{\rho_{t}})) \\
&~ = W_{p}^{\mathsf{P}_{c}}(\tr(\cdot) \mu_{s}, \tr(\cdot) \mu_{t}) \\
&~ = |t-s|W_{p}^{\mathsf{P}_{c}}(\tr(\cdot) \mu_{0}, \tr(\cdot) \mu_{1}) \\
&~ = |t-s|W_{p}(\varphi_{\rho_{0}},\varphi_{\rho_{1}}),
\end{align*}
proving the claim.
\end{proof}

\begin{remark}\label{R:curves}
If the condition $\mu_{t} \in \mathcal{D}_{1}^{\perp}(\mathsf{P}_{c})$ is not known, 
then one can anyway define a curve of normal states because
$\Psi(\mu_{t}) = \sum_{i} \alpha_{i}P_{i}(t) = : \rho_t$ is a 
well-defined element of $\mathcal{C}(\mathsf{H})$ implying that 
$\varphi_{\rho_{t}} \in \mathcal{S}_{n}(\mathsf{B(H)})$ (see Lemma \ref{L:normalsum}).
However, the spectral decomposition of ${\rho}_t$ will not be given by $\sum_{i} \alpha_{i}P_{i}(t)$
and
$$
\Phi(\varphi_{\rho_{t}}) = \Phi(\Psi(\mu_{t})) \neq \mu_{t},
$$
and nothing can be deduced on $W_{p}(\varphi_{\rho_{t}},\varphi_{\rho_{s}})$. 
\end{remark}

\begin{remark}\label{R:identification}
In the proof of Proposition \ref{P:consequences} it was not directly used the fact that 
 $\tr(\cdot)\mu_{t}$ is a $W_{p}^{\mathsf{P}_{c}}$-geodesic, rather that  
there exists $\boldsymbol{\gamma} \in \mathcal{P}(\Geo(\mathsf{P}_{c}))$ 
such that 
$$
\tr(\cdot)\mu_{t} = (\ee_{t})_{\sharp}\boldsymbol{\gamma}, 
\quad \mu_{t} \in\mathcal{D}_{1}^{\perp}(\mathsf{P}_{c}).
$$
This implies indeed that $\boldsymbol{\gamma}$ has to be a discrete measure as well 
and $t \mapsto \tr(\cdot)\mu_{t}$ is a $W_{p}^{\mathsf{P}_{c}}$-continuous,
this two facts being enough to close the argument.
\end{remark}

Proposition \ref{P:consequences} admits a partial converse.
\begin{proposition}\label{P:consequencesinvert}
Let $\varphi,\psi \in \mathcal{S}_{n}(\mathsf{B}(\mathsf{H}))$ be given states
with $W_{p}(\varphi,\psi) < \infty$. 
If there exists a bijective map $T: \supp(\mu_{\varphi}) \to \supp(\mu_{\psi})$ 
such that $(\operatorname{Id}\times T)_{\sharp}\mu_{\varphi} \in \Pi_{\sfd}(\mu_{\varphi}, \mu_{\psi})$ 
then $\varphi$ and $\psi$ are in the same unitary orbit. There is a unitary $u$ with $\rho_{\psi}=u\rho_{\varphi}u^*$.
\end{proposition}

\begin{proof}
Let $\rho_\varphi= \sum_i \lambda_i P_{V_i}$ be the spectral decomposition with distinguished positive eigenvalues $\lambda_i$. The condition $(\operatorname{Id}\times T)_{\sharp}\mu_{\varphi} \in \Pi_{\sfd}(\mu_{\varphi}, \mu_{\psi})$ implies that the spectral decomposition of $\rho_{\psi}$ is: $\rho_{\psi}=\sum_i \lambda_i T(P_{V_i})$ with $P_{V_i}=R(P_{V_i})\cong R(T(P_{V_i}))=T(P_{V_i})$ (because we have a $\mathsf{d}$-transport plan). 
Identifying projections with subspaces we think $T$ defined on the collection of the eigenspaces of $\rho_{\varphi}$.
Every eigenspace is finite dimensional and this implies $V_i^{\bot} \cong T(V_i)^{\bot}$. 
Since these are mutually orthogonal 
we have $R(\rho_{\varphi})\cong R(\rho_{\psi})$ and $\operatorname{Ker}(\rho_{\varphi}) \cong \operatorname{Ker}(\rho_{\psi})$. 
By \cite[Proposition 7]{CIJM} we find an invertible $u \in \mathsf{B}(\mathsf{H})$ such that $\rho_{\psi}=u\rho_{\varphi}u^*$. In our case this $u$ is unitary.
Concretely let $\{e_i,e_j'\}$ an orthonormal system adapted to $\mathsf{H}=R(\rho_{\varphi})\oplus N(\rho_{\varphi})$ and $\{f_i,f_j'\}$ a corresponding one for $\rho_{\psi}$; then
$$u=\sum_i |f_i\rangle \langle e_i |+ \sum_j |f_j'\rangle \langle e_j'|.$$
We can do better finding a $u$ such that $u P_{V_i}u^*=T(P_{V_i})$ for every $i$. It suffices to choose the basis $\{e_i\}$ adapted to the spectral decomposition $\rho_\varphi= \sum_i \lambda_i P_{V_i}$. 
\end{proof}

\medskip

\section{Tensor product interpretation: a generalization}
\label{S:tensor-Wasserstein}

As specified in Section \ref{Ss:marginals},
the tensor product Hilbert space $\mathsf{H} \otimes \mathsf{H}$
corresponds, in quantum mechanics, to a composite system and a natural 
way to match two normal states $\varphi,\psi$ of $\mathsf{B}(\mathsf{H})$ 
would be via a normal 
state $\Xi \in \mathcal{S}_{n}(\mathsf{B}(\mathsf{H}\otimes \mathsf{H}))$ 
satisfying the partial trace conditions
$J^{1}_{\flat}\Xi = \varphi$ and $J^{2}_{\flat}\Xi = \psi$.

To fix notations we will use 
\begin{equation}\label{E:tracenormalstate}
\mathfrak{J}(\varphi, \psi) := 
\Big{\{} 
\Xi \in \mathcal{S}_{n}(\mathsf{B}(\mathsf{H}\otimes \mathsf{H}))
\colon J^{1}_{\flat}\Xi = \varphi, \ J^{2}_{\flat}\Xi = \psi
\Big{\}}.
\end{equation}

In this section we will reconcile this approach with the one we 
presented in Sections \ref{S:spectral measure} and \ref{S:Wasserstein}
 based on transport plans 
between spectral-projections measures. \\

We begin with some preliminaries.
We will follow \cite{Herbut}, 
and the appendix \ref{antilinearappendix} for basics on antilinear operators. 
Let $AHS(\mathsf{H})$ be the space of the antilinear 
Hilbert--Schmidt operators acting on $\mathsf{H}$. 
Firstly an antilinear operator is an additive operator 
$T:\mathsf{H} \to \mathsf{H}$ such that $T(\lambda x)=\overline{\lambda} x$. 
To define the Hilbert--Schmidt ones we begin with the antilinear rank-one operators. 
They are in the form $x \mapsto T_{\xi,\eta}(x):=\langle x, \xi \rangle \eta$ 
for fixed vectors $\xi,\eta \in \mathsf{H}$. 
On such operators we define the Hilbertian product (conjugate-linear in the first entry)
$\langle A,B \rangle:=\operatorname{tr}( A^*B)$ and we complete 
the linear span of all the operators in the form $T_{\xi,\eta}$ 
with respect to this Hilbert structure. If we compute 
$$
\langle T_{\xi,\eta}, T_{x,y}\rangle
=\operatorname{tr}(T_{\eta,\xi}T_{x,y})
=\operatorname{tr}\Big{(}\zeta \mapsto 
\langle x,\zeta \rangle \langle y,\eta \rangle \xi\Big{)}
=\langle \xi,x \rangle \langle y,\eta \rangle.
$$
On the right we have the inner product defined on $\mathsf{H}\otimes \mathsf{H}$ 
i.e. $\langle \xi \otimes \eta , x\otimes y\rangle = \langle \xi, x \rangle \langle \eta,y\rangle$
indeed we have a $\C$-linear isomorphism
\begin{equation}\label{isohilbertschmidt}
\Theta:\mathsf{H}\otimes \mathsf{H} \longrightarrow AHS(\mathsf{H})
\end{equation}
uniquely determined by linearity and continuity on simple tensors by
$\Theta_{\xi \otimes \eta} := T_{\xi, \eta}\in AHS(\mathsf{H})$.
Some basic identities are immediate to prove
\begin{equation}\label{thetaproperties}
\Theta_{\xi\otimes \eta}^*=\Theta_{\eta\otimes\xi}, \quad |\Theta_{\zeta}|^2=\operatorname{Tr}_2 |\zeta \rangle\langle \zeta | \quad \textrm{and} \quad 	|\Theta_{\zeta}^*|^2=\operatorname{Tr}_1 |\zeta \rangle\langle \zeta |,
\end{equation}
for $\zeta \in \mathsf{H}\otimes \mathsf{H}$.
\medskip
\noindent Let now
 $W:\mathsf{H}\to \mathsf{H}$ be a linear  (antilinear) partial isometry with initial space $R(W^*W)$ and final space $R(WW^*)$, 
then $W| :R(W^*W)\to R(WW^*)$ is a unitary (antiunitary) isomorphism. Let $\mathsf{P}_c(R(W^*W))\subset \mathsf{P}_c$ and  $\mathsf{P}_c(R(W^*W)\subset \mathsf{P}_c$ be the corresponding grassmannians
(recall $\mathsf{P}_c=\mathsf{P}_c(\mathsf{H})$).
This means that we are identifying
$$\mathsf{P}_c(R(W^*W)) \cong \big{\{}   P\in \mathsf{P}_c : P \leq W^*W\big{\}}$$
by taking ortogonal complements. The corresponding identification is understood for $WW^*$.
The adjoint action induces a diffeomorphism
 $$\widetilde{W}:\mathsf{P}_c(W^*W) \longrightarrow \mathsf{P}_c(WW^*), \quad \widetilde{W}(P): =\operatorname{Ad}_{W|}P=  (W|)P(W|)^*$$
 for  $ \ P\in \mathsf{P}_{c} \colon P\leq  R(W^*W).
$
We define
\begin{equation}\nonumber
\mathcal{G}(\mathsf{H}):=
\left\{
\begin{array}{l}
(V,\phi,W)\colon V,W \subset \mathsf{H}\, \mbox{closed subspaces, }\phi:\mathsf{P}_c(V) \rightarrow \mathsf{P}_c(W) \mbox{ smooth map} \\
\mbox{preserving the connected components}
\end{array}
\right\}
\end{equation}

\subsection{Pure States as transport maps}\label{Ss:puretransport}
Let us consider a pure state 
$\omega_{\zeta} \in \mathcal{PS}_n(\mathsf{B}(\mathsf{H} \otimes \mathsf{H})) 
\subset \mathcal{S}_{n}(\mathsf{B}(\mathsf{H} \otimes \mathsf{H}))$ 
represented by a  vector $\zeta \in \mathsf{H}\otimes \mathsf{H}$ 
with $\|\zeta\|=1$ and reduced density matrices
$$
\rho_1=\operatorname{Tr}_2 |\zeta \rangle \langle \zeta |, \quad \textrm{and}\quad \rho_2=\operatorname{Tr}_1 |\zeta \rangle \langle \zeta|.
$$ 
We will associate to $\omega_{\zeta}$ a unique family of transport plans 
from the spectral-projection measures of $\varphi_{\rho_{1}}$ to the one of 
$\varphi_{\rho_{2}}$.

We write the polar decomposition (see the appendix \ref{appendix polar}) of the antilinear operator $\Theta_{\zeta}$ associated via \eqref{isohilbertschmidt} to $\zeta$. 
Thus $\Theta_{\zeta}=U_{\zeta}|\Theta_{\zeta}|=|\Theta^{*}_{\zeta}|U_{\zeta}$ and 
$|\Theta_{\zeta}^{*}|=U_{\zeta} |\Theta_{\zeta}|U_{\zeta}^{*}.$ 
By \eqref{thetaproperties} 
we see that $|\Theta_{\zeta}|=\rho_1^{1/2}$ and 
$|\Theta_{\zeta}^{*}|=\rho_2^{1/2}$. It follows 
\begin{equation}\label{conjugationr}
\Theta_{\zeta}=U_{\zeta}\,\rho_{1}^{1/2}, \quad   \Theta_{\zeta}=\rho_2^{1/2}\,U_{\zeta} \quad \textrm{and} \quad \rho_2=U_{\zeta}\, \rho_1\, U_{\zeta}^{*}.
\end{equation}
The antilinear partial isometry 
$U_{\zeta}:\mathsf{H}\longrightarrow \mathsf{H}$
is called {\em{correlation operator}} and restricts to an antiunitary  
isomorphism $\xymatrix{\overline{R}(\rho_1) \ar[r]^{\cong}&\overline{R}(\rho_2)}$. 
The correlation operator is uniquely specified if we add one of  the following equivalent conditions
$$N(U_{\zeta})=N(\rho_1)
\quad \textrm{and} \quad  U_{\zeta}^*U_{\zeta}\mathsf{H}=N(\rho_1)^{\bot}$$
that we will always consider being satisfied. \\

\noindent We have associated to $\zeta \in \mathsf{H}\otimes \mathsf{H}$ its marginals and an antilinear partial isometry intertwining them. 
In the following we won't use the map $\Upsilon$ in the next proposition, rather some kind of its measure theory version.

\begin{proposition}\label{Upsilon}
	The following map is well defined
	$$\Upsilon:\mathcal{P}\mathcal{S}_n(\mathsf{H}\otimes \mathsf{H}) \longrightarrow \mathcal{G}(\mathsf{H}), \quad \omega_{\zeta}\longmapsto \Big{(}R(U_{\zeta}^*U_{\zeta}), \widetilde{U_{\zeta}}, R(U_{\zeta}U_{\zeta}^*)\Big{)}.$$
	It has the property 
	$\Upsilon(\omega_{\zeta})=\Upsilon(\omega_{\eta}) \implies U_{\zeta}=U_{\eta}$ up to a phase i.e. $U_{\zeta}=\lambda U_{\eta}$ for some $\lambda \in \mathsf{U}(1)$.
\end{proposition}
\begin{proof}
If $\zeta$ is changed into $\lambda \zeta$ for a phase $\lambda\in \mathsf{U}(1)$ then $U_{\lambda \zeta}=\lambda U_{\zeta}$ and $\widetilde{U_{\lambda \zeta}}=\widetilde{U_{\zeta}}$. The map is well defined at the states level. Assume now that $\Upsilon(\omega_{\zeta})=\Upsilon(\omega_{\eta})$; then $U_{\zeta}$ and $U_{\eta}$ have the same initial and final space. Let $x\in R(U_{\zeta}^*U_{\zeta})$ be a unit vector. Evaluating on the rank-one projections 
$$
\widetilde{U_{\zeta}}\big{(}|x\rangle \langle x|\big{)}=|U_{\zeta}x\rangle \langle U_{\zeta}x|=\widetilde{U_{\eta}}\big{(}|x\rangle \langle x|\big{)}=|U_{\eta}x\rangle \langle U_{\eta}x|.
$$
Evaluate again on the vector $U_{\zeta}x$ to obtain $U_{\zeta}x= \langle U_{\eta} |U_{\zeta}x \rangle U_{\eta}x$. By computing the norm we find $|\langle U_{\eta}x |U_{\zeta}x \rangle |=1$. Cauchy--Schwartz implies $U_{\zeta}x =f(x)U_{\eta}x$ for every unit vector $x$ (in the initial support of the involved isometries) where $f$ is a map from the unit sphere of the initial support to $\mathsf{U}(1)$.
But $f$ has to be constant by the antilinearity of our isometries.
\end{proof}
Given two sets $A,B$ we denote $\operatorname{Bij}(A,B)$ the set of Bijections from $A$ to $B$
and similarly to before we define a set of triples 
\begin{equation}
\mathcal{M}(\mathsf{H}):=
\Big{\{}
(\varphi_1,F,\varphi_2)\colon \varphi_1,\varphi_2 \in \mathcal{S}_n(\mathsf{B}(\mathsf{H})), \, \, F \in \operatorname{Bij}(\Lambda_{\varphi_1}^{\bot},\Lambda_{\varphi_2}^{\bot})
\Big{\}}.
\end{equation}
then we have a map
$$
\begin{array}{l}
\Phi_{\otimes}:\mathcal{P}\mathcal{S}_n(\mathsf{B}(\mathsf{H}\otimes \mathsf{H})) \longrightarrow   \mathcal{M}(\mathsf{H})	\\
\omega_{\zeta} \longmapsto  \Big{(}\varphi_1, (\widetilde{U_{\zeta}})_{\sharp},\varphi_2\Big{)} \\
 \\
\mbox{marginals}\, \varphi_1, \varphi_2\,  \\
\varphi_1=\varphi_{\rho_1},\, \varphi_2=\varphi_{\rho_2}  \\
\rho_1=\operatorname{Tr_2} |\zeta \rangle \langle \zeta |, \, \rho_2=\operatorname{Tr}_1 |\zeta \rangle \langle \zeta |  \\
\Theta_{\zeta}=U_{\zeta}\rho_1^{1/2}  
\end{array}
$$
Recall that $J^1_{\flat}$ is the map on $\mathcal{S}_n(\mathsf{B}(\mathsf{H}\otimes \mathsf{H}))$ that takes the first  marginal. In the following we are going to omit the identification $\mathcal{C}(\mathsf{H}) \cong \mathcal{S}_n(\mathsf{B}(\mathsf{H}))$. In particular the integration map $\Psi$ will be considered as a map $\Psi: \mathcal{D}_1^{\bot}(\mathsf{P}_c) \longrightarrow \mathcal{S}_n(\mathsf{B}(\mathsf{H}))$.
\begin{definition}
Let $F:\mathcal{P}\mathcal{S}_n(\mathsf{B}(\mathsf{H}\otimes\mathsf{H})) \longrightarrow \mathcal{D}_1^{\bot}(\mathsf{P}_c)$ be a map. 
We say that $F$ is compatible with the first marginal if 
$\Psi (F(\omega))=\rho(J^1_{\flat}(\omega))$ for any 
$\omega \in \mathcal{P}\mathcal{S}_n(\mathsf{B}(\mathsf{H}\otimes\mathsf{H})).$ 
This means that the following diagram commutes
$$
\xymatrix{\mathcal{P}\mathcal{S}_n(\mathsf{B}(\mathsf{H}\otimes\mathsf{H})) \ar[r]^-F \ar[d]_{J^1_{\flat}}& \mathcal{D}_1^{\bot}(\mathsf{P}_c) \ar[dl]^{\Psi}\\
\mathcal{S}_n(\mathsf{B}(\mathsf{H}))
 & {}}
$$
\end{definition}
Of course as a particular example we can take the map $\mathcal{F}$ obtained by the composition
\begin{equation}\label{definitionofF}
\xymatrix{\mathcal{P}\mathcal{S}_n(\mathsf{B}(\mathsf{H}\otimes \mathsf{H})) \ar[r]^{J^1_{\flat}}& \mathcal{S}_n(\mathsf{B}(\mathsf{H}))\ar[r]^{\Phi}&\mathcal{D}_1^{\bot}(\mathsf{P}_c)
.}
\end{equation}
\begin{theorem}\label{T:injection}
The map $\Phi_{\otimes}$ 	is well defined and injective. Fix any $\varphi \in \mathcal{S}_n(\mathsf{B}(\mathsf{H}))$ and a 
 measure $\mu \in \Lambda_{\varphi}^{\bot}$ representing $\varphi$; then the $\mu$-``component" of $\Phi_{\otimes}$ provides a map 
$$
\Phi_{\otimes}^{\mu}:
\Big{\{}\omega \in 
\mathcal{P}\mathcal{S}_n(\mathsf{B}(\mathsf{H}\otimes \mathsf{H})) 
\colon J^1_{\flat}\omega =\varphi \Big{\}}
 \longrightarrow \mathcal{D}(\mathsf{P}_c \times \mathsf{P}_c)
 , \quad \Phi_{\otimes}^{\mu}(\omega_{\zeta})=(\operatorname{Id} \times \widetilde{U_{\zeta}})_{\sharp}(\operatorname{tr}(\cdot)\mu).
$$
This map is valued in the set of admissible transport plans 
$$
\Phi_{\otimes}^{\mu}(\omega_{\zeta})\in \Pi_d\big{(}\operatorname{tr}(\cdot) \mu, (\widetilde{U}_{\zeta})_{\sharp}(\operatorname{tr}(\cdot) \mu)\big{)}.
$$
 In a similar way a map $F$ which is compatible with the first marginal can be combined with $\Phi_{\otimes}$ to the map
$$
\Phi_{\otimes}^{F}:\mathcal{P}\mathcal{S}_n(\mathsf{B}(\mathsf{H} \otimes \mathsf{H}))\longrightarrow \mathcal{D}_1^{\bot}(\mathsf{P}_c \times \mathsf{P}_c), \quad \Phi_{\otimes}^F(\omega):=(\operatorname{Id}\times \widetilde{U_{\zeta}})_{\sharp}(F(\omega)).
$$
We have a compatibility property expressed by the commutative diagram
$$
\xymatrix{
\mathcal{P}\mathcal{S}_n(\mathsf{B}(\mathsf{H}\otimes
\mathsf{H}))\ar[dr]^{F}\ar[d]_{J^1_{\flat}} \ar[r]^{\Phi_{\otimes}^F}& \mathcal{D}_1^{\bot}(\mathsf{P}_c \times \mathsf{P}_c) \ar[d]^{\pi_{\sharp}^1}\\
\mathcal{S}_n(\mathsf{B}(\mathsf{H})) & \mathcal{D}_1^{\bot}(\mathsf{P}_c)\ar[l]^{\Psi}
.}
$$
When $F=\mathcal{F}$ as before (eq. \eqref{definitionofF}) this becomes a compatibility with $\Phi$ as the diagram
$$
 \xymatrix{
 \mathcal{P}\mathcal{S}_n(\mathsf{B}(\mathsf{H}\otimes
\mathsf{H}))\ar[d]_{J^1_{\flat}} \ar[r]^{\Phi_{\otimes}^F}& \mathcal{D}_1^{\bot}(\mathsf{P}_c \times \mathsf{P}_c) \ar[d]^{\pi_{\sharp}^1}\\
\mathcal{S}_n(\mathsf{B}(\mathsf{H})) \ar[r]^{\Phi}& \mathcal{D}_1^{\bot}(\mathsf{P}_c)
}
$$
commutes.
\end{theorem}

\begin{proof}
Among all the spectral measures associated to $\varphi_1$ there are those with all the projections of rank-one. Then starting from the assumption $\Phi_{\otimes} (\omega_{\zeta})=\Phi_{\otimes}(\omega_{\eta})$ and testing the equality $\widetilde{U_{\zeta}}=\widetilde{U_{\eta}}$ for an arbitrary choice of one of these rank-one presentations of spectral measures we get the existence of a ortonormal set of vectors $(e_i)_i$ spanning the initial domain of $U_{\zeta}$ and $U_{\eta}$ where $|U_{\zeta}e_i\rangle \langle U_{\zeta}e_i|=|U_{\eta}e_i \rangle \langle U_{\eta}e_i |$ for every $i$. As in the proof of Proposition \ref{Upsilon} $U_{\zeta}=\lambda U_{\eta}$ for a phase $\lambda$. The marginals now coincide and this means $\Theta_{\zeta}=\lambda \Theta_{\eta}$ which implies that the corresponding states are equal. The rest of the proof is straightforward. In particular notice we get admissible transport plans because at any instance the discrete measures are in the form $\sum_i \lambda_i P_{V_i}$ for an orthogonal family of finite rank projections and the transport maps are induced by antilinear partial isometries $U_{\zeta}$ with $P_{V_i} \leq U_{\zeta}^*U_{\zeta}$. This means that for every $i$ the points $P_{V_i}$ and $\widetilde{U_{\zeta}}(P_{V_i})$ belong to the same connected component.
\end{proof}
\begin{remark}
Of course the role of the marginals is symmetric. 
The flip automorphism 
$\mathsf{H}\otimes 	\mathsf{H} 
\longrightarrow \mathsf{H}\otimes 	\mathsf{H}$ 
that on simple tensors is defined by 
$\underline{x \otimes y}=y\otimes x$ 
induces an homeomorphism of the space of the states 
that switches the marginals.
One checks immediately 
$\Phi_{\otimes}(\underline{\omega_{{\zeta}}})=\Big{(}\varphi_2, (\widetilde{U_{\zeta}}^{-1})_{\sharp},\varphi_1\Big{)} .$
\end{remark}

\begin{remark}[Wasserstein Cost of pure states]\label{R:costpure}

After Theorem \ref{T:injection}, 
we can define a Wasserstein cost, depending on $p$, 
for any pure normal state of $\mathsf{B}(\mathsf{H\otimes H})$. 
In particular given 
$\varphi_{1}, \varphi_{2} \in \mathcal{S}_{n}(\mathsf{B(H)})$ 
and $\omega_{\zeta} \in \mathcal{PS}_n(\mathsf{B(H\otimes H)})$, 
for each $\mu \in \Lambda_{\varphi_{1}}^{\perp}$
we have the transport plan 
$\Phi_{\otimes}^{\mu}(\omega_{\zeta})$ 
(induced by the map $\widetilde U_{\zeta}$)
between admissible representations of $\varphi_{1}$ and $\varphi_{2}$ 
whose $p$-cost will be
$$
\int_{\mathsf{P}_{c} \times\mathsf{P}_{c} } \sfd^{p}(P,Q) \, 
\Phi_{\otimes}^{\mu}(\omega_{\zeta})(dPdQ) = 
\int_{\mathsf{P}_{c}} \sfd^{p}(P,\widetilde U_{\zeta}(P)) \tr(P)\,
\mu(dP).
$$
Hence, the cost of $\omega_{\zeta}$ 
will be given by taking the lowest possible cost among all 
$\Phi_{\otimes}^{\mu}(\omega_{\zeta})$: 
\begin{equation}\label{E:Costzeta}
\mathcal{C}_{p}(\omega_{\zeta})^{p} : = 
\inf_{\mu \in \Lambda_{\varphi_{1}}^{\perp}}
\int_{\mathsf{P}_{c}} \sfd^{p}(P,\widetilde U_{\zeta}(P)) \tr(P)\,
\mu(dP).
\end{equation}
Following Proposition \ref{P:attained}, it is equivalent to restrict 
the minimisation only among those $\mu$ concentrated inside $\mathsf{P}_{1}$. 
Moreover the $\inf$ is actually attained giving that 
there exists $\mu \in \mathsf{P}_{1}$, 
a priori not unique and depending on $p \geq 1$,  such that  
$$
\mathcal{C}_{p}(\omega_{\zeta})^{p} = 
\int_{\mathsf{P}_{c}} \sfd^{p}(P,\widetilde U_{\zeta}(P))\,
\mu(dP).
$$
Notice however that by construction, it is immediate to see  
that 
$$
\mathcal{C}_{p}(J^{1}_{\flat} \omega_{\zeta},J^{2}_{\flat} \omega_{\zeta}) \leq \mathcal{C}_{p}(\omega_{\zeta}).
$$
\end{remark}

\appendix

\section{Homogeneous spaces and principal bundles}\label{Pbundles}
 \subsubsection*{Homogeneous spaces and principal bundles}
Let $G$ be a group acting (say on the right) on 
a space $M$.
We usually denote this action with
 $x \cdot g$. Sometimes also the symbol $\mathcal{R}_g(x)=x\cdot g$ will be used.
 The action is {\em{free}} whenever $ x\cdot g=x$ for some $x \in M$ implies $g=e$. 
Assume that $G$ acts on
two spaces $M$ and $N$. A map $\varphi:M \longrightarrow N$ is equivariant if 
$$\varphi(x\cdot g)=\varphi(x) \cdot g, \quad \forall x \in M, \textrm{ and }g \in G.$$
\begin{definition}
Let $G$ be a Lie group. A homogeneous space is a manifold $M$ with a transitive left action of $G$.
\end{definition}
Given a closed subgroup $B\subset G$ we can prove that  the space of the left cosets $G/B$ is a manifold. The left action of $G$ on itself commutes with the right $B$-action so that it descends to a left transitive action on $G/B$. Thus $G/B$ is a basic example of a homogeneous space. On the other hand, let $M$ be a homogeneous space and fix a point $p\in M$. The stabiliser $I_p :=\big{\{}g\in G: \,g\cdot p=p \big{\}}$ is a closed subgroup. It is easy to prove that $M$ is equivariantly diffeomorphic to $G/I_p$. Therefore every homogeneous space is in the form $G/B$ with $B\subset G$ closed.

\begin{definition}(cfr \cite{Kobayashi}).
	Let $M$ be a manifold and $G$ a Lie group. A principal bundle over $M$ with structure group $G$ consists in a manifold $E$ with a right action of $G$ such that: 
	\begin{enumerate}
	\item The action is free and $M$ is the quotient space $E/G$ with smooth canonical projection $\pi:E\longrightarrow M$.
	\item The following local triviality of $E$ is satisfied: any point $x \in M$ has a neighborhood $U$ such that $\pi^{-1}(U)$ is isomorphic to $U\times G$. In the present context {\em{isomorphic}} means that we can find a diffeomorphism $\psi: \pi^{-1}(U) \to U\times G$ in the form $\psi(u)=(\pi(u),\varphi(u))$ for a smooth map $\varphi: \pi^{-1}(U) \to G$ satisfying $\varphi(u\cdot g)=\varphi(u)g$ for every $g \in G$.
	\end{enumerate}
To synthetize this definition we say that $E\longrightarrow M$ is a principal bundle.
\end{definition}
\begin{example}\label{homspace}Every homogeneous space
$G/B$ is the base of a principal bundle.
Indeed we can prove that $G\longrightarrow G/B$ is a principal bundle with structure group $B$. In particular the local triviality follows from the existence of local smooth sections of the projection. If we consider the left translation action of $G$ on itself we also see that the projection is equivariant.   
\end{example}
Let 
$\xymatrix{E\ar[r]^{\pi}&M}$ be a $G$-principal bundle. At every point $p\in E$, the vertical space $\mathcal{V}_p:=\operatorname{N}(d\pi:T_pE \to T_{\pi}M)\subset T_pE$ is the tangent space of the fiber. Using the $G$-action it can be canonically identified with the Lie algebra $\mathfrak{g}$ of $G$ in the following way: every $X\in \mathfrak{g}$ defines the 
{\em fundamental} vector field $\widetilde{X} \in \Gamma(TE)$ (sections of the tangent bundle) with $\widetilde{X}_p:=\dfrac{d}{dt}\bigg{|}_{t=0}p \cdot (\operatorname{exp}tX).$ Fundamental vector fields are of course vertical and at every point the map $\mathfrak{g} \longrightarrow \mathcal{V}_p$ given by $X \longmapsto \widetilde{X}_p$ is an isomorphism. However in general there is no preferred choice of {\em{horizontal}} subspaces of $TE$. This is extra structure amounts to a {\em{connection}}.
\begin{definition}\label{connections}
A connection on the principal bundle $E\to M$ is a smooth distribution $p \mapsto \mathscr{H}_p \subset T_pE$ of vector subspaces called {\em{horizontal}} with the properties:	
\begin{enumerate}
\item For every $p \in E$ we have $T_pE=\mathcal{V}_p \oplus \mathscr{H}_p$.
\item Invariance: for every $g\in G$ and $p\in E$ then $d\mathcal{R}_g \mathscr{H}_{p}=\mathscr{H}_{gp}$. 
\end{enumerate}
\end{definition}
A connection on $E$ provides us with a notion of horizontal curves and horizontal liftings of curves. Moreover given any representation $G\longrightarrow \operatorname{End}(V)$ on a vector space, a classical construction going under the name of {\em{associated bundle construction}} produces a vector bundle $W\to M$ having $V$ as typical fiber and the connection on $E$ induces a covariant derivative (in the usual meaning) on $W$. In particular this gives a covariant derivative in the tangent bundle $TM$ of the base.


\section{Some basic facts in operator theory}
\subsection{Antilinearity} \label{antilinearappendix}
Recall that our Hilbert spaces have inner products complex linear in the first entry.
We follow \cite{Black} (there the inner product is linear in the first entry). 
An antilinear operator $T:\mathsf{H} \longrightarrow \mathsf{K}$ is an additive operator such that $T\lambda \xi=\overline{\lambda}T\xi$. 

Let $J:\mathsf{H}\longrightarrow \mathsf{H}$ be antilinear and isometric: $\|J\xi\|=\|\xi\|$ for every $\xi$. By polarization it follows $\langle J\xi, J\eta \rangle = \langle \eta,\xi \rangle$ for every couple of vectors. If such $J$ is invertible it is called {\em{antiunitary}}.

An antilinear and isometric $J:\mathsf{H}\longrightarrow \mathsf{H}$ is called an {\em{involution}} if $J^2=\operatorname{Id}_{\mathsf{H}}$. It follows that $J$ is an antiunitary.
Involutions always exist for every Hilbert space and are very useful: if $T:\mathsf{H}\longrightarrow \mathsf{K}$ is antilinear then $JT$ is linear and we can safely talk about bounded antilinear operators by looking at $JT$ (for just one $J$; it does not depends on the choice).

Let $T:\mathsf{H}\longrightarrow \mathsf{K}$ be antilinear bounded, then the adjoint of $T$ is the unique antilinear bounded operator $T^*: \mathsf{K} \longrightarrow \mathsf{H}$ such that  
$$\langle T^*\xi, \eta \rangle=\langle T\eta, \xi \rangle, \quad  \xi \in \mathsf{K},\, \eta \in \mathsf{H}.$$  It satisfies: $(\lambda T)^*=\lambda T^*$ as opposite to the behaviour of the adjoint for linear operators.
Using an involution on $\mathsf{H}$ we can compute
$T^*:=J(TJ)^*$ in terms of the adjoint of a linear operator.

\subsection{Polar decompositions}\label{appendix polar}

A bounded operator $T:\mathsf{H}\longrightarrow \mathsf{K}$ is a partial isometry if $T^*T$ is a projection $P$.
Therefore $P\mathsf{H}=N(T)^{\bot}$ and
 also $Q:=TT^*$ is the projection onto $R(T)$, the range of $T$. These are called respectively initial and final support of $T$.
It also follows that $T$
restricts to an isometry $N(T)^{\bot}\longrightarrow R(T)$. 
\begin{theorem}(Left polar decomposition)
Any $T\in \mathsf{\mathsf{B}}(\mathsf{H},\mathsf{K})$ (two Hilbert spaces) has the decomposition $T=UP$ for a non negative operator $P:\mathsf{H} \to \mathsf{H}$ and a partial isometry $U:\mathsf{H}\to \mathsf{K}$. This decomposition is unique if we require that $N(U)=N(P)$.
Equivalently if we require that 
the initial support $(U^*U)\mathsf{H}$ of $U$ is $N(P)^{\bot}$. In this case we have the properties: $P=|T|:=\sqrt{T^*T}$ and the decomposition reads $$T=U|T|,$$ with $$\quad U^*U=\operatorname{Proj} \overline{N(T)^{\bot}}  \quad \textrm{and }\,\,\,UU^*=\operatorname{Proj}\overline{R(T)}.$$   \end{theorem}
\begin{proof}
Let $P:=|T|=\sqrt{T^*T}$ then $N(P)=N(T)$ and $N(T)^{\bot}=\overline{R(T^*)}=\overline{R(|T|)}.$ It follows that on $R(|T|)$ is well defined an isometric map $U$ such that $U(|T|x)=Tx$. On the orthogonal, which is $N(T)$ we declare it zero. Then $U$ is defined everywhere (and remains isometric on the closure of $R(|T|)$. Notice $R(U)=\overline{R(T)}$.

\noindent Assume we have decomposition $T=UP$ with $N(U)=N(P)$. Then $T^*=PU^*$ and $T^*T=PU^*UP$ but $U^*U=\operatorname{Proj} \overline{R(P)}$ i.e. $T^*T=P^2$ which means $P=|T|$. We already know that $U$ is uniquely determined on the range $P=|T|$ and we are done. 
\end{proof}
\noindent The left polar decomposition of $T^*$ gives rise to the \emph{right  polar decomposition} 
$$T=|T^*|U$$
of $T$.  Begin with $T=U|T|$. Then $T^*=|T|U^*$ and $TT^*=U|T|^2U^*$ which we can iterate getting for every power: $(TT^*)^n=U|T|^nU^*.$ It follows by the unicity of the functional calculus that $|T^*|=U|T|U^*$ i.e. $|T^*|U=U|T|$ (because $|T|U^*U=|T|$).

Let us now consider an antilinear bounded operator $T:\mathsf{H}\longrightarrow \mathsf{K}$. Using an involution as before we can construct polar decompositions  
$$T=V|T|=|T^{*}|V,$$ 
for $|T|=\sqrt{T^{*}T}$ a linear operator while $V$ is an antilinear partial isometry with $V^{**}V=\operatorname{Proj}(N(T))^{\bot}$ and $VV^{*}=\operatorname{Proj} \overline{R(T)}$. In particular $V$ reverts the order inside the inner product: on $(N(T))^{\bot}$ we have $\langle V\xi,V\eta \rangle =\langle \eta,\xi \rangle$. We also have $$|T^{*}|=V|T|V^{*}.$$

\end{document}